\documentclass[12pt,a4paper,reqno]{amsart}
\usepackage[all]{xy} 
\usepackage{amsmath}
\usepackage{amsfonts}
\usepackage{amssymb}
\usepackage{mathrsfs}

\usepackage{tensor}
\usepackage{cancel}

\usepackage{xcolor}
\usepackage{graphicx}

\usepackage{amsmath,calligra,mathrsfs}

\newcommand{\scHom}{\mathscr{H}\text{\kern -3pt {\calligra\large om}}}
\newcommand{\scExt}{\mathscr{E}\text{\kern -3pt {\calligra\large xt}}}

\DeclareMathAlphabet{\mathbbold}{U}{bbold}{m}{n}

\DeclareFontEncoding{OT2}{}{} 
\newcommand{\textcyr}[1]{%
 {\fontencoding{OT2}\fontfamily{wncyr}\fontseries{m}\fontshape{n}
 \selectfont #1}}
\newcommand{\Sha}{{\!\be\lbe\mbox{\textcyr{Sh}}}}

\theoremstyle{plain}
\newtheorem{theorem}{Theorem}[subsection]
\newtheorem*{myassumption}{Assumption}
\newtheorem{remark}[theorem]{{\textrm{Remark}}}
\newtheorem{lemma}[theorem]{Lemma}
\newtheorem{corollary}[theorem]{Corollary}
\newtheorem{myproposition}[theorem]{Proposition}
\newtheorem{definition}[theorem]{Definition}

\def\le{\kern 0.03em}

\def\C{{\mathbb C}}

\def\F{{\mathbb F}}

\def\N{{\mathcal N}}
\def\O{{\mathcal O}}

\def\Q{{\mathbb Q}}

\def\Z{{\mathbb Z}}

\def\e{\kern 0.08em}
\def\be{\kern -.1em}
\def\lbe{\kern -.025em}

\DeclareMathOperator{\coh}{H}
\DeclareMathOperator{\Frob}{Frob}
\DeclareMathOperator{\Gal}{Gal}

 \DeclareMathOperator{\Nm}{N}

\DeclareMathOperator{\Hom}{Hom} 
 
\DeclareMathOperator{\ord}{ord} \DeclareMathOperator{\coker}{coker}

\begin{document}
\title[$p$-adic $L$-functions]{$p$-adic $L$-functions for elliptic curves over global function fields}
\author{Ki-Seng Tan}
\address{Department of Mathematics\\
National Taiwan University\\
Taipei 10764, Taiwan}
\email{tan@math.ntu.edu.tw}

\begin{abstract} We introduce a $p$-adic $L$-function $\mathscr L_{A/L}$ associated to each ordinary elliptic curve $A$ over a global function field $K$ of characteristic $p$ together with a $\mathbb{Z}_{p}^{d}$-extension $L/K$, $d=0$ allowed, unramified outside a finite set of places where $A$ has ordinary (good ordinary or multiplicative) reductions. This $\mathscr L_{A/L}$ is characterized by its interpolation of the
special values of twisted Hasse-Weil $L$-functions.
We show that it satisfies the desired functional equation, specialization formula, and restriction formula in connection with the characteristic ideal of the dual $p^\infty$-Selmer group of $A/L$. The Iwasawa main conjecture having $\mathscr{L}_{A / L}$ as the analytic side is proven in several cases. In the $d\geq 3$ case, 
the conjecture holds for $A/L$ if and only if it holds for all intermediate $\Z_p^2$-extensions $L'/K$ belonging to
a given non-empty Zariski open subset of the Grassmannian $\mathrm{Gr}(d-2,d)(\Z_p)$.

Recently, subject to a technical $\mu$-invariant hypothesis, if $A/K$ has semistable reduction everywhere,
the Iwasawa main conjecture is proven for $A$ over $L$ \cite{ttt26}.

\end{abstract}

\maketitle

\section{Introduction}\label{s:int} 
Let $A$ be an abelian variety defined over a global field $K$ and $L/K$ a $\mathbb{Z}_{p}^{d}$-extension, $d=0$ allowed, unramified outside a finite set of places where $A$ has ordinary (good ordinary or multiplicative) reductions. The dual $p^{\infty}$-Selmer group $X_L$ of $A / L$ is known to be finitely generated over the Iwasawa algebra
$\Lambda:=\Lambda_\Gamma:=
\mathbb{Z}_{p}[[\Gamma]]$, where $\Gamma:=\Gal(L / K)$
(see \cite[Proposition 1.1]{tan13} and \cite[Corollary 4.1.3]{ttt23}), so the algebraic side of the Iwasawa main conjecture, that is, 
the characteristic ideal $\mathrm{CH}_{\Lambda}(X_{L})$,  is defined.

In this article, we consider the case where $K$ is a global function field of characteristic $p$ and $A/K$ an ordinary elliptic curve. In view of the modularity of $A/K$ (or of its quadratic twist, when $A/K$ is a constant curve), it is more or less known that a good candidate of the analytic side for the Iwasawa main conjecture does exist. Our first aim is to establish the explicit definition of such a $p$-adic $L$-function $\mathscr L_{A/L}$. It turns out that  $\mathscr L_{A/L}$ is defined as an element in $\Q_p\Lambda$ and is
characterized by its interpolation of the special values of twisted Hasse-Weil $L$-functions. 
Thus, in this article the Iwasawa main conjecture for $A/L$, or $A/L/K$ to emphasis the role of $L/K$, 
means 
\begin{equation}\label{e:imc}
\mathscr L_{A/L}\in\Lambda,\;\text{and in}\;\Lambda,\; \text{the ideals}\;  \mathrm{CH}_{\Lambda}(X_{L})=(\mathscr L_{A/L}).
\end{equation}

Next, we investigate properties of the principal fractional ideal $\mathscr L_{A/L}\cdot\Lambda$,
and show that it and $\mathrm{CH}_\Lambda(X_L)$ both satisfy the same types of formulae:
the algebraic functional equation, the specialization formula and the restriction formula.
We also check in several cases that the Iwasawa main conjecture \eqref{e:imc} holds.
As an application, we show that if $A/K$ has semistable reduction everywhere, then as expected,
$\mathscr L_{A/L}\in \Lambda$, and has the same $\mu$-invariant as that of $X_L$.

 Finally, we reduce the proof of \eqref{e:imc} to the $d=2$ cases, namely, if $d\geq 3$, 
the conjecture holds for $A/L$ if and only if it holds for all $\Z_p^2$-extensions $L'/K$ belonging to a given non-empty Zariski open subset of the corresponding Grassmannian $\mathrm{Gr}(d-2,d)(\Z_p)$.

Recently, the Iwasawa main conjecture \eqref{e:imc} is proven, subject to a technical $\mu$-invariant hypothesis, as long as  $A/K$ has semistable reduction everywhere \cite{ttt26}.

Below, the main results of \S 2 $\sim$ \S 6, are respectively described  in the corresponding
subsections: \S1.1 $\sim$ \S1.5. 

\subsection{The specialization formula of $\mathrm{CH}_\Lambda(X_L)$}\label{su:chxl}
Section 2 is dedicated to modifying the existing specialization formula given in \cite[Theorem 1]{tan13}
for $\mathrm{CH}_\Lambda(X_L)$, which only treats the special case
when $L/K$ is also unramified at each non-split multiplicative place (a place at which $A/K$ has non-split multiplicative reduction).

In our setting, $L/K$ is assumed to be unramified only outside a finite set of ordinary places. The main result, Proposition \ref{p:aspf}, says that if $L'/K$ is an intermediate $\Z_p^e$-extension of $L/K$, $d>e\geq 0$, $\Gamma'=\Gal(L'/K)$, and $p^L_{L'}:\Lambda_\Gamma\longrightarrow \Lambda_{\Gamma'}$
denotes the $\Z_p$-algebra homomorphism induced by the projection $\Gamma\longrightarrow \Gamma'$,
then in $\Lambda_{\Gamma'}$
\begin{equation*}\label{e:aspe1}
\varrho_{L/L^{\prime} }\cdot p^L_{ L'}(\mathrm{CH}_{\Lambda}(X_{L}))=\vartheta_{L / L^{\prime}} \cdot \mathrm{CH}_{\Lambda^{\prime}}(X_{L^{\prime}}).
\end{equation*}
The correcting factors $\varrho_{L/L'}$ and $\vartheta_{L/L'}:=\prod_{\text{all}\;v}\vartheta_{L/L',v}$ are respectively defined in Definition \ref{d:varrho} and Definition
\ref{d:vartheta}.

The proof of Proposition \ref{p:aspf} reduces to the $d=e+1$ case, and by following the results of \cite{tan13}, it  basically relies on investigating the local cohomology groups over each non-split multiplicative place $v$:
$$\mathcal W_v^i:=\bigoplus_{x\mid v} \coh^i(L_x/L'_x, A(L_x))^\vee,\; i=1,2,$$
the direct sum over all places of $L'$ sitting over $v$, where $M^\vee$ denotes the Pontryagin dual of $M$.
At such place $v$, by Lemma \ref{l:locch}, $\mathcal W^2_v=0$ and over $\Lambda_{\Gamma'}$, the characteristic ideal of $\mathcal W_v^1$ equals $\vartheta_{L/L',v}$.

\subsection{The $p$-adic $L$-functions}\label{su:lfun}
In \S 3, we define the $p$-adic $L$-function.
For each effective divisor $D$ of $K$, let $W_D$ denote the corresponding Weil group and let $\tilde W_D$ denote its profinite completion (see \S\ref{su:notation}). Theorem \ref{t:theta} summarizes results in \cite{tan93} on the modular element proposed by Mazur \cite{maz87}, saying that for each $D$ there is associated an element
$\Theta_{D}\in \frac{1}{p^\aleph} \Z[W_D]$, which satisfies the interpolation formula and the transition formulae as $D$ varies. Here, $\aleph$ is a non-negative integer independent of $D$.

Let $S$ denote the ramification locus of $L/K$ and put 
$$\tilde W_S:=\varprojlim_{\mathrm{Supp}(D)=S} \tilde W_D,\quad 
\tilde\Lambda_S:=\Z_p[[\tilde W_S]]=\varprojlim_{\mathrm{Supp}(D)=S} \Z_p[[\tilde W_D]].$$

Following the method in \cite{mtt86}, we derive from $\Theta_D$ the modified element $\tilde{\mathscr L}_D$
that has desirable transition formulae (see Lemma \ref{l:thetacomp}) such that the inverse limit 
$$\tilde{\mathscr L}_{A,S}:=\varprojlim_{\mathrm{Supp}(D)=S}\tilde{\mathscr L}_D$$
exists as an element of $ \frac{1}{p^\aleph}\tilde\Lambda_S\subset \Q_p\tilde\Lambda_S$.

Let $p_L:\Q_p\tilde\Lambda_S\longrightarrow \Q_p\Lambda$ be the $\Q_p$-algebra homomorphism induced from the projection $\tilde W_S\longrightarrow \Gamma$, and define 
\begin{equation}\label{e:hatL}
\hat{\mathscr{L}}_{A / L}= \begin{cases}
p_{L}(\tilde{\mathscr{L}}_{A, S}), & \text { if } L \neq K; \\ 
q^{\frac{\deg(\Delta)}{12}+\kappa-1} \cdot L_{A}( 1), & \text { if } L=K.
\end{cases}
\end{equation}
Here $q$ denotes the order of the constant field $\F_q$ of $K$, $\kappa$ and $\Delta:=\Delta_{A/K}$
the genus of $K$ and the global discriminant of $A/K$, while $L_{A}(s)$ denotes the Hasse-Weil $L$-function associated to $A/K$.

 As the $p$-adic $L$-function introduced in \cite{mtt86} has an ``extra order of vanishing", so is our 
 $\hat{\mathscr L}_{A/L}$. In Proposition \ref{p:additional},  this extra order of vanishing is expressed as the divisibility of $\hat{\mathscr L}_{A/L}$ by the element $\dag_{A/L}\in\Lambda$ defined in Definition \ref{d:gimel}. 
 We define
\begin{equation}\label{e:l}
\mathscr{L}_{A / L}=\mathrm{t}_{A / L} \cdot \nabla_{A / L} \cdot \dag_{A / L}^{-1} \cdot \hat{\mathscr{L}}_{A / L}. 
\end{equation}
Here the fudge factors $\mathrm{t}_{A/L}$ and $\nabla_{A/L}$ are respectively defined in \S\ref{ss:mathrmt} and Definition \ref{d:nabla}.
If it is needed to refer to $K$ as well, we shall denote $\mathscr{L}_{A / L / K}:=\mathscr{L}_{A / L}$.

Let $I_\Gamma\subset\Lambda$ denote the augmentation ideal and let $s_L$ denote the number of split multiplicative places of $K$ ramified over $L$.
In \S\ref{ss:vanishing}, we shall show that $\hat{\mathscr L}_{A/L}\in \Q_pI_\Gamma^{s_L}$ (Lemma \ref{l:vanishing}). A Mazur-Tate-Teitelbaum type formula is given in Proposition \ref{p:mtt} under the assumption of the 
Iwasawa main conjecture for every $\Z_p$-extension of $K$ unramified outside a finite set of ordinary places.

\subsection{Basic Properties of $\mathscr L_{A/L}$}\label{su:padicL}
In \S\ref{s:padicL}, we discuss basic properties of $\mathscr L_{A/L}$, afterwards we check the Iwasawa main conjecture for several special cases.

\subsubsection{The interpolation formula}\label{ss:intpol}
Lemma \ref{l:intL} asserts that for each continuous character $\omega:\Gamma\longrightarrow \mu_{p^\infty}$, satisfying $\omega(\dag_{A / L})\not=0$,
$$\omega(\mathscr{L}_{A / L})=\mathrm{t}_{A / L} \cdot \omega(\nabla_{A / L}) \cdot \omega(\dag_{A / L})^{-1} \cdot \alpha_{D_{\omega}}^{-1} \cdot \tau_{\omega} \cdot q^{\frac{\deg(\Delta)}{12}+\kappa-1} \cdot \Xi_{S, \omega} \cdot L_{A}(\omega, 1).$$
The definitions of  
$\alpha_{D_{\omega}}$, $\tau_{\omega}$, 
$ \Xi_{S, \omega}$, and $L_{A}(\omega, 1)$ can be found
respectively in 
\S\ref{ss:elts} (for $D_\omega$, see \S\ref{su:notation}), \S\ref{ss:gauss}, 
\S\ref{ss:inttildeL}, and \S\ref{ss:thw}. The formula determines $\mathscr L_{A/L}$, because if $\mathscr F_{A/L}$ has the same interpolation as that of $\mathscr L_{A/L}$, then $\dag_{A/L}\cdot (\mathscr L_{A/L}-\mathscr F_{A/L})=0$,
but $\dag_{A/L}\not=0$.

\subsubsection{The functional equation}\label{ss:funeq}
Consider the $\Z_p$-algebra involution
$$^\sharp:\Lambda\longrightarrow \Lambda, \; \xi\mapsto \xi^\sharp,$$ 
given by $\gamma\mapsto\gamma^{-1}$, for $\gamma\in\Gamma$.
The characteristic ideal $\mathrm{CH}_{\Lambda}(X_{L})$ satisfies the algebraic functional equation (see [LLTT18, Theorem 2])
\begin{equation}\label{e:afel}
\mathrm{CH}_\Lambda(X_L)^\sharp=\mathrm{CH}_\Lambda(X_L).
\end{equation}
Extend $^\sharp$ to an involution of $\Q_p\Lambda$.
In Proposition \ref{p:fel}, we show that $\mathscr{L}_{A / L}$ satisfies the desired functional equation 
$$\mathscr{L}_{A / L}^{\sharp}\cdot\Lambda=\mathscr{L}_{A / L}\cdot\Lambda.
$$
\subsubsection{The specialization formula of $\mathscr L_{A/L}$}\label{ss:specialization}
The specialization formula of $\mathscr L_{A/L}$ is given in Proposition \ref{p:spl} saying that
if $p_{L^{\prime}}^{L}(\dag_{A / L}) \neq 0$, then as principal fractional ideal of $\Lambda'$,
\begin{equation*}\label{e:spe1}
 p_{L^{\prime}}^{L}(\mathscr{L}_{A / L})\cdot\varrho_{L / L^{\prime}} = \mathscr{L}_{A / L^{\prime}}\cdot\vartheta_{L / L^{\prime}} .
\end{equation*}

Suppose  $p^L_{L'}(\dag_{A/L})\cdot\varrho_{L/L'}\cdot \vartheta_{L/L'}\not=0$ (which is the case when $L'=K_\infty^{(p)}$, the unramified $\Z_p$-extension of $K$, or $L=L'K_\infty^{(p)}$, see Corollary \ref{c:1}) and
$\mathscr L_{A/L}\in\Lambda_\Gamma$.  By the specialization formula, one proves directly that the Iwasawa main conjecture for $A/L$
implies that of $A/L'$ (Lemma \ref{l:1}). For the opposite direction, 
if one of $(\mathscr L_{A/L})$ and $\mathrm{CH}_{\Lambda}(X_L)$ divides the other and $X_{L'}$ is torsion, then the main conjecture of $A/L'$ implies that for $A/L$ (see Lemma \ref{l:imcsp}). 

The specialization formula of $\mathrm{CH}_{\Lambda}(X_L)$
proven in \cite{tan13} actually holds for abelian varieties over global fields, thus, if under such circumstance,
an analytic side is proposed, it seems necessary to check if it satisfies the specialization formula of the same type.

\subsubsection{The restriction formulae}\label{ss:rest}
While the specialization formula handles the transfer of information between $L/K$ and $L'/K$, the restriction formula handles that between $L/K$ and $L/K'$, for a finite intermediate extension $K'/K$ of $L/K$.

Denote $\Phi=\Gal(L/K')$, $\Lambda_\Phi:=\Z_p[[\Phi]]$.
Choose a complete set of representatives $C \subset \Gamma$ of $\Gamma / \Phi$, so as to have the coset decomposition
$$
\Gamma=\bigsqcup_{\sigma \in C} \Phi \cdot \sigma.
$$
For each $f \in \Lambda_{\Gamma}$, write $f=\sum_{\sigma \in C} f_{\sigma} \cdot \sigma$, $f_{\sigma} \in \Lambda_{\Phi}$. View the dual group $\widehat{\Gamma / \Phi}$ as a subgroup of $\hat{\Gamma}$ consisting of characters trivial on $\Phi$. For each $\chi \in \widehat{\Gamma / \Phi}$, define
$$
f_{\chi}:=\sum_{\sigma \in C} f_{\sigma} \cdot \chi(\sigma) \cdot \sigma,
$$
which is independent of the choice of $C$. Define the restriction of $f$ to $\Phi$ to be 
$$
f_{\Phi}^{\Gamma}:=\prod_{\chi \in \widehat{\Gamma /\Phi}} f_\chi.$$

In \S\ref{su:rest}, we prove that $f_{\Phi}^{\Gamma} \in \Lambda_{\Phi}$, thus for a principal ideal $I=(f)\subset\Lambda_\Gamma$, we can define $I_\Phi^\Gamma:=(f_\Phi^\Gamma)$. Extend the restriction map to
$\Q_p\Lambda_\Gamma\longrightarrow \Q_p\Lambda_\Phi$.

The restriction formulae, given in Proposition \ref{p:restriction}, says that
\begin{equation*}
\mathrm{CH}_{\Lambda_{\Phi}}(X_{L})=\mathrm{CH}_{\Lambda_{\Gamma}}(X_{L})_{\Phi}^{\Gamma},
\end{equation*}
and as principal fractional ideals
\begin{equation*}
\mathscr{L}_{A / L / {K}^{\prime}}\cdot\Lambda_\Phi={\mathscr{L}_{A / {L} / {K}}}_{\Phi}^{\Gamma}\cdot\Lambda_\Phi.
\end{equation*}
It follows directly that if \eqref{e:imc} holds for $A/L/K$ then it also holds for $A/L/K'$, see Corollary \ref{c:restriction}.

\subsubsection{Special cases}\label{ss:imcspec}

Based on previously known results, the Iwasawa main conjecture \eqref{e:imc} is proven in several cases: 
\begin{enumerate}
\item[(I)] $L=K$ (Theorem \ref{t:l=k}),  
\item[(II)] $A / K$ is a constant elliptic curve (Theorem \ref{t:constant}), 
\item[(III)] $A / K$ has semistable reduction everywhere  and $L $ is the unramified $\mathbb{Z}_{p}$-extension $K_\infty^{(p)}$ of $K$ (Theorem \ref{t:constantfield}).
\end{enumerate}

\subsection{The valuations associated to characters}\label{su:app} 
Let $\mathsf{v}_{p}$ be the unique valuation on $\overline{\mathbb{Q}}_{p}$ with $\mathsf{v}_{p}(p)=1$, 
$\mathsf{v}_{p}(0)=+\infty$. For $\omega \in \hat{\Gamma}$, $\xi \in \mathbb{Q}_{p}\cdot \Lambda$, define
$$
\mathsf{v}_{\omega, L}(\xi)=\mathsf{v}_{p}(\omega(\xi)),
$$
For $I=(\xi)$, define
$$
\mathsf{v}_{\omega, L}(I)=\mathsf{v}_{\omega, L}(\xi).
$$
This is well-defined: replacing $\xi$ by $u\xi$ with $u\in\Lambda_\Gamma^\times$ multiplies $\omega(\xi)$ by the $p$-adic unit $\omega(u)$, leaving $\mathsf v_p(\omega(\xi))$ unchanged.

Recall that in the Noetherian topology of Monsky \cite[\S 1]{monsky},
a subset $Y\subset \hat\Gamma$ is closed if and only if it is a finite union of subsets each 
consisting of characters $\omega$ satisfying a finite set of equations $\omega(\tau_j)=\zeta_j$, $\tau_j\in\Gamma$, $\zeta_j\in\mu_{p^\infty}$, $j=1,...,m$, (\cite[Definition 1,2]{monsky}). 

\begin{myproposition}\label{p:val} Suppose $A / K$ has semistable reduction everywhere. There exists a proper Monsky closed set 
$\mathsf Z\subset  \hat{\Gamma}$, such that 
if $\omega\not\in\mathsf Z$, then
$$
\mathsf{v}_{\omega, L}(\mathscr{L}_{A / L})=\mathsf{v}_{\omega, L}(\mathrm{CH}_{\Lambda}(X_{L})).
$$
\end{myproposition}
The proposition is proven in \S\ref{ss:pfpval} based on \S\ref{ss:imcspec}(III) as well as direct consequences of 
the specialization and the restriction formulae, see Lemma \ref{l:red} and Lemma \ref{l:prod}.
From it, we can deduce the following proposition which is proven in \S\ref{ss:int}.

The $\mu$-invariant of a non-zero $\xi\in \Q_p\Lambda_\Gamma$ is defined to be the exponent $\mu(\xi)$
such that $\xi=p^{\mu(\xi)}\cdot \xi_0$,  where $\xi_0$ belongs to $\Lambda_\Gamma$, not divisible by $p$. Define
$\mu(0)=+\infty$.
If the ideal $I=(\xi)$, define $\mu(I):=\mu(\xi)$. For a finitely generated $\Lambda_\Gamma$-module $M$, define
$\mu(M):=\mu(\mathrm{CH}_{\Lambda_\Gamma}(M))$.

\begin{myproposition}\label{p:int}
If $A/K$ has semistable reduction everywhere, then $\mathscr{L}_{A / L} \in \Lambda$ and $\mu(\mathscr{L}_{A / L})=\mu(X_L)$.
\end{myproposition}

The following corollary of Proposition \ref{p:int} can be viewed as a special case of \eqref{e:imc}.

\begin{myproposition}\label{p:nontorsion} Suppose $A/K$ has semistable reduction everywhere. The dual Selmer group $X_{L}$ is non-torsion over $\Lambda$ if and only if $\mathscr{L}_{A / L/K}=0$.
\end{myproposition}

 \subsection{More applications}\label{su:gen} 
 Let $\mho(e, \Gamma)$ denote  the set of all intermediate $\Z_p^e$-extensions $L'/K$ of $L/K$.
 Write $\Gamma$ additively and identify $\mho(e,\Gamma)$ with the Grassmannian $\mathrm{Gr}(d-e,d)(\Z_p)=\mathrm{Gr}(d-e,d)(\Q_p)$, by assigning $L'/K$ to the $d-e$ dimensional subspace $\Q_p\Gal(L/L')\subset \Q_p\Gamma$, and by assigning to a $d-e$ dimensional subspace
$V\subset \Q_p\Gamma$ the fixed field of $V\cap\Gamma$. Then we endow $\mho(e,\Gamma)$ with the
Zariski topology of $\mathrm{Gr}(d-e,d)(\Z_p)$.

 
 \begin{myproposition}\label{p:e} Suppose $d> e\geq 2$. The Iwasawa main conjecture holds over $L$
 if and only if there is a non-empty Zariski open subset $O \subset \mho(e,\Gamma)$ such that
 the conjecture holds for all $\Z_p^e$-extensions belonging to $O$.
 \end{myproposition}
 The proof relies on the specialization formulae and is completed in \S\ref{su:specialcase}. 
We intentionally divide it into three steps, because the approaches in the first two steps are solely group ring theoretical, and might possibly be useful in other situation. See Proposition \ref{p:gen} and Proposition \ref{p:ge}.

In contrast to this proposition, for $e=1$, there are counterexamples.
For instance, take $f=t_1^2-\l \cdot t_2^2$, 
$g=t_1^2-(\l+p^3)\cdot t_2^2$, where $\l\in\Z$ but $\l$ not a square in $\Z_p$, $t_i=\sigma_i-1$,
and $\sigma_1,\sigma_2\in\Gamma$ extendable to a $\Z_p$-basis of $\Gamma$.
It is clear that $f$ and $g$ are relatively prime, while for every $L'/K\in\mho(1,\Gamma)$, in $\Lambda_{\Gamma'}$
the ideals
$$(p^L_{L'}(f))=(p^L_{L'}(g)).$$

\subsection{Notation}\label{su:notation} 
Let $\mathbb{F}_{q}$ and $\kappa$ respectively denote the {\em{constant field}} and the {\em{genus}} of $K$. As usual, for each place $v$ of $K$, let $K_{v}$, $\mathcal{O}_{v}$, and $\pi_{v}$ denote the completion of $K$ at $v$, the ring of integers of $K_{v}$, and a prime element. Denote $\mathbb{F}_{v}=\mathcal{O}_{v} / \pi_{v} \mathcal{O}_{v}$, $q_{v}=|\mathbb{F}_{v}|$. Let $\Frob_{q} \in \Gal(\overline{\mathbb{F}}_{q} / \mathbb{F}_{q})$ be the Frobenius substitution sending $x$ to $x^{q}$ and write $\Frob_{v}$ for $\Frob_{q_v}$. By abuse of notation, 
we also let them denote their images in $\Gal(\F_{q^{p^\infty}}/F_q)$. 

Write $\Gamma:=\Gal(L / K)$ and $\Lambda:=\Lambda_{\Gamma}:=\mathbb{Z}_{p}[[\Gamma]]$. Let $I_\Gamma$ denote the {\em{augmentation ideal}}, the kernel of $p^L_K$.
 For a finitely generated $\Lambda$-module $\mathrm{m}$, let $\mathrm{CH}_{\Lambda}(\mathrm{m}) \subset \Lambda$ denote the associated {\em{characteristic ideal}}, in particular, $\mathrm{CH}_{\Lambda}(\mathrm{m})=0$ if and only if $\mathrm{m}$ is non-torsion. Let $\Gamma_{v} \subset \Gamma$ denote the {\em{decomposition subgroup}}. If $M / K$ is an abelian extension unramified at $v$, let $[v]_{M} \in \Gal(M / K)$ denote the Frobenius element at $v$. For a divisor $D$ of $K$ such that every $v \in \mathrm{Supp}(D)$ is unramified over $M / K$, put $[D]_{M}=\prod_{v}[v]_{M}^{\ord_{v}(D)}$.

Call a place $v$ respectively good, bad, ordinary, multiplicative, split multiplicative, or non-split multiplicative, if $A$ has the corresponding type of reduction at $v$. Let $N=N_{A/K}$ and $\Delta_{A/K}$ denote the {\em{arithmetic conductor}} and the {\em{global discriminant}} of $A / K$. If $v \notin N$, let $\bar{A}$ denote the {\em{reduction}} of $A$ at $v$. Let $X_{L}$ denote the Pontryagin dual of the $p^{\infty}$-Selmer group of $A / L$.

Let $\mathbb{A}_{K}$ and $\mathbb{A}_{K}^{*}$ denote the ring of adeles and the group of ideles of $K$. 
For each effective divisor $D$ of $K$ define
$$
U_{D}:=\prod_{v \notin \mathrm{Supp}(D)} \mathcal{O}_{v}^{*} \cdot \prod_{v \in \mathrm{Supp}(D)} 1+\pi_{v}^{\ord_{v}(D)} \mathcal{O}_{v}
$$
which is an open subgroup of $\mathbb{A}_{K}^{*}$. Endow the quotient topology to the Weil group
$$
W_{D}:=K^{*} \backslash \mathbb{A}_{K}^{*} / U_{D}.
$$

For a given quasi-character $\omega: W_{D} \longrightarrow \mathbb{C}^{*}$, let $D_{\omega}$ denote the conductor. Let $\omega_{s}$, $s \in \mathbb{C}$, be the quasi-character of $W_{D}$, sending $x$ to $q^{-s \operatorname{deg}(x)}$. For a divisor $D^{\prime}=\sum_{v} D_{v}^{\prime}$, such that $\mathrm{Supp}(D^{\prime}) \cap \mathrm{Supp}(D)=\emptyset$, write $[D^{\prime}]_{D}$ for $[D^{\prime}]_{K^{D}}$, where $K^{D}$ denotes the ray class field associated to $D$. Thus, $[D^{\prime}]_{D}$ is the image of $\prod_{v} \pi_{v}^{\ord_{v}(D^{\prime})}$ under the natural map
$$
\bigoplus_{v} K_{v}^{*} \longrightarrow \mathbb{A}_{K}^{*} \longrightarrow W_{D}\longrightarrow \Gal(K^D/K),
$$
where the rightmost map is by the Artin reciprocity law, so that $\Gal(K^D/K)$ can be identified with the profinite completion $\tilde W_D$ of $W_D$. 

For a finite set $T$ of places of $K$, let $T_{m}$ (resp. $T_{o}$ ) denote the subset consisting of multiplicative (resp. good ordinary) places in $T$. The Galois group of the maximal abelian extension $\tilde K_T/K$ unramified outside $T$ is identified with the projective limit
$$
\tilde{W}_{T}:=\varprojlim _{\mathrm{Supp}(D)=T} \tilde W_{D}
$$
Endow the group $\mu_{p^{\infty}}$ with the discrete topology.
For a continuous character 
$$\omega: \tilde{W}_{T} \longrightarrow \mu_{p^\infty},$$ 
let $D_{\omega}$ denote the
global Artin conductor.
Since $\omega$ factors through $\tilde{W}_{T}\longrightarrow \tilde W_{D_\omega}$, 
it has a finite image: $\mu_{m}$, for some $m$. We shall fix an embedding
$$
\mu_{\infty} \longrightarrow \mathbb{C}^{*}
$$
such that $\omega$ can be viewed as a quasi-character of $\tilde{W}_{T}$, $\tilde W_{D}$, or $W_{D}$ (and the two definitions of $D_\omega$ coincide), $D \succeq D_{\omega}$. Put
$$
\tilde{\Lambda}_{T}:=\varprojlim _{\mathrm{Supp}(D)=T} \Z_p[[\tilde W_{D}]]
$$
and write $\mathbb{Q}_{p}\cdot \tilde{\Lambda}_{T}:=\mathbb{Q}_{p} \otimes_{\Z_p} \tilde{\Lambda}_{T}$. Also, we fix an embedding
$$
\mu_{\infty} \longrightarrow \overline{\mathbb{Q}}_{p}^{*},
$$
and extend $\omega$ to a continuous $\mathbb{Q}_{p}$-algebra homomorphism $\omega: \mathbb{Q}_{p} \tilde{\Lambda}_{T} \longrightarrow \mathbb{Q}_{p}(\mu_{m})$.

Let $S$ denote the {\em{ramification locus}} of $L / K$. In general, if $L^{\prime} / K$ is an intermediate  (not necessary multuple $\Z_p$) extension of $L / K$, denote the ramification locus $S^{\prime}$, write $\Gamma^{\prime}=\operatorname{Gal}(L^{\prime} / K), \Lambda^{\prime}=\Z_p[[\Gamma^{\prime}]]$, and let $p_{L^{\prime}}: \mathbb{Q}_{p} \tilde{\Lambda}_{S^{\prime}} \longrightarrow \mathbb{Q}_{p} \Lambda^{\prime}$ be the continuous $\mathbb{Q}_{p}$-algebra homomorphism induced from the surjection $\tilde{W}_{S^{\prime}} \longrightarrow \Gamma^{\prime}$. Let $\hat{\Gamma}^{\prime}$ denote the dual group of $\Gamma^{\prime}$, which consists of continuous characters $\omega: \Gamma^{\prime} \longrightarrow \mu_{p^{\infty}}$. We identify each $\omega \in \hat{\Gamma}^{\prime}$ with a continuous character on $\tilde{W}_{S^{\prime}}$, and also extend $\omega$ to a continuous $\mathbb{Q}_{p}$-algebra homomorphism $\mathbb{Q}_{p} \Lambda^{\prime} \longrightarrow \mathbb{Q}_{p}(\mu_{p^{m}})$. Let $p_{L^{\prime}}^{L}: \mathbb{Q}_{p} \Lambda \longrightarrow \mathbb{Q}_{p} \Lambda^{\prime}$ be the $\mathbb{Q}_{p}$-algebra homomorphism induced from $\Gamma \longrightarrow \Gamma^{\prime}$.

Let $A_{p^{m}}$ denote the kernel of the multiplication by $p^{m}$ on $A$, as a commutative group scheme over $K$, put $A_{p^{\infty}}:=\bigcup_{m=1}^{\infty} A_{p^{m}}$. For a $p$-primary abelian group $M$, write $M_{\text {div }}$ for its $p$-divisible part. Let $M^{\vee}$ denote the Pontryagin dual of $M$, whenever it is defined.

\section{The specialization formula of $\mathrm{CH}_\Lambda(X_L)$}\label{s:chx}
Let $L'/K$ be an intermediate $Z_p^e$-extension of $L/K$, $d>e\geq 0$. Denote $\Psi:=\Gal(L/L')$.
The specialization formula relates $p^L_{L'}(\mathrm{CH}_{\Lambda_\Gamma}(X_L))$ and $\mathrm{CH}_{\Lambda'}(X_{L'})$, and also involvs correcting factors.

\subsection{Correcting factors}\label{su:correct}
Recall that for each elliptic curve $\mathsf A$ defined over a finite field $\mathsf k$
of order $\mathsf q$, there are associated conjugate Weil $\mathsf q$-numbers $\alpha$ and $\beta$ in $\overline{\mathbb{Q}}$ such that $\alpha\beta=\mathsf q$ and $|A(\mathsf k)|=1-(\alpha+\beta)+\mathsf q$.
Furthermore, if $\mathsf A/\mathsf k$ is ordinary, then the imaginary quadratic field $\Q(\alpha,\beta)$ splits at
$p$, so we can
choose the embedding of it into $\Q_p$ such that $\alpha\in\mathbb{Z}_{p}^{*}$.

If $\mathsf A$ is the reduction of $A$ at $v$, write $\alpha_v$, $\beta_v$ for $\alpha$, $\beta$.

\subsubsection{The factor $\nabla_{A / L}$}\label{ss:nabla}Suppose $A / K$ is a constant elliptic curve defined over $\mathbb{F}_{q}$, and let $\alpha$ and $\beta$ be the associated Weil $q$-numbers.
Since $A/K$ is ordinary, $\alpha\in\mathbb{Z}_{p}^{*}$,
we put
$$
\nabla_{A}=(1-\alpha^{-1} \Frob_{q})(1-\alpha^{-1} \operatorname{Frob}_{q}^{-1}) \in \mathbb{Z}_{p}[[\mathrm{Gal}(
\F_{q^{p^\infty}} / \mathbb{F}_{q})]].
$$
Note that $S=\emptyset$, if and only if $L$ is the unramified $\Z_p$-extension of $K$, namely,
$$
L=K_{\infty}^{(p)}=\mathbb{F}_{q^{p^\infty}} K.
$$
In this case, we identify $\Gamma$ with $\Gal(\mathbb{F}_{q^{p^\infty}} / \mathbb{F}_{q})$
so that $\nabla_A\in\Lambda$.
\begin{definition}\label{d:nabla} Define the element
$$
\nabla_{A / L}= \begin{cases}\nabla_{A}, & \text { if } A / K \text { is a constant elliptic curve and } L=K_{\infty}^{(p)} \\ 1, & \text { otherwise. }\end{cases}
$$
Write $\nabla_{A / L / K}$ for $\nabla_{A / L}$, if needed.
\end{definition}

\subsubsection{The number $\mathrm{t}_{A / L}$}\label{ss:mathrmt} 
Define the natural number
\begin{equation}\label{e:t}
\mathrm{t}_{A / L}= 
\begin{cases}|A_{p^\infty}(K)|^{2}, & \text { if } L=K; \\ 
1, & \text {otherwise. }
\end{cases}
\end{equation}

\subsubsection{The factor $\varrho_{L/L'}$}\label{ss:varrho}
\begin{definition}\label{d:varrho} 
Define the principal ideal
$$
\varrho_{L / L^{\prime}}:=\begin{cases}
\Lambda', & \text { if } e \geq 2 ; \\
\nabla_{A / L^{\prime}}\cdot\Lambda', & \text { if } e=1 ; \\
|A_{p^{\infty}}(K)|^{2}\cdot\Lambda', & \text { if } e=0, L \neq K_{\infty}^{(p)} ; \\
( |A_{p^{\infty}}(K)|^{2} /|A_{p^{\infty}}(K) \cap A_{p^{\infty}}(L)_{\text {div }}|^{2})\cdot\Lambda', & \text { if } e=0, L=K_{\infty}^{(p)} .
\end{cases}
$$
\end{definition}
The factor $\varrho_{L / L^{\prime}}$ is generated by a natural number if $e\not=1$.
Note that if $A/K$ is a constant curve, then $\alpha$, being a Weil $q$-number, is not a root of $1$, so in general, for every $\omega\in\hat\Gamma'$, 
\begin{equation}\label{e:varrho}
\omega(\varrho_{L / L^{\prime}} )\not=0, \text{ and in particular, } \varrho_{L / L^{\prime}}\not\subset I_{\Gamma'}.
\end{equation}

\begin{lemma}\label{l:varrho} We have 
$$
\varrho_{L / L^{\prime}}=p_{L^{\prime}}^{L}(\mathrm{t}_{A / L}^{-1} \cdot \nabla_{A / L}^{-1}) \cdot \mathrm{t}_{A / L^{\prime}} \cdot \nabla_{A / L^{\prime}}\cdot\Lambda' .
$$
\end{lemma}
\begin{proof} If $e \geq 2$, then $\nabla_{A / L}=\mathrm{t}_{A / L}=\nabla_{A / L^{\prime}}=\mathrm{t}_{A / L^{\prime}}=1$. Also,
$\nabla_{A / L}=$ $\mathrm{t}_{A / L}=\mathrm{t}_{A / L^{\prime}}=1$, if $e=1$. 
If $e=0$ and $L=K_\infty^{(p)}$, use the fact (see \cite[Lemma 1.2]{tan13}) that in $\Z_p$, the ideals
$(p^L_{L'}(\nabla_{A/L}))=(|A_{p^{\infty}}(K) \cap A_{p^{\infty}}(L)_{div }|^2)$. 
\end{proof}

\subsubsection{The factor $\vartheta_{L/L'}$}\label{ss:vartheta}
If $v \nmid N$, put $\lambda_{v}=\alpha_{v}+\beta_{v}$, it is the trace of the Frobenius endomorphism $\mathrm{F}_{v}$ of the reduction elliptic curve $\bar{A}$. 
As usual, for $v  \mid N$, denote

$$
\lambda_{v}= \begin{cases}
1, & \text { if } \mathrm{A} \text { has split multiplicative reduction; } \\ 
-1, & \text { if } \mathrm{A} \text { has non-split multiplicative reduction; } \\ 0, & \text { if } \mathrm{A} \text { has additive reduction. }\end{cases}
$$
If $A / K$ is non-constant and $f_{A}$ is the corresponding cuspidal modular function, then $\lambda_{v}$ is the eigenvalue of $f_{A}$ under the Hecke operator $T_{v}$ [De70, We71]. 
If $A / K$ is a constant elliptic curve and $\alpha,\beta$ are the 
Weil $q$-numbers in \S\ref{ss:nabla}, then $A$ has good ordinary reduction at every $v$ and 
$$
\alpha_{v}=\alpha^{\operatorname{deg}(v)}, \beta_{v}=\beta^{\operatorname{deg}(v)}. 
$$

Let $S_{1}^{\prime}$ be the subset consisting of all $v \in S^{\prime}$ such that $\Gamma_{v}^{\prime} \simeq \mathbb{Z}_{p}$. For each $v \in S_{1}^{\prime}$, choose a topological generator $\sigma_{v}^{\prime}$ of $\Gamma_{v}^{\prime}$. The ideal $(1\pm\sigma_v')\cdot\Lambda'$ is independent of the choice of $\sigma_v'$. 
Let $m_v$ denote the order of the component group of $A/K_v$.

\begin{definition}\label{d:vartheta} Define $\vartheta_{L / L^{\prime}}=\prod_{v} \vartheta_{L / L^{\prime}, v}$, where 
each
$\vartheta_{L / L^{\prime}, v}$ is a principal ideal  in $\Lambda^{\prime}$, defined as follows:
\begin{enumerate}
\item[(a)] Suppose $v \notin S$. If $\Psi_{v} \neq 0$, then $\vartheta_{L / L^{\prime}, v}:=m_{v}\cdot\Lambda'$; otherwise, $\vartheta_{L / L^{\prime}, v}=\Lambda'$.

\item[(b)] Suppose $v$ is a good ordinary place in $S$. If $v$ is unramified over $L^{\prime} / K$, then
$$
\vartheta_{L / L^{\prime}, v}:=(1-\alpha_{v}^{-1}[v]_{L^{\prime}})(1-\alpha_{v}^{-1}[v]_{L^{\prime}}^{-1})\cdot\Lambda' ;
$$
otherwise, $\vartheta_{L / L^{\prime}, v}:=\Lambda'$.

\item[(c)] Suppose $v$ is a split multiplicative place in $S$. Let $Q_{v}$ be the local Tate period and let 
$\overline{\{Q_{v}\}}$ denote the closed subgroup generated by the image of $Q_{v}$ under the local reciprocity law map $K_{v}^{*} \longrightarrow \Gamma_{v}$. Then\footnote{We take this opportunity to make the erratum that in [Tan14, Definition1.3(c)], the condition in the 3rd item should be replaced by ``if 
$\Psi_{v} \simeq \Z_{p}^{f}$, $f \geq 1$, and $\Gamma_{v}^{\prime}$ is topologically generated by $\sigma$ ".}$$
\vartheta_{L / L^{\prime}, v}:= 
\begin{cases}
0, & \text { if } \Psi_{v} \simeq \mathbb{Z}_{p}^{f}, f \geq 2, \text { and } \Gamma_{v}^{\prime}=0 ; \\ 
(1-\sigma_{v}^{\prime})\cdot\Lambda', & \text { if } \Psi_{v} \simeq \mathbb{Z}_{p}^{f}, f \geq 1, v \in S_{1}^{\prime} ; 
\\ (1-[v]_{L^{\prime}})\cdot\Lambda', & \text { if } \Psi_{v} \simeq \mathbb{Z}_{p}^{f}, f \geq 1, \Gamma_{v}^{\prime} \simeq \mathbb{Z}_{p}, v \notin S^{\prime} ; 
\\ \mathfrak{w}_{v} \cdot\Lambda', & \text { if } \Psi_{v} \simeq \mathbb{Z}_{p}, \Gamma_{v}^{\prime}=0, \text { and }(\mathfrak{w}_{v})=\mathrm{CH}_{\mathbb{Z}_{p}}(\Gamma_{v} / \overline{\{Q_{v}\}}) ; \\ \Lambda', & \text { otherwise. }
\end{cases}
$$

\item[(d)] Suppose $v$ is a non-split multiplicative place in $S$. Then
$$
\vartheta_{L / L^{\prime}, v}:= 
\begin{cases}
2 m_{v}\cdot\Lambda', & \text { if } \Gamma_{v}^{\prime}=0, \Psi_{v} \simeq \mathbb{Z}_{p}^{f}, f \geq 2, 
\\ 
2 m_v\cdot\Lambda', & \text{ if }  \Gamma_{v}^{\prime}=0, \Psi_{v} \simeq \mathbb{Z}_{p}, \mathbb{F}_{q_{v}^{2}} \not \subset L_{v} ; \\
m_{v}\cdot\Lambda', & \text { if } \Gamma_{v}^{\prime}=0, \Psi_{v} \simeq \mathbb{Z}_{p}, \mathbb{F}_{q_{v}^{2}} \subset L_{v} ; 
\\ (1+[v]_{L^{\prime}})\cdot\Lambda', & \text { if } \Gamma_{v}^{\prime}\simeq \Z_p, v\not\in S', \\ 
(1+\sigma_{v}^{\prime})\cdot\Lambda', & \text { if } v \in S_{1}^{\prime}, \Psi_{v} \neq 0, \mathbb{F}_{q_{v}^{2}} \subset L_{v}^{\prime} ; \\ \Lambda', & \text { otherwise. }
\end{cases}
$$
\end{enumerate}
\end{definition}

\subsection{The specialization formula}\label{su:aspf} 
\begin{myproposition}\label{p:aspf} If $L^{\prime} / K$ is an intermediate $\mathbb{Z}_{p}^{e}$-extension of $L / K, d>e \geq 0$, then as ideals in $\Lambda'$,
\begin{equation}\label{e:chspf}
\varrho_{L / L'} \cdot p^L_{ L^{\prime}}(\mathrm{CH}_{\Lambda}(X_{L}))=\vartheta_{L / L^{\prime}} \cdot \mathrm{CH}_{\Lambda^{\prime}}(X_{L^{\prime}}).
\end{equation}
\end{myproposition}

If $S$ contains no non-split multiplicative places, it turns out that the proposition is actually a consequence of \cite[Theorem 1]{tan13}. To see this, we have to compare the local factor $\vartheta_v$ defined in \cite[definition 1.7]{tan13} with our $\vartheta_{L/L',v}$, and the global factor defined in \cite[Definition 1.3]{tan13},
which we denote by $\tilde\varrho_{L/L'}$, with our $\varrho_{L/L'}$. That $\vartheta_v=(\vartheta_{L/L',v})$, for $v\in S$, good ordinary or split multiplicative, or $v\not\in S$, can be checked directly.

\subsubsection{The factor $\tilde\varrho_{L/L'}$}\label{ss:global} For a field $M\supset K$,
let $\mathrm{T}_{p} A_{p^{\infty}}(M)$ denote the Tate module of $A_{p^{\infty}}(M)$.

\begin{lemma}\label{l:tatemodule} For a $\mathbb{Z}_{p}^{\mathsf f}$-extension $\mathsf L/ K$, over
$\mathsf\Lambda:=\mathbb{Z}_{p}[[\Gal(\mathsf{L} / K)]]$,
the Tate module $\mathrm{T}_{p} A_{p^{\infty}}(\mathsf {L})$ satisfies
$$
\mathrm{CH}_{\mathsf\Lambda}(\mathrm{T}_{p} A_{p^{\infty}}(\mathsf{L})) \cdot \mathrm{CH}_{\mathsf\Lambda}(\mathrm{T}_{p} A_{p^{\infty}}(\mathsf{L}))^{\sharp}=(\nabla_{A/\mathsf{L}}).
$$
If $\mathsf{f}=1$, then $A_{p^{\infty}}(\mathsf{L})$ is finite, unless $A / K$ is a constant curve and $\mathsf{L}=K_{\infty}^{(p)}$.
\end{lemma}

\begin{proof} If $\mathsf{f}=0$, then the lemma is trivial. By \cite[Theorem 4.2]{blv09}, if $A / K$ is non-isotrivial, then $A_{p^{\infty}}(\bar{K}^{s})$ is finite, and hence $\mathrm{T}_{p} A_{p^{\infty}}(\mathsf{L})=0$, so the lemma follows. Moreover, as a $\Z_p$-module $\mathrm{T}_{p} A_{p^{\infty}}(\mathsf{L}) \simeq \mathbb{Z}_{p}$, or 0 , so
it is pseudo-null over $\mathsf\Lambda$ unless $\mathsf{f}=1$. 

Suppose $\mathsf f=1$, $A / K$ is isotrivial and $A_{p^\infty}(\mathsf L)$ is infinite. Then $A$ is a twist of a constant curve $B / \F_{q}$. All endomorphisms of $B$ are defined over $\F_{q}$ \cite[Theorem 2(c)]{tat66},
and $A$ is the twist by some 
$\varphi \in \Hom(\Gal (\bar{K}^{s} / K), \mathrm{Aut}(B))$. Let $F$ denote the fixed field of $\ker(\varphi)$.

Put $M:=K(B[p])$, so that $K(B_{p^\infty}(\bar K))=K_{\infty}^{(p)}M$.
Since over $MF$, $A=B$, so $K_\infty^{(p)}MF$ is the unique $\Z_p$-extension
of $MF$ containing $K(A_{p^\infty}(\bar K))$, hence $\mathsf LMF=K_\infty^{(p)}MF$. Since $MF$ is a finite abelian extension of
$K$, it follows that $\mathsf L=K_\infty^{(p)}$.



We claim that $F \subset M$. Since $M \subset \overline{\mathbb{F}}_{q}K$, this implies $A / K$ is a constant curve. If $F \not \subset M$, then there exists some $\sigma \in \Gal(\bar{K}^{s} / K)$ such that $\sigma|_{M}=$ $i d$, but $\varphi(\sigma) \neq i d$. By the twist, we know that $A_{p^{\infty}}(M\mathsf L)=A_{p^{\infty}}(\mathsf{L}) \simeq \mathbb{Q}_{p} / \mathbb{Z}_{p}$, is a subset of the fixed points of $\varphi(\sigma)$ acting on 
$B(FM\mathsf L)$. But a non-identity automorphism has only finite amount of fixed points, a contradiction.

Finally, if $A / K$ is a constant curve and $\mathsf{L}=K_{\infty}^{(p)}$, then $K(A_{p^{\infty}}(\bar K))=K(A_p(\bar K))\mathsf{L}$, where $K(A_p(\bar K))/K$ is a finite cyclic extension of degree dividing $p-1$. If $K(A_p(\bar K))=K$, then the lemma follows, since $A_{p^{\infty}}(\mathsf{L})=A_{p^{\infty}}(\overline{\F})$, and hence as an ideal in $\Lambda$,
$$
\mathrm{CH}_{\Lambda}(\mathrm{T}_{p} A_{p^{\infty}}(\mathrm{L}))=(\Frob_{q}-\alpha)=(1-\alpha^{-1} \Frob_{q}).
$$

If $K(A_p(\bar K))/ K$ is a nontrivial extension, then $A_{p^{\infty}}(\mathsf{L})$ must be trivial, so 
$$\mathrm{CH}_{\Lambda}(\mathrm{T}_{p} A_{p^{\infty}}(\mathsf{L}))=\Lambda.$$
But in this case, $\alpha \not\equiv 1{\pmod p}$. 
Hence

$$
1-\alpha^{-1} \Frob_{q}=(1-\alpha^{-1})-\alpha^{-1}(\Frob_{q}-1) \in \Lambda^{*}.
$$
\end{proof}

When $S$ contains a non-split multiplicative place, \cite[Definition 1.3]{tan13} still can be used to define
an ideal, also denoted by $\tilde{\varrho}_{L / L^{\prime}}$.
\begin{corollary}\label{c:varrho} As ideals in $\Lambda^{\prime}, \tilde{\varrho}_{L / L^{\prime}}=(\varrho_{L / L^{\prime}})$.
\end{corollary}

\begin{proof} Comparing Definition \ref{d:varrho} and \cite[Definition 1.3]{tan13}, we see that
if $e \geq 2$, then $\varrho_{L / L^{\prime}}=\tilde{\varrho}_{L / L^{\prime}}=\Lambda'$; 
if $e=1$, then by Lemma \ref{l:tatemodule}, $\varrho_{L / L^{\prime}}=(\nabla_{A / L^{\prime}})=\tilde{\varrho}_{L / L^{\prime}}$; if $e=0$, 
then $
\varrho_{L / L^{\prime}}=(\left|A_{p^{\infty}}(K)|^{2} /|A_{p^{\infty}}(K) \cap A_{p^{\infty}}(L)_{d i v}\right|^{2})\cdot \Lambda'=\tilde{\varrho}_{L / L^{\prime}}$.

\end{proof}

\subsubsection{The proof of Proposition \ref{p:aspf}}\label{ss:pfaspf}

As in \cite{tan13}, the proof of the proposition can be simplified by setting a sequence 
$$
L^{\prime} \subset L^{\prime \prime} \subset \cdots \subset L^{(i)} \subset L^{(d-e+1)}=L
$$
with $\Gal(L^{(i)} / K) \simeq \mathbb{Z}_{p}^{e-1+i}$. Because
$$
p^L_{L^{\prime}}=p^{L^{\prime \prime}}_{L^{\prime}} \circ \cdots \circ p^{L^{(i+1)} }_{ L^{(i)}} \circ \cdots \circ p^L_{ L^{(d-e)}}.
$$
\begin{lemma}\label{l:spvartheta}
Let $L^{(i)}, i=1, \ldots, d-e+1$, be as above. Then in $\Lambda'$, the ideals
$$\vartheta_{L / L^{\prime}}=\prod_{i=1}^{d-e} p^{L^{(i)}}_{ L^{\prime}}(\vartheta_{L^{(i+1)} / L^{(i)}}),$$
and 
$$
\varrho_{L / L^{\prime}}=\prod_{i=1}^{d-e} p^{L^{(i)}}_{ L^{\prime}}(\varrho_{L^{(i+1)} / L^{(i)}}). 
$$
\end{lemma}

\begin{proof}The second identity follows from Lemma \ref{l:varrho}. The first can be checked locally: $\vartheta_{L / L^{\prime}, v}=\prod_{i=1}^{d-e} p^{L^{(i)}}_{L^{\prime}}(\vartheta_{L^{(i+1)} / L^{(i)}, v})$ by using Definition \ref{d:vartheta}, see Appendix for the details.
\end{proof}

In view of Lemma \ref{l:spvartheta}, for proving Proposition \ref{p:aspf}, we may, and 
we will assume that $d=e+1$. Since our $S$ is allowed to contain non-split multiplicative places,
to follow the proof of [Tan14, Theorem 1], we need to replace \cite[Lemma 1.8, Proposition 3.5, 3.6, 3.7]{tan13} by the following lemma for which the proof is put in \S\ref{su:loc}.
At a place $v$ of $K$, define
$$
\mathcal{W}_{v}^{i}:=(\bigoplus_{x} \mathrm{H}^{i}(\Psi_{x}, A(L_{x})))^{\vee}, i=1,2,
$$
where $\mathrm{H}^{i}(\Psi_{x}, A(L_{x}))$ is viewed as a discrete group, $x$ runs through all places of $L^{\prime}$ sitting over $v$, and $L_{x}$ denote the completion of $L$ at any chosen place $w$ sitting over $x$. Put
$$
\mathcal{W}^{i}(A, L / L^{\prime}):=\bigoplus_{\text {all } v} \mathcal{W}_{v}^{i}.
$$

\begin{lemma}\label{l:calW} We have $\mathrm{CH}_{\Lambda^{\prime}}(\mathcal{W}^{1}(A, L / L^{\prime}))=\vartheta_{L / L^{\prime}}$. The $\Lambda^{\prime}$-modules $\mathcal{W}^{1}(A, L / L^{\prime})$ and $\mathcal{W}^{2}(A, L / L^{\prime})$ are either both torsion or both non-torsion, they are non-torsion if and only if there is a split multiplicative $v \in S$ such that $\Gamma_{v}^{\prime}=0$ and 
$\mathfrak{w}_{v}=0$. Moreover,
$\mathcal{W}^{2}(A, L / L^{\prime})=0$, if it is torsion.
\end{lemma}

Note that \cite[\S2]{tan13} is irrelevant to $S$, so its results still hold,
\cite[Lemma 4.5]{tan13} remains valid for our $\mathcal{W}^{i}(A, L / L^{\prime})$, $i=1,2$.
Also, other results in \cite[\S4, \S5.1]{tan13} do not involve the nature of $S$, so they all hold in our situation.
Thus, we complete the proof by following the procedure of the proof of \cite[Theorem 1]{tan13},
because the results used in the proof remain valid under our assumption on $S$.

\subsection{Local cohomology}\label{su:loc} In this subsection, we assume that $d=e+1$, or equivalently, $\Psi \simeq \mathbb{Z}_{p}$. Let $x$ be a place of $L'$ sitting over a place $v$ of $K$ and $w$ a place of $L$ sitting over $x$. To have notation consistent with \cite{tan13}, 
we shall write 
$$\mathcal{W}_{w}^{i}:=\mathrm{H}^{i}(\Psi_{w}, A(L_{w}))^\vee=\mathrm{H}^{i}(\Psi_{x}, A(L_{x}))^\vee.$$
Let $\Lambda_{w}^{\prime}$ denote the Iwasawa algebra $\mathbb{Z}_{p}[[\Gamma_{x}^{\prime}]]$. Define $\vartheta_{v}^{(i)}=\mathrm{CH}_{\Lambda^{\prime}} \mathcal{W}_{v}^{i}$ and $\vartheta_{w}^{(i)}=\mathrm{CH}_{\Lambda_{w}^{\prime}} \mathcal{W}_{w}^{i}$. By \cite[Lemma 3.2]{tan13}, which is independent of the reduction of $A$ at $v$, if $\mathcal{W}_{w}^{i}$ is finitely generated over $\Lambda_{w}^{\prime}$, then $\mathcal{W}_{v}^{i}$ is finitely generated over $\Lambda^{\prime}$ and
$$
\mathcal{W}_{v}^{i}=\Lambda^{\prime} \otimes_{\Lambda_{w}^{\prime}} \mathcal{W}_{w}^{i}, \quad \vartheta_{v}^{(i)}=\vartheta_{w}^{(i)} \cdot \Lambda^{\prime}.
$$

\begin{lemma}\label{l:locold} Suppose $\Psi \simeq \mathbb{Z}_{p}$ and $v \notin S$, or $v$ is either good ordinary or split multiplicative place in $S$. Then $\vartheta_{v}^{(1)}=\vartheta_{L / L^{\prime}, v}$. 
Moreover, $\mathcal{W}_{w}^{2}=0$ and $\vartheta_{v}^{(2)}=(1)$, unless $v \in S$ is a split multiplicative place, 
$\Gamma_{v} \simeq \mathbb{Z}_{p}$, $\Gamma_{v}^{\prime}=0$ and $\vartheta_{L / L^{\prime}, v}=0$, in which case $\vartheta_{v}^{(2)}=0$.
\end{lemma}

\begin{proof} This is due to \cite[Proposition 3.5, 3.6, 3.7]{tan13}\footnote{In \cite{tan13}, our $m_v$ is denoted as $\pi_v$.}.
\end{proof}

\subsubsection{Non-split multiplicative places}\label{ss:nonsp}
Next, we assume that $v$ is a non-split multiplicative place of $A / K$. Let $\tilde{K}_{v} / K_{v}$ be the unique unramified quadratic extension. Consider the Tate curve $B=\mathbb{G}_{m} / Q^\Z$ over $K_{v}$ such that
$A$ equals the twist of $B$ by the non-trivial character of $\Gal(\tilde{K}_{v} / K_{v})$. In general, if $\tilde{\mathcal{K}} / K_{v}$ is a Galois extension containing $\tilde{K}_{v}$, or equivalently containing the quadratic extension $\F_{q_{v}^{2}}$ of $\F_{v}$, then $A(\tilde{\mathcal{K}})=\tilde{\mathcal{K}}^{*} / Q^\Z$ and for $P \in A(\tilde{\mathcal{K}})$, $P=x 
\pmod {Q^\Z}$, $x \in \tilde{\mathcal{K}}^{*}$, the Galois action of $\mathcal{G}:=\Gal(\tilde{\mathcal{K}} / K_{v})$, is given by
\begin{equation}\label{e:gact}
\tensor[^g] P{}=(\tensor[^g] x{})^{\epsilon_{g}} \quad \pmod {Q^\Z},\:g\in\mathcal G,
\end{equation}
where $\epsilon_{g}$ is the image of $g$ under the composition $\mathcal{G} \longrightarrow \Gal(\mathbb{F}_{q_{v}^{2}} / \mathbb{F}_{v}) \stackrel{\sim}{\longrightarrow}\{ \pm 1\}$. 

Let $\bar A_o(\F_v)$ denote the rational points on identity component of the special fibre of the N\'{e}ron model, and let $A_o(K_v)$ be its pre-image under the reduction map. By taking $\tilde{\mathcal K}=\tilde{K}_{v}$ and applying \cite[Appendix C. Theorem 14.1(c)]{sil86}, we see that
$$A_o(K_v)=A(K_v)\cap B_o(\tilde K_v)=\{x\in \O_{\tilde K_v}^*\;\mid\; \Nm_{\tilde K_v/K_v}(x)=1\},$$
so 
\begin{equation}\label{e:cv}
m_v= \begin{cases}2, & \text { if } \ord_{v}(Q) \text { is an even number; } \\ 
1, & \text { otherwise }.
\end{cases}
\end{equation}

For a Galois extension $\mathcal{K} / K_{v}$, let $\tilde{\mathcal{K}}$ denote $\tilde{K}_{v} \mathcal{K}$. If $\tilde{\mathcal{K}} / \mathcal{K}$ is non-trivial, and hence an unramified quadratic extension, then by the twist,
$$
\coh^{1}(\tilde{\mathcal{K}} / \mathcal{K}, A(\tilde{\mathcal{K}}))=\coh^{2}(\tilde{\mathcal{K}} / \mathcal{K}, B(\tilde{\mathcal{K}}))=B(\mathcal{K}) / \Nm_{\tilde{\mathcal{K}} / \mathcal{K}}(B(\tilde{\mathcal{K}}))=\mathcal{K}^{*} / \Nm_{\tilde{\mathcal{K}} / \mathcal{K}}(\tilde{\mathcal{K}}^{*}) Q^{\mathbb{Z}}.
$$
Since all local units in $\mathcal{K}$ are contained in $\Nm_{\tilde{\mathcal{K}} / \mathcal{K}}(\tilde{\mathcal{K}}^{*})$, we see that
\begin{equation}\label{e:h12or1}
|\coh^{1}(\tilde{\mathcal{K}} / \mathcal{K}, A(\tilde{\mathcal{K}}))|=2 \text {, or } 1,
\end{equation}
depending on if the $\mathcal{K}$-valuation of $Q$ is even or not. Therefore,
\begin{equation}\label{e:h1h2}
|\coh^{1}(\tilde{K}_{v} / K_{v}, A(\tilde{K}_{v}))|=m_v.
\end{equation}

Note that if $\mathcal{K} / K_{v}$ is finite, then by Tate's local duality theorem $\coh^{2}(\tilde{\mathcal{K}} / \mathcal{K}, A(\tilde{\mathcal{K}}))$ is the Pontryagin dual of $\coh^{1}(\tilde{\mathcal{K}} / \mathcal{K}, A(\tilde{\mathcal{K}}))$, see \cite[Lemma 2.13]{tan13}. Hence, its order is bounded by 2. In general, $\coh^{2}(\tilde{\mathcal{K}} / \mathcal{K}, A(\tilde{\mathcal{K}}))$ is the dual of $\varinjlim _{\mathcal{F}} \coh^{1}(\tilde{\mathcal{F}} / \mathcal{F}, A(\tilde{\mathcal{F}}))$, where $\mathcal{F}$ runs through all finite intermediate extensions of $\mathcal{K} / K_{v}$, and hence also
\begin{equation}\label{e:h21or2}
|\coh^{2}(\tilde{\mathcal{K}} / \mathcal{K}, A(\tilde{\mathcal{K}}))| \leq 2.
\end{equation}

\subsubsection{The group $\coh_A^i$}\label{ss:coha}
Let notations be as in \S\ref{ss:nonsp}. For an elliptic curve $C / K_{v}$, denote
$$
\coh_{C}^{i}:=\coh^{i}(\tilde{L}_{w} / \tilde{L}_{w}^{\prime}, C(\tilde{L}_{w})),\; i=1,2,
$$
where the Galois group $\Gal (\tilde L'_w / K_v)$ acts on $\coh_{C}^{i}$ in the usual way.

\begin{lemma}\label{l:calh} 
$\Gal (\tilde{L}_w' / K_v)$ acts trivially on $\coh_{B}^{i}$, $i=1,2$.
$\coh_B^1$ is isomorphic to either $\Q_{p} / \Z_{p}$ or a finite cyclic p-group, the latter case occurs only if $\tilde{L}_{w}^{\prime}=\tilde{K}_{v}$ or $\tilde{L}_{w}^{\prime}=\tilde{L}_{w}$.
$\coh_{B}^{2}=0$, unless $\Gamma_{v} \simeq \Z_{p}$, $L_{w}^{\prime}=K_{v}$, and $\coh_{B}^{1} \simeq \Q_{p} / \Z_{p}$, in which case $\coh_{B}^{2} \simeq \Q_{p} / \Z_{p}$.
 \end{lemma}

\begin{proof}
Because $\coh^{1}(\tilde{L}_{w} / \tilde{L}_{w}^{\prime}, \tilde{L}_{w}^{*})=0$, the long exact sequence induced by
$$
\xymatrix{0 \ar[r] &  Q^\Z \ar[r] &  \tilde{L}_{w}^{*} \ar[r] &  B(\tilde{L}_{w})) \ar[r] & 0}
$$
gives rise to the inclusion
$$
\xymatrix{\coh_{B}^{1} \ar@{^(->}[r] &  \coh^{2}(\tilde{L}_{w} / \tilde{L}_{w}^{\prime}, Q^\Z),}
$$
while
$$
\xymatrix{\coh^{2}(\tilde{L}_{w} / \tilde{L}_{w}^{\prime}, Q^\Z) \ar[r]^-\sim &  
\coh^{2}(\tilde{L}_{w} / \tilde{L}_{w}^{\prime}, \Z)=\Hom(\Gal(\tilde{L}_{w} / \tilde{L}_{w}^{\prime}), \Q / \Z) \ar@{^(->}[r] & \mathbb{Q}_{p} / \mathbb{Z}_{p},}
$$
as $\Gal(\tilde{L}_{w} / K_{v})$-modules. The action of $\Gal(\tilde{L}_{w} / K_{v})$ on $\Gal(\tilde{L}_{w} / \tilde{L}_{w}^{\prime})$ is via the conjugation in $\Gal (\tilde{L}_{w} / K_{v})$, since $\tilde{L}_{w} / K_{v}$ is abelian, the action is trivial. If $\tilde{L}_{w} \neq \tilde{L}_{w}^{\prime} \neq$ $\tilde{K}_{v}$, it is known that $\coh_{B}^{1}$ is infinite (see \cite[p.1058-1059]{tan13}).

The assertion regarding $\mathrm{H}_{B}^{2}$ is mainly deduced by applying Lemma \ref{l:locold} to the triple $\tilde{L}_{w} / \tilde{L}_{w}^{\prime} / \tilde{K}_{v}$. In the case where $\Gamma_{v} \simeq \mathbb{Z}_{p}, L_{w}^{\prime}=K_{v}$, and $\mathrm{H}_{B}^{1} \simeq \mathbb{Q}_{p} / \mathbb{Z}_{p}$, the lemma says $\mathrm{H}_{B}^{2}$ is infinite. Let $\tilde{L}_{w, n}^{\prime}$ denotes the $n$th layer of $\tilde{L}_{w} / \tilde{L}_{w}^{\prime}$. Now $\mathrm{H}_{B}^{2}$ is the direct limit of $\mathrm{H}^{2}(\tilde{L}_{w, n}^{\prime} / \tilde{L}_{w}^{\prime}, B(\tilde{L}_{w, n}^{\prime}))$, which is the Pontryagin dual of the finite cyclic $p$-group $\mathrm{H}^{1}(\tilde{L}_{w, n}^{\prime} / \tilde{L}_{w}^{\prime}, B(\tilde{L}_{w, n}^{\prime}))$, so $\mathrm{H}_{B}^{2} \simeq \mathbb{Q}_{p} / \mathbb{Z}_{p}$, as $\Gal (\tilde{L}_{w} / K_{v})$-modules.
\end{proof}
The underlying abelian group of $\mathrm{H}_{A}^{i}$ is the same as $\mathrm{H}_{B}^{i}$.
By the twist, if $a\in \mathrm{H}_{A}^{i}$ corresponds to $b\in \mathrm{H}_{B}^{i}$, then for $\sigma\in\Gal (\tilde{L}_{w}^{\prime} / K_{v})$, $\tensor[^\sigma]a{}$ corresponds to $\chi(\sigma)\cdot b$,
where
$$
\chi: \Gal (\tilde{L}_{w}^{\prime} / K_{v}) \longrightarrow \Gal(\tilde{K}_{v} / K_{v}) \stackrel{\sim}{\longrightarrow} \Gal(\mathbb{F}_{q_{v}^{2}} / \mathbb{F}_{v}) \stackrel{\sim}{\longrightarrow}\{ \pm 1\}.
$$

\begin{lemma}\label{l:hGal} 
If $\tilde{L}_{w}^{\prime} / L_{w}^{\prime}$ is non-trivial, then $(\mathrm{H}_{A}^{i})^{\Gal(\tilde{L}_{w}^{\prime} / L_{w}^{\prime})}$ is of order 2 or 1, depending on if $\mathrm{H}_{B}^{i}$ has an element of order 2 or not.
\end{lemma}
\begin{proof} By Lemma \ref{l:calh},
an element in $(\mathrm{H}_{A}^{i})^{\Gal(\tilde{L}_{w}^{\prime} / L_{w}^{\prime})}$
corresponds to one $x\in \mathrm{H}_{B}^{i}$ such that $x=-x$. The lemma also says that $\coh_B^i$ has at most one order $2$ element. 
\end{proof}

\subsubsection{The proof of Lemma \ref{l:calW}}\label{ss:calWpf}
Lemma \ref{l:calW} follows from Lemma \ref{l:locold} and the lemma below. Let the notations be as in \S\ref{ss:nonsp}.
\begin{lemma}\label{l:locch} Suppose $\Psi \simeq \mathbb{Z}_{p}$. If $v \in S$ is a non-split multiplicative place, then $\vartheta_{v}^{(1)}=\vartheta_{L / L^{\prime}, v}$ and $\mathcal{W}_{v}^{2}=0$.
\end{lemma}

\begin{proof} We only need to treat the $\Psi_{w} \simeq \mathbb{Z}_{p}$ case.
For simplicity, denote
$$
\mathcal{D}^{i}:=\mathrm{H}^{i}(\Psi_{w}, A(L_{w})),\; i=1,2.
$$
These are $p$-primary groups, since $\Psi_w\simeq\Z_p$.
Note that $\mathcal{D}^2$ is $p$-divisible. Indeed, the fact that $\Z_p$ is of cohomological dimension $1$ implies 
for $i>1$, $$\coh^i(\Psi_{w}, A_p(L_{w}))=\coh^i(\Psi_{w}, A(L_{w})/pA(L_{w}))=0,$$
and hence by the exact sequence
$$\xymatrix{0\ar[r] & A_p(L_w) \ar[r] & A(L_{w}) \ar[r]^-{[p]} & A(L_{w}) \ar[r] & A(L_{w})/pA(L_{w}) \ar[r] & 0,}$$  
we conclude that $p\cdot \mathcal D^2=\mathcal D^2$.


Suppose $\mathbb{F}_{q_{v}^{2}} \subset L_{w}^{\prime}$. Then $\Gamma_{v}^{\prime} \neq 0$ and $A/L_w'$
has split multiplicative reduction. Lemma \ref{l:locold} applied to the triple $L_{w} / L_{w}^{\prime} / \tilde{K}_{v}$ shows that $\mathcal{W}_{w}^{2}=0$. Also, Lemma \ref{l:calh} says $\mathcal{D}^{1} \simeq \mathbb{Q}_{p} / \mathbb{Z}_{p}$ with $\Gamma_{w}^{\prime}=\Gal(\tilde{L}_{w}^{\prime} / K_{v})$ acting via $\chi$. Thus, if the $\mathbb{Z}_{p}$-rank of $\Gamma_{v}^{\prime}$ is greater than 1 , then $\mathcal{W}_{w}^{1}$ is a pseudo-null $\Lambda_{w}^{\prime}$-module; if $v \in S_{1}^{\prime}$, then $\mathrm{CH}_{\Lambda_{w}^{\prime}}(\mathcal{W}_{w}^{1})=(-1-\sigma_{v}^{\prime})$; if $\Gamma_{v}^{\prime} \simeq \mathbb{Z}_{p}$, $v \notin S^{\prime}$, then $\mathrm{CH}_{\Lambda_{w}^{\prime}}(\mathcal{W}_{w}^{1})=(-1-[v]_{L^{\prime}}^{-1})$.

For the rest of the proof, assume that $\F_{q_{v}^{2}} \not \subset L_{w}^{\prime}$, so by Lemma \ref{l:hGal}, 
$(\mathrm{H}_{A}^{i} )^{\Gal (\tilde{L}_{w}^{\prime} / L_{w}^{\prime})}$ has order bounded by $2$. Consider the composition of inflation and restriction:
$$
\mathsf c_{i}: \xymatrix{\mathcal{D}^{i} \ar[r]^-{\mathsf{inf_i}} & \coh^{i}(\tilde{L}_{w} / L_{w}^{\prime}, A(\tilde{L}_{w})) \ar[r]^-{\mathsf{res_i}} & (\coh_{A}^{i})^{\Gal (\tilde{L}_{w}^{\prime} / L_{w}^{\prime})}.}
$$
By the Hochschild-Serre spectral sequences (see [Mil80, p.105]), $\mathsf{inf_1}$ is injective, and
$$
\ker(\mathsf{res_1})=\coh^{1}(\tilde{L}_{w}^{\prime} / L_{w}^{\prime}, A(\tilde{L}_{w}^{\prime})),
$$
which, by \eqref{e:h12or1}, is of order $1$ or $2$, and so is $\ker (\mathsf c_{1})$. This shows that $\mathcal{D}^{1}$ is a finite 2-group. Also, $\ker(\mathsf{inf_2})$ is a quotient of $\coh^{1}(\tilde{L}_{w} / L_{w}, A(\tilde{L}_{w}))^{\Gal (\tilde{L}_{w} / L_{w}^{\prime})}$ which, by \eqref{e:h12or1} again, is finite, and there is an exact sequence
$$
\coh^{2} (\tilde{L}_{w}^{\prime} / L_{w}^{\prime}, A(\tilde{L}_{w}^{\prime})) \longrightarrow \ker(\mathsf{res_2}) \longrightarrow \coh^{1}(\tilde{L}_{w}^{\prime} / L_{w}^{\prime}, \coh_{A}^{1}).
$$
By \eqref{e:h21or2}, $\coh^{2}(\tilde{L}_{w}^{\prime} / L_{w}^{\prime}, A(\tilde{L}_{w}^{\prime}))$ is finite. The group $\coh^{1}(\tilde{L}_{w}^{\prime} / L_{w}^{\prime}, \coh_{A}^{1})$ is finite, because it is isomorphic to
$$
\coh^{2}(\tilde{L}_{w}^{\prime} / L_{w}^{\prime}, \coh_{B}^{1})=\coh_{B}^{1} / 2 \coh_{B}^{1},
$$
which, by Lemma \ref{l:calh}, is finite. Hence, $\ker (\mathsf{res_2})$ is finite. Therefore, $\mathcal{D}^{2}$ is finite, $p$-primary and $p$-divisible, so must be trivial. This shows that $\mathcal{W}_{w}^{1}$ is finite and $\mathcal{W}_{w}^{2}=0$.

If $\Gamma_{w}^{\prime} \neq 0$, then $\mathcal{W}_{w}^{1}$ is pseudo-null. In this case either $p\not=2$ or $p=2$ and $v\in S'$, both imply  $\vartheta_{L/L',v}=\Lambda'$. Hence, the lemma holds. 

It remains to treat the $\Gamma_{w}^{\prime}=0$ case. 
Let $\sigma_{v}$ be a topological generator of $\Gamma_v$. Let $K_{w, m}$ and $\tilde{K}_{w, m}$ respectively denote the $m$th layers of $L_{w} / K_{v}$ and $\tilde{L}_{w} / \tilde{K}_{v}$. Write $\Gamma_{v, m}$ for 
$\Gamma_{v} / \Gamma_{v}^{p^{m}}=\Gal (K_{w, m} / K_{v})$.

Suppose $L_{w}$ contains $\F_{q_v^2}$. This occurs if and only if $p=2$ and $K_{w, 1}=\tilde{K}_{v}$.
The Galois group $\Gamma_{v}^2=\Gal (L_{w} / \tilde{K}_{v})$ is topologically generated by $\sigma_{v}^{2}$. We claim that the restriction
$$
\mathsf{res} :  \mathcal{D}^{1} \longrightarrow \coh^{1} (L_{w} / \tilde{K}_{v}, A(L_{w}))
$$
is a zero map. Then it follows that
$$
\mathcal{D}^{1}=\coh^{1}(\tilde{K}_{v} / K_{v}, A(\tilde{K}_{v})),
$$
and by \eqref{e:cv}, \eqref{e:h1h2}, it is of order $m_v$. Hence, $\mathrm{CH}_{\Lambda_{v}^{\prime}}(\mathcal{W}_{v}^{1})=(\vartheta_{L / L^{\prime}, v})$, as desired.

We prove the claim by showing that, for $m>1$, the restriction map $\mathsf {r_{(m)}}$ below is trivial. By choosing the restrictions of $\sigma_{v}$ and $\sigma_{v}^{2}$ respectively as the generators of $\Gamma_{v, m}$ and $\Gamma_{v, m}^{2}$, we have the commutative diagram
$$\xymatrix{\coh^1(K_{w,m}/K_v, A(K_{w,m})) \ar[r]^-{\mathsf{res_{(m)}}} \ar[d]^-{\simeq} 
& \coh^1(K_{w,m}/K_{w,1}, A(K_{w,m}))\ar[d]^-{\simeq}\\
\coh^{-1} (K_{w,m}/K_v, A(K_{w,m}))  \ar[r]^-{\mathsf{r_{(m)}}} &  \coh^{-1}(K_{w,m}/K_{w,1}, A(K_{w,m})).} 
$$

Let $\eta$ be an element in $\coh^{-1} (K_{w, m} / K_{v}, A(K_{w, m}))$, represented by $P \in A(K_{w, m})$ such that $\Nm_{K_{w, m} / K_{v}}(P)=0$. Then $\mathsf{r_{(m)}}(\eta)$ is represented by $P+\tensor[^{\sigma_{v}}] P{}$. Furthermore, if we write $P=x \pmod {Q^\Z},\; x \in K_{w, m}^{*}$, then, due to the twist,
$$
P+{ }^{\sigma_{v}} P=x \cdot { }^{\sigma_{v}}x^{-1} \pmod {Q^\Z}.
$$
Because $x \cdot { }^{\sigma_{v}} x^{-1} \in \mathcal{O}_{w, m}^{*}$, we must have
$$
\Nm_{K_{w, m} / K_{w, 1}}(x \cdot { }^{\sigma_{v}} x^{-1}) \in \mathcal{O}_{w, 1}^{*} \cap Q^\Z=\{1\} .
$$
By Hilbert's Theorem 90, we can write
$$
x \cdot { }^{\sigma_{v}} x^{-1}=y \cdot { }^{\sigma_{v}^{2}} y^{-1}, \;\text { for some } \;y \in K_{w, m}^{*} .
$$
This actually means $P+{ }^{\sigma_{v}} P$ is a co-boundary. The claim is proved.

Finally, consider the case where $\tilde{K}_{v}$ is not contained in $L_{w}$. This occurs if and only if $p \neq 2$, or $p=2$ and $L_{w} / K_{v}$ is totally ramified. Since $\mathcal{D}^{1}$ is a finite 2 group, if $p \neq 2$, then 
$\mathrm{CH}_{\Z_p} (\mathcal{W}_{w}^{1})=(1)$. Since $\vartheta_{L / L^{\prime}, v}=-2m_v$ is a
$p$-adic unit, the lemma is proved.

It remains to treat the $p=2$ case. By the Tate's local duality theorem, we need to show that $|A(K_{v}) / \mathrm{N}_{K_{w, m} / K_{v}}(A(K_{w, m}))|=2 m_v=|\vartheta_{L/L',v}|$, for sufficiently large $m$.

Let $\tau$ denote the generator of $\Gal (\tilde{L}_{w} / L_{w})$. For each non-negative integer $m$, let
$\mathrm N_{m,w}$ denote the norm map $\tilde K_{w,m}\longrightarrow K_{w,m}$. By the twist, we can write
\begin{equation}\label{e:tw}
A(K_{w,m})=\mathrm N^{-1}_{m,w}(Q^\Z)/Q^Z.
\end{equation}
By Hilbert's Theorem 90,
$$
\mathcal{N}_{m}:=\mathrm N^{-1}_{w,m}(1)=\{{ }^{\tau} x \cdot x^{-1} \mid x \in \tilde{K}_{w, m}^{*}\}.
$$
Since  $\mathcal N_m\subset \mathcal{O}_{\tilde{K}_{w, m}}^{*}$,  $\mathcal N_m\cap Q^\Z=\{1\}$, 
so we can view $\mathcal N_m$  as a subgroup of $A(K_{w,m})$. Because $Q\in \mathrm N^{-1}_{w,m}(Q^2)$, the quotient $A(K_{w,m})/\mathcal N_m$ has order $1$ or $2$, and by \eqref{e:tw},
$$A(K_{w,m})=\mathcal N_m \sqcup \mathrm N_{m,w}^{-1}(Q),$$ 
where $\mathrm N_{m,w}^{-1}(Q)$ is either an empty set or an $\mathcal N_m$-coset.
Since $L_{w} / K_{v}$ is totally ramified, for $m\geq 1$, $\Nm_{\tilde K_{m,w}/\tilde K_v}(\mathrm N_{m,w}^{-1}(Q))\subset\Nm_{\tilde K_{m,w}/\tilde K_v}(\mathcal N_m)\cdot Q^\Z$, so as a subgroup of $A(K_v)$,
$$\Nm_{K_{w,m}/K_v}(A(K_w,m))=\Nm_{\tilde K_{m,w}/\tilde K_v}(\mathcal N_m)\subset \mathcal N_0.$$
Also, since $\tilde K_v/K_v$ is unramified, $\mathrm N^{-1}_{0,w}(Q)$ is non-empty,
if and only if $\ord_v(Q)$ is even, or equivalently, $m_v=2$ by \eqref{e:cv}. Therefore,
$|A(K_v)/\mathcal N_0|=m_v$.
The snake lemma applied to 
$$\xymatrix{0\ar[r] & K_{w,m}^* \ar[r] \ar[d]^-{N_1} & \tilde{K}_{w,m}^*\ar[r]^-{\tau-1}\ar[d]^-{N_2} & \mathcal N_m
\ar[d]^-{\Nm_{\tilde K_{w,m}/\tilde K_v}} \ar[r] & 0\\
0\ar[r] & K_v^* \ar[r] & \tilde K_v^* \ar[r]^-{\tau-1} & \mathcal N_0 \ar[r] & 0,}
$$
where $m \geq 1$, and $N_{1}$, $N_{2}$ are norm maps, gives the exact sequence
$$
\cdots \longrightarrow \coker (N_{1}) \stackrel{j}{\longrightarrow} \coker (N_{2}) \longrightarrow \mathcal{N}_{0} / \mathrm{N}_{\tilde{K}_{w, m} / \tilde{K}_{v}}(\mathcal{N}_{m}) \longrightarrow 0 .
$$
By the local class field theory, both $\coker(N_{1})$ and $\coker(N_{2})$ can be identified with an open dense subgroup of $\Gamma_{v, m}$ and under this identification the map $j$ becomes the multiplication by 2 , so the exact sequence implies $|\mathcal{N}_{0} / \mathrm{N}_{\tilde{K}_{w, m} / \tilde{K}_{v}}(\mathcal{N}_{m})|=2$.
Finally,
$$
|A(K_{v}) / \mathrm{N}_{K_{w, m} / K_{v}}(A(K_{w, m}))|=|A (K_{v}) / \mathcal{N}_{0} | \cdot |\mathcal{N}_{0} / \mathrm{N}_{\tilde{K}_{w, m} / \tilde{K}_{v}}(\mathcal{N}_{m})|=m_v \cdot 2 .
$$

\end{proof}

\section{The $p$-adic $L$-functions}\label{s:lfun} 
In \S\ref{su:pre}, we recall related materials, especially the Theta elements $\Theta_D$ formulated by Mazur \cite{maz87}.
In \S\ref{su:tildeL}, the element $\tilde{\mathscr L}_{A,T}\in \Q_p\tilde\Lambda_T$ is introduced, while
$\hat{\mathscr L}_{A/L}, \mathscr L_{A/L}\in\Q_p\Lambda$ are respectively defined in \S\ref{su:hatL} and \S\ref{su:L}. Finally, in \S\ref{su:mtt}, the Mazur-Tate-Teitelbaum conjecture for $\hat{\mathscr L}_{A/L}$
is discussed.

\subsection{The preliminary}\label{su:pre}
Let $\lambda_v$, $\alpha_v$, and $\beta_v$ be as defined in \S\ref{ss:vartheta}.

\subsubsection{Twisted Hasse-Weil L-functions}\label{ss:thw} 
Let $\omega$ be a quasi-character on $W_{D}$. As usual, for each $v \notin \mathrm{Supp}(D_{\omega})$ and $s \in \mathbb{C}$, define 
$$L_{A, v}(\omega, s)= 
\begin{cases}
(1-\alpha_{v} \cdot \omega([v]_{D_{\omega}}) \cdot q_{v}^{-s})^{-1}(1-\beta_{v} \cdot \omega([v]_{D_{\omega}}) \cdot q_{v}^{-s})^{-1}, & \text { if } v \notin \mathrm{Supp}(N) ; \\ 
(1-\lambda_{v} \cdot \omega([v]_{D_{\omega}}) \cdot q_{v}^{-s})^{-1}, & 
\text { if } v \in \mathrm{Supp}(N),
\end{cases}$$
and define the associated twisted Hasse-Weil $L$-function to be

$$
L_{A}(\omega, s):=\prod_{v \notin \mathrm{Supp}(D_{\omega})} L_{A, v}(\omega, s), \text { for } s \geq 1 .
$$
If it is necessary to emphasize the ground field $K$, we will write $L_{A/K}(\omega, s)$ for $L_A(\omega, s)$.
The $L$-function has an analytical continuation to a polynomial in $q^{-s}, s \in \mathbb{C}$, unless $A / K$ is the constant curve and $D_{\omega}=0$, such that $\omega$ factors through 
$\xymatrix{W_{D} \ar[r]^{\deg} & \mathbb{Z} \ar@{^(->}[r] & \Gal(\overline{\mathbb{F}}_{q} / F_{q})}$, 
in which case $\delta_{\alpha}(q^{-s}) \cdot \delta_{\beta}(q^{-s})\cdot  L_{A}(\omega, s)$ is a polynomial in 
$q^{-s}$. Here
$$
\delta_{t}(q^{-s}):=(1-t \cdot \omega(\Frob_{q}) \cdot q^{-s})(1-t \cdot \omega(\Frob_{q}) \cdot q^{-(s-1)}).
$$

Since $\alpha, \beta$ are conjugate Weil $q$-numbers, 
$\delta_{\alpha}(q^{-1}) \cdot \delta_{\beta}(q^{-1}) \neq 0$. Therefore, in all cases, $L_{A}(\omega, 1)$ is defined.

\subsubsection{The Gauss sum}\label{ss:gauss} Here we define the Gauss sum $\tau_{\omega}$ associated to a quasi-character $\omega$ of $\tilde{W}_{S}$. Our Gauss sum is related to the Gauss Sum 
in \cite[VII, Proposition 14]{we74}, which we denote by $\tau^{W}(\omega)$, see $\eqref{e:compare}$.

Let $\Psi: K \backslash \mathbb{A}_{K} \longrightarrow \mathbb{C}^{*}$ be a non-trivial continuous character. Write $\Psi=\prod_{v} \Psi_{v}$ and let $a=(a_{v})_v$ be a differential idele attached to $\Psi$ (see \cite{tat50}). For a given $\omega$, 
put $b=(b_{v})_v$ with $b_{v}=a_{v} \cdot \pi_{v}^{\ord_{v}(D_{\omega})}$,
set
\begin{equation}\label{e:tau}
\tau_{\omega, v}= 
\begin{cases}
\sum_{x \in \mathcal{O}_{v}^{*} / 1+\pi_{v}^{\ord_v(D_{\omega})} \mathcal{O}_{v}} \omega(b_{v}^{-1} x) \Psi_{v}(b_{v}^{-1} x), & \text { if } \ord_{v}(D_{\omega})>0 \\ 
\omega(b_{v}^{-1}), & \text { if } \ord_{v}(D_{\omega})=0,
\end{cases}
\end{equation}
and define the Gauss Sum:
$$
\tau_{\omega}=\prod_{v} \tau_{\omega, v},
$$
which is independent of the choice of $\Psi$. To compare $\tau_{\omega}$ with $\tau_{\omega}^{W}$, consider the Haar measure $\mu_{v}$ on $K_{v}$ such that $\mu_{v}(\mathcal{O}_{v})=1$. Then
$$
\mu_{v}(1+\pi_{v}^{\ord_{v}(D_{\omega})} \mathcal{O}_{v})=q_{v}^{-\ord_{v}(D_{\omega})} .
$$
Hence, if $\ord_{v}(D_{\omega})>0$, then
\begin{equation}\label{e:haar}
\begin{aligned}
\tau_{\omega, v} & =\omega(b_{v}^{-1}) \cdot \sum_{x \in \mathcal{O}_{v}^{*} / 1+\pi_{v}^{\ord_{v}(D_{\omega})} \mathcal{O}_{v}} \omega(x) \Psi_{v}(b_{v}^{-1} x) \\
& =\omega(b_{v}^{-1}) \cdot q_{v}^{\ord_{v}(D_{\omega})} \cdot \int_{x \in \mathcal{O}_{v}^{*}} \omega(x) \Psi_{v}(b_{v}^{-1} x) d \mu_{v}(x) .
\end{aligned}
\end{equation}
Write $\tilde{\mu}_{v}:=\|a_{v}\|_{v}^{\frac{1}{2}} \mu_{v}$ which is the self-dual Haar measure with respect to $\Psi_{v}$. Then
$$
\tau_{\omega, v}= 
\begin{cases}
\omega(b_{v}^{-1}) q_{v}^{\frac{1}{2} \ord_{v}(D_{\omega})}\|b_{v}\|_{v}^{-\frac{1}{2}} \int_{x \in \mathcal{O}_{v}^{*}} \omega(x) \Psi_{v}(b_{v}^{-1} x) d \tilde{\mu}_{v}(x), & \text { if } \ord_{v}(D_{\omega})>0 \\ 
\omega(b_{v}^{-1}), & \text { if } \ord_{v}(D_{\omega})=0.
\end{cases}
$$
This together with the definition of $\tau^W$ implies
\begin{equation}\label{e:compare}
\tau_{\omega}=\omega(b^{-1}) q^{\frac{1}{2} \operatorname{deg}(D_{\omega})} \tau^{W}(\omega^{-1})
\end{equation}

By \cite[Lemma 3]{tan93}, we have
\begin{equation}\label{e:taus}
\tau_{\omega \omega_{s}}=q^{s \cdot(\deg(D)+2 \kappa-2)} \cdot \tau_{\omega} .
\end{equation}

\subsubsection{Theta elements}\label{ss:theta} Here we recall the theta element $\Theta_{D}=\Theta_{A, D} \in \mathbb{Q}[W_{D}]$ defined in \cite{maz87}, that interpolates special values of HasseWeil L-functions twisted by $\omega\in \hat D$. The main reference is \cite{tan93}.

For $D^{\prime} \succeq D \succeq 0$, let
$$
Z_{D}^{D^{\prime}}: \mathbb{Q}_{p}[W_{D^{\prime}}] \longrightarrow \mathbb{Q}_{p}[W_{D}]
$$
be the ring homomorphism induced by the projection $p_{D}^{D^{\prime}}: W_{D^{\prime}} \longrightarrow W_{D}$. 
Also, let
$$
V_{D^{\prime}}^{D}: \mathbb{Q}_{p}[W_{D}] \longrightarrow \mathbb{Q}_{p}[W_{D^{\prime}}]
$$
denote the homomorphism of $\Q_p$-modules such that for $g \in W_{D}$,
$$
V_{D^{\prime}}^{D}(g):= 
\begin{cases}
\sum_{g^{\prime} \mapsto g} g^{\prime}, & \text { if } D \neq 0, \text { or } D^{\prime}=D=0 ; \\ 
(q-1) \sum_{g^{\prime} \mapsto g} g^{\prime}, & \text { if } D=0 \text { but } D^{\prime} \neq 0,
\end{cases}
$$
where in the summation, $g^{\prime}$ runs through all pre-image of $g$ under
$p_{D}^{D^{\prime}}: W_{D^{\prime}} \longrightarrow W_{D}$.
Then the following equations hold. First, 
\begin{equation}\label{e:vz}
Z_{D}^{D^{\prime}} \circ V_{D^{\prime}}^{D}=b_{D}^{D^{\prime}} \cdot id,
\end{equation}
where
\begin{equation}\label{e:bdd}
b_{D}^{D^{\prime}}=\prod_{\ord_{v}(D^{\prime})>\ord_{v}(D)>0} q_{v}^{\ord_{v}(D^{\prime}-D)} \prod_{\ord_{v}(D^{\prime})>\ord_{v}(D)=0}(q_{v}^{\ord_{v}(D^{\prime})}-q_{v}^{\ord_{v}(D^{\prime})-1}).
\end{equation}
For effective $D_{1}, D_{2}, D_{3}$,
\begin{equation}\label{e:z123}
Z_{D_{1}}^{D_{1}+D_{2}+D_{3}}=Z_{D_{1}}^{D_{1}+D_{2}} \circ Z_{D_{1}+D_{2}}^{D_{1}+D_{2}+D_{3}},
\end{equation}
as well as
\begin{equation}\label{e:v123}
V_{D_{1}+D_{2}+D_{3}}^{D_{1}}=V_{D_{1}+D_{2}+D_{3}}^{D_{1}+D_{2}} \circ V_{D_{1}+D_{2}}^{D_{1}}.
\end{equation}
Finally, if $\mathrm{Supp}(D_{2}) \cap \mathrm{Supp}(D_{3})=\emptyset$, then
\begin{equation}\label{e:vz123}
V_{D_{1}+D_{2}}^{D_{1}} \circ Z_{D_{1}}^{D_{1}+D_{3}}=Z_{D_{1}+D_{2}}^{D_{1}+D_{2}+D_{3}} \circ V_{D_{1}+D_{2}+D_{3}}^{D_{1}+D_{3}}.
\end{equation}

\begin{theorem}\label{t:theta} For each effective divisor $D$ supported on a given finite set $T$ of places of $K$, there exists a unique theta element $\Theta_{D} \in 
\frac{1}{p^\aleph} \mathbb{Z}[W_{D}]$, for some non-negative integer $\aleph$ independent of $D$, satisfying the following:

\begin{enumerate}
\item[(a)] If $v \notin \mathrm{Supp}(N) \cup \mathrm{Supp}(D)$, then
$$
Z_{D}^{v+D}(\Theta_{v+D})=(\lambda_{v}-[v]_{D}-[v]_{D}^{-1}) \Theta_{D} .
$$

\item[(b)] If $v \in \mathrm{Supp}(N)$ and $v \notin \mathrm{Supp}(D)$, then
$$
Z_{D}^{v+D}(\Theta_{v+D})=(\lambda_{v}-[v]_{D}^{-1}) \Theta_{D} .
$$

\item[(c)] If $v \notin \mathrm{Supp}(N)$ and $v \in \mathrm{Supp}(D)$, then
$$
Z_{D}^{v+D}(\Theta_{v+D})=\lambda_{v} \cdot \Theta_{D}-V_{D}^{-v+D}(\Theta_{-v+D}) .
$$

\item[(d)] If $v \in \mathrm{Supp}(N) \cap \mathrm{Supp}(D)$, then
$$
Z_{D}^{v+D}(\Theta_{v+D})=\lambda_{v} \cdot \Theta_{D} .
$$

\item[(e)] If $\omega$ is a primitive quasi-character of $W_{D}$, then
$$
\omega(\Theta_{D})=\tau_{\omega} \cdot q^{\frac{\deg(\Delta)}{12}+\kappa-1} L_{A}(\omega, 1) .
$$
\end{enumerate}
\end{theorem}

\begin{proof} If $A / K$ is non-constant, then it is modular \cite{de70, we71}. 
Let
$$
f: \mathrm{GL}_{2}(\mathbb{A}_{K}) \longrightarrow \mathbb{C}
$$
be the associated cuspidal modular function of level $N$. As in \cite[(2.2)]{tan93}, define
$$
\Theta_{D, f}:=\ell_{D}^{-1} \cdot \sum_{w \in W_{D}} f((\begin{array}{cc}
d \cdot w & w \\
0 & 1
\end{array})) \cdot w,
$$
where $d$ is an idele representing $D$, each $w \in W_{D}$ is lifted to an element of 
$\mathbb{A}_{K}^{*}$, and
$$
\ell_{D}= 
\begin{cases}1, & \text { if } D \text { is nontrivial } \\ 
q-1, & \text { if } D \text { is trivial }.
\end{cases}
$$
Put
$$
\Theta_{D}:=q^{\deg(\Delta) / 12+\kappa-1} \cdot \Theta_{D, f} .
$$
Basically, the theorem is the same as \cite[Proposition 2]{tan93} 
for which we have to make the following erratum:
\begin{enumerate}
\item $\Theta_{D, E}$ should be $\Theta_{D, f}$, and so on.
\item In (c), $V_{D^{\prime}}$ should be $V_{D-v}$ and in the case where $D=v$, 
$V_{D-v}$ should be multiplied by $q-1$ so that it is the same as our $V_{D}^{D-v}$.
\end{enumerate}

If $A / K$ is a constant elliptic curve, the theorem is proven in \cite[p. 308 $\thicksim$ 309]{tan93}.
\end{proof}

\subsection{The element $\tilde{\mathscr{L}}_{A,T}$}\label{su:tildeL}
We view  $\Theta_D$ as an element of $\frac{1}{p^\aleph}\Z_p[[\tilde W_D]]$ and extend $Z^{D'}_D$ and $V_{D'}^D$ respectively to 
$$\Q_p\Z_p[[\tilde W_{D'}]]\longrightarrow \Q_p\Z_p[[\tilde W_D]]$$ 
and 
$$\Q_p\Z_p[[\tilde W_D]]\longrightarrow \Q_p\Z_p[[\tilde W_{D'}]].$$

Let $T$ be a finite set containing only ordinary places of $K$. By a standard method (see [MTT86]), 
we shall put together all $\Theta_{D}$ for which $D$ has support equal $T$ to form an element $\tilde{\mathscr{L}}_{A, T} \in \mathbb{Q}_{p}\cdot \tilde{\Lambda}_{T}$ that interpolates special values of twisted Hasse-Weil L-functions
as below.

\subsubsection{The elements $\tilde{\mathscr{L}}_{D}$ and $\tilde{\mathscr{L}}_{A, T}$}\label{ss:elts} 
For a non-empty $J \subset T$, denote $D_{J}:=\sum_{v \in J} v$, and put $D_{\emptyset}:=0$. If $J \subset T_{o}$, denote
$$
\alpha^{J}:=\prod_{v \in J} \alpha_{v},
$$
and put $\alpha^{\emptyset}:=1$. For a divisor $D$ supported on $T$, put
$$
\alpha_{D}=\prod_{v \in T_{o}} \alpha_{v}^{\ord_{v}(D)} \cdot \prod_{v \in T_{m} \cap \mathrm{Supp}(D)} \lambda_{v}^{\ord_{v}(D)-1}.
$$
Following \cite{mtt86}, for each $D$ supported on $T$, define

\begin{equation}\label{e:tildeL}
\tilde{\mathscr{L}}_{D}=\sum_{\emptyset \subseteq J \subseteq T_{o} \cap D}(-1)^{|J|} \frac{1}{\alpha_{D} \alpha^{J}} V_{D}^{D-D_{J}}(\Theta_{D-D_{J}})
\end{equation}
as an element of $\frac{1}{p^{\aleph}} \mathbb{Z}_{p}[[\tilde W_{D}]]$.

\begin{lemma}\label{l:thetacomp} For divisors $D_{1} \succeq D_{2} \succeq 0$, supported on $T$,
denote $T^{(i)}=\mathrm{Supp}(D_{i})$, $T_{m}^{(i)}=T_{m} \cap T^{(i)}$, and $T_{o}^{(i)}=T_{o} \cap T^{(i)}$. Then
$$
Z_{D_{2}}^{D_{1}}(\tilde{\mathscr{L}}_{D_{1}})=\prod_{v \in T_{m}^{(1)} \backslash T_{m}^{(2)}}(\lambda_{v}-[v]_{D_{2}}^{-1}) \cdot \prod_{v \in T_{o}^{(1)} \backslash T_{o}^{(2)}}(1-\alpha_{v}^{-1}[v]_{D_{2}})
(1-\alpha_{v}^{-1}[v]_{D_{2}}^{-1}) \cdot \tilde{\mathscr{L}}_{D_{2}}.
$$
\end{lemma}

\begin{proof}
By \eqref{e:z123}, it is sufficient to consider the case where $D_{1}=D_{2}+v$. First,
 assume that $\ord_{v}(D_{1}) \geq 2$, or equivalently $v \in \mathrm{Supp}(D_{2})$. 
 If $v \in T_{m}$, then the lemma follows from \eqref{e:vz123} and Theorem \ref{t:theta}(d), because we have
$$
\begin{array}{rcl}
Z_{D_{2}}^{D_{1}}(\frac{(-1)^{|J|}}{\alpha_{D_{1}} \alpha^{J}} V_{D_{1}}^{D_{1}-D_{J}}(\Theta_{D_{1}-D_{J}})) & =& \frac{(-1)^{|J|}}{\lambda_{v} \alpha_{D_{2}} \alpha^{J}} V_{D_{2}}^{D_{2}-D_{J}}(Z_{D_{2}-D_{J}}^{D_{1}-D_{J}}(\Theta_{D_{1}-D_{J}})) \\
& =& \frac{(-1)^{|J|}}{\alpha_{D_{2}} \alpha^{J}} V_{D_{2}}^{D_{2}-D_{J}}(\Theta_{D_{2}-D_{J}}).
\end{array}
$$

Suppose $v \in T_{o}$. For $J$ not containing $v$, write $J^{\prime}=J \cup\{v\}$. Since $J$ is a subset of $\mathrm{Supp}(D_{2})$, by \eqref{e:v123}, \eqref{e:vz123} and Theorem \ref{t:theta}(c)
$$
\begin{array}{rcl}
& & Z_{D_{2}}^{D_{1}}(\frac{(-1)^{|J|}}{\alpha_{D_{1}} \alpha^{J}} V_{D_{1}}^{D_{1}-D_{J}}(\Theta_{D_{1}-D_{J}}))\\
 & =& \frac{(-1)^{|J|}}{\alpha_{D_{1}} \alpha^{J}} V_{D_{2}}^{D_{2}-D_{J}}(Z_{D_{2}-D_{J}}^{D_{1}-D_{J}}(\Theta_{D_{1}-D_{J}})) \\
& =& \frac{\lambda_{v}(-1)^{|J|}}{\alpha_{v} \alpha_{D_{2}} \alpha^{J}} V_{D_{2}}^{D_{2}-D_{J}}(\Theta_{D_{2}-D_{J}})+\frac{(-1)^{|J^{\prime}|}}{\alpha_{D_{2}} \alpha^{J^{\prime}}} V_{D_{2}}^{D_{2}-D_{J^{\prime}}}(\Theta_{D_{2}-D_{J^{\prime}}}) .
\end{array}
$$
For $J$ containing $v$, write $J^{\prime}=J \backslash\{v\}$. Since $v \in \operatorname{Supp}(D_{1}-D_{J})$, by \eqref{e:vz}, \eqref{e:bdd}, and \eqref{e:v123},
$$
\begin{array}{rcl}
Z_{D_{2}}^{D_{1}}(\frac{(-1)^{|J|}}{\alpha_{D_{1}} \alpha^{J}} V_{D_{1}}^{D_{1}-D_{J}}(\Theta_{D_{1}-D_{J}})) & =& Z_{D_{2}}^{D_{1}}(\frac{(-1)^{|J|}}{\alpha_{D_{1}} \alpha^{J}} V_{D_{1}}^{D_{2}}(V_{D_{2}}^{D_{1}-D_{J}}(\Theta_{D_{1}-D_{J}}))) \\
& = & \frac{-q_{v}(-1)^{|J^{\prime}|}}{\alpha_{v}^{2} \alpha_{D_{2}} \alpha^{J^{\prime}}} V_{D_{2}}^{D_{2}-D_{J^{\prime}}}(\Theta_{D_{2}-D_{J^{\prime}}})).
\end{array}
$$
Thus, we can write
$$
Z_{D_{2}}^{D_{1}}(\tilde{\mathscr{L}}_{D_{1}})=\sum_{\emptyset \subseteq J^{\prime} \subseteq T_{o}^{(2)}} F_{J^{\prime}} \cdot V_{D_{2}}^{D_{2}-D_{J^{\prime}}}(\Theta_{D_{2}-D_{J^{\prime}}}).
$$
Here, if $v \in J^{\prime}$, then the coefficient
$$
\mathrm{F}_{J^{\prime}}=\frac{(-1)^{|J^{\prime}|}}{\alpha_{D_{2}} \alpha^{J^{\prime}}},
$$
while if $v \notin J^{\prime}$, then
$$
\mathrm{F}_{J^{\prime}}=(\frac{\lambda_{v}}{\alpha_{v}}-\frac{q_{v}}{\alpha_{v}^{2}}) \cdot
\frac{(-1)^{|J^{\prime}|}}{\alpha_{D_{2}} \alpha^{J^{\prime}}}=\frac{(-1)^{|J^{\prime}|}}{\alpha_{D_{2}} \alpha^{J^{\prime}}}.
$$
This proves the lemma for the $\ord_{v}(D_1) \geq 2$ case.

Now assume that $\ord_{v}(D_{1})=1$, or equivalently $v \notin \mathrm{Supp}(D_{2})$. 
If $v \in T_{m}$, then the lemma follows from \eqref{e:vz123} and Theorem \ref{t:theta}(b), because we have
$$
\begin{array}{rcl}
Z_{D_{2}}^{D_{1}}(\frac{(-1)^{|J|}}{\alpha_{D} \alpha^{J}} V_{D_{1}}^{D_{1}-D_{J}}(\Theta_{D_{1}-D_{J}})) 
& =& \frac{(-1)^{|J|}}{\alpha_{D_{1}} \alpha^{J}} V_{D_{2}}^{D_{2}-D_{J}}(Z_{D_{2}-D_{J}}^{D_{1}-D_{J}}(\Theta_{D_{1}-D_{J}})) \\
& = & \frac{(-1)^{|J|}}{\alpha_{D_{2}} \alpha^{J}} V_{D_{2}}^{D_{2}-D_{J}}((\lambda_{v}-[v]_{D_{2}}^{-1}) \Theta_{D_{2}-D_{J}}).
\end{array}
$$

Suppose $v \in T_{o}$. For $J$ not containing $v$, we have $J \subset \operatorname{Supp}(D_{2})$, and hence by \eqref{e:v123}, \eqref{e:vz123} and Theorem \ref{t:theta}(a),
$$
\begin{array}{rcl}
Z_{D_{2}}^{D_{1}}(\frac{(-1)^{|J|}}{\alpha_{D_{1}} \alpha^{J}} V_{D_{1}}^{D_{1}-D_{J}}(\Theta_{D_{1}-D_{J}})) & =& \frac{(-1)^{|J|}}{\alpha_{D_{1}} \alpha^{J}} V_{D_{2}}^{D_{2}-D_{J}}(Z_{D_{2}-D_{J}}^{D_{1}-D_{J}}(\Theta_{D_{1}-D_{J}})) \\
& =& \frac{(-1)^{|J|}(\lambda_{v}-[v]_{D_{2}}-[v]_{D_{2}}^{-1})}{\alpha_{v} \alpha_{D_{2}} \alpha^{J}} V_{D_{2}}^{D_{2}-D_{J}}(\Theta_{D_{2}-D_{J}}).
\end{array}
$$
For $J$ containing $v$, put $J^{\prime}=J \backslash\{v\}$. By \eqref{e:vz} and \eqref{e:bdd},
$$
\begin{array}{rcl}
Z_{D_{2}}^{D_{1}}(\frac{(-1)^{|J|}}{\alpha_{D_{1}} \alpha^{J}} V_{D_{1}}^{D_{1}-D_{J}}(\Theta_{D_{1}-D_{J}})) 
& = & \frac{(-1)^{|J|}}{\alpha_{D_{1}} \alpha^{J}} Z_{D_{2}}^{D_{1}}(V_{D_{1}}^{D_{2}}(V_{D_{2}}^{D_{2}-D_{J^{\prime}}}(\Theta_{D_{2}-D_{J^{\prime}}}))) \\
& =& \frac{-(q_{v}-1)(-1)^{|J^{\prime}|}}{\alpha_{v}^{2} \alpha_{D_{2}} \alpha^{J^{\prime}}} V_{D_{2}}^{D_{2}-D_{J^{\prime}}}(\Theta_{D_{2}-D_{J^{\prime}}}).
\end{array}
$$

Then we write
$$
Z_{D_{2}}^{D_{1}}(\tilde{\mathscr{L}}_{D_{1}})=\sum_{\emptyset \subseteq J^{\prime} \subseteq T_{o}^{\prime}} \mathrm{F}_{J^{\prime}} V_{D_{2}}^{D_{2}-D_{J^{\prime}}}(\Theta_{D_{2}-D_{J^{\prime}}}),
$$
where
$$
\begin{array}{rcl}
\mathrm{F}_{J^{\prime}} & = & \frac{(-1)^{|J^{\prime}|}}{\alpha_{D_{2}} \alpha^{J^{\prime}}} 
\cdot (\frac{\lambda_{v}-[v]_{D_{2}}-[v]_{D_{2}}^{-1}}{\alpha_{v}}-\frac{q_{v}-1}{\alpha_v^2}) \\
& =& \frac{(-1)^{|J^{\prime}|}}{\alpha_{D_{2}} \alpha^{J^{\prime}}} \cdot(1+\alpha_{v}^{-2} q_{v}-\alpha_{v}^{-1}[v]_{D_{2}}-\alpha_{v}^{-1}[v]_{D_{2}}^{-1}-\alpha_{v}^{-2} q_{v}+\alpha_{v}^{-2}) \\
& = & \frac{(-1)^{|J^{\prime}|}}{\alpha_{D_{2}} \alpha^{J^{\prime}}}(1-\alpha_{v}^{-1}[v]_{D_{2}})(1-\alpha_{v}^{-1}[v]_{D_{2}}^{-1}).
\end{array}
$$
\end{proof} 
By Lemma \ref{l:thetacomp}, if $\mathrm{Supp}(D_{1})=\mathrm{Supp}(D_{2})$, then $Z_{D_{2}}^{D_{1}}(\tilde{\mathscr{L}}_{D_{1}})=\tilde{\mathscr{L}}_{D_{2}}$, so define
\begin{equation}\label{e:tildeLT}
\tilde{\mathscr{L}}_{A, T}:=\varprojlim_{\mathrm{Supp}(D)=T} \tilde{\mathscr{L}}_{D}
\end{equation}
as an element of $\frac{1}{p^\aleph}\tilde{\Lambda}_{T}\subset \Q_{p} \cdot\tilde{\Lambda}_{T}$.

\subsubsection{The interpolation formula for $\tilde{\mathscr L}_{A,T}$}\label{ss:inttildeL}
For a continuous character $\omega: \tilde{W}_{T} \longrightarrow \mu_{p^\infty}$, define
$$
\Xi_{T, \omega}=\prod_{\substack{v \in T_{m} \\ v \notin \mathrm{Supp}(D_{\omega})}}(\lambda_{v}-\omega([v]_{D_{\omega}})^{-1}) \prod_{\substack{v \in T_{o} \\ v \notin \mathrm{Supp}(D_{\omega})}}(1-\alpha_{v}^{-1} \omega([v]_{D_{\omega}}))(1-\alpha_{v}^{-1} \omega([v]_{D_{\omega}})^{-1}).
$$
It follows that if $T=\emptyset$, then $\alpha_{D_{\omega}}=\tau_{\omega}=\Xi_{T, \omega}=1$.

\begin{lemma}\label{l:tildeL}
For each continuous character $\omega: \tilde{W}_{T} \longrightarrow \mu_{p^\infty}$,
$$
\omega(\tilde{\mathscr{L}}_{A, T})=\alpha_{D_{\omega}}^{-1} \cdot \tau_{\omega} \cdot q^{\frac{\operatorname{deg}(\Delta)}{12}+\kappa-1} \cdot \Xi_{T, \omega} \cdot L_{A}(\omega, 1).
$$
\end{lemma}
\begin{proof}
By Lemma \ref{l:thetacomp}, after applying the natural projection $Z^T_{\mathrm{Supp}(D_\omega)}$ to
\(\tilde\Lambda_{\mathrm{Supp}(D_\omega)}\), one has
\[
Z_{\mathrm{Supp}(D_\omega)}^T(\tilde{\mathscr L}_{A,T})
=
\Xi_{T,\omega}\cdot
\tilde{\mathscr L}_{A,\mathrm{Supp}(D_\omega)}.
\]
It is therefore enough to treat the case
\(\mathrm{Supp}(D_\omega)=T\).
Then by Lemma \ref{l:thetacomp} again, $\omega(\tilde{\mathscr{L}}_{A,T})=\omega(\tilde{\mathscr{L}}_{D_\omega})$. In view of Theorem \ref{t:theta}(e) and the definition \eqref{e:tildeL}, we need to show that $\omega(V^{D_\omega}_{D_\omega-D_J}(\Theta_{D_\omega-DJ}))=0$, for all non-empty $J\subset T_o$. Now, let $G$ denote the kernel of the projection $W_{D_\omega}\longrightarrow W_{D_\omega-D_J}$. Then in 
$\frac{1}{p^\aleph}\Z[[W_{D_\omega}]]$, the element $V^{D_\omega}_{D_\omega-D_J}(\Theta_{D_\omega-D_J})$ is divisible by $\mathsf e:=\sum_{g\in G} g$, while since $D_\omega$ is the conductor of $\omega$, there must be some $g$
in the group $G$ such that $\omega(g)\not=1$, hence $\omega(\mathsf e)=0$.
\end{proof}

\subsection{The $p$-adic $L$-function $\hat{\mathscr L}_{A/L}$}\label{su:hatL}
For an intermediate (not necessary multiple $\Z_p$) extension $L^{\prime} / K$ of $L / K$ with ramification locus $S^{\prime}$, define, as an element of $\Q_p\Lambda'$,
$$
\hat{\mathscr{L}}_{A / L'}= \begin{cases}
p_{L'}(\tilde{\mathscr{L}}_{A, S'}), & \text { if } L' \neq K; \\ 
q^{\frac{\deg(\Delta)}{12}+\kappa-1} \cdot L_{A}(\omega_{0}, 1), & \text { if } L'=K.
\end{cases}
$$
Recall that $p_{L'}$ is the homomorphism $\Q_p\tilde\Lambda_{S'}\longrightarrow \Q_p\Lambda'$ defined in \S\ref{su:notation}.

\subsubsection{The specialization and the interpolation formulae}\label{ss:spinthat}
The specialization formula
\begin{equation}\label{e:spechat}
p_{L^{\prime}}^{L}(\hat{\mathscr{L}}_{A / L})=\prod_{v \in S_{m} \backslash S_{m}^{\prime}}(\lambda_{v}-[v]_{L^{\prime}}^{-1}) \cdot \prod_{v \in S_{o} \backslash S_{o}^{\prime}}(1-\alpha_{v}^{-1}[v]_{L^{\prime}})
(1-\alpha_{v}^{-1}[v]_{L^{\prime}}^{-1}) \cdot \hat{\mathscr{L}}_{A / L^{\prime}}.
\end{equation}
is a direct consequence of Theorem \ref{t:theta}(e) and Lemma \ref{l:thetacomp}. The interpolation formula
\begin{equation}\label{e:inthat}
\omega(\hat{\mathscr{L}}_{A / L})=\alpha_{D_{\omega}}^{-1} \cdot \tau_{\omega} \cdot q^{\frac{\operatorname{deg}(\Delta)}{12}+\kappa-1} \cdot \Xi_{S, \omega} \cdot L_{A}(\omega, 1) 
\end{equation}
follows from Lemma \ref{l:tildeL}.

\subsubsection{The factor $\dag_{A / L}$}\label{ss:gimel}
Recall that $m_v$ denote the order of the components group of $A/K_v$,  and by \eqref{e:cv},
at a non-split multiplicative place $m_v=1$ or $2$ .

Let $S_{1} \subset S$ denote the set consisting of $v \in S$ at which $\Gamma_{v} \simeq \mathbb{Z}_{p}$. For such $v$, we fix a topological generator $\sigma_{v}$ of $\Gamma_{v}$.

\begin{definition}\label{d:gimel} For each place $v$ of $K$, define the element 
$\dag_{A/L, v} \in \Lambda$ as follows:

\begin{enumerate}
\item $\dag_{A / L, v}=1$, if $v \notin S$ and $\Gamma_{v}$ is non-trivial.

\item $\dag_{A / L, v}=m_v$, if $v \notin S$ and $\Gamma_{v}$ is trivial.

\item $\dag_{A / L, v}=\lambda_{v}-\sigma_{v}$, if $v \in S_{1}$ and $v$ is a split multiplicative place, or $v$ is a non-split multiplicative place in $S_{1}$  and $\F_{q_v^2}\subset L_v$.

\item $\dag_{A / L, v}=1$, for other $v \in S$.
\end{enumerate}
Then define
$$
\dag_{A / L}=\prod_{\text {all } v} \dag_{A / L, v}.
$$
\end{definition}

The ideal $(\dag_{A / L, v}) \subset \Lambda$ is independent of the choice of $\sigma_{v}$. 
Also, in (3), $\dag_{A / L, v}=-1-\sigma_{v}$ only occurs when $p=2$. We shall write $\dag_{A / L / K}$ for $\dag_{A / L}$, if necessary.

\begin{lemma}\label{l:thetav}
In $\mathbb{Q}_{p} \Lambda$, the element $\hat{\mathscr{L}}_{A / L}$ is divisible by $\dag_{A / L, v}$.
\end{lemma}
\begin{proof}
We may assume that $v \in S_{1}$ so that $\Gamma_{v}$ is topologically generated by $\sigma_{v}$. Denote $L^{\prime}=L^{\Gamma_{v}}$. Since $v$ splits completely over $L^{\prime}$, we have $1-[v]_{L^{\prime}}=0$. Thus, if $A$ has split multiplicative reduction at $v$, then by \eqref{e:spechat}, $p^L_{L^{\prime}}(\hat{\mathscr{L}}_{A / L})=0$. In $\mathbb{Q}_{p} \Lambda$, the ideal $\ker(p_{L^{\prime}}^{L})=(1-\sigma_{v})=(\dag_{A / L, v})$, so the lemma is proved.

Suppose $A$ has non-split multiplicative reduction at $v$ and $(\dag_{A / L, v})=(-1-\sigma_{v})$. 
Let $L^{\prime \prime}$ denote the fixed field of $\sigma_{v}^{2}$ and write $\Gamma^{\prime \prime}:=\Gal(L^{\prime \prime} / K)$. Definition \ref{d:gimel}(3) says $L^{\prime \prime} / K$ is 
unramified at $v$ and the image of $\sigma_{v}$ under the projection $\Gamma \longrightarrow \Gamma^{\prime \prime}$ is the Frobenius element $[v]_{L^{\prime \prime}}$. 
Identify $p^L_{L''}$ with $\Lambda\longrightarrow \Lambda/(\sigma_v^2-1)$, to see that
$\mathrm q: \Lambda \longrightarrow \Lambda /(\sigma_{v}+1)$ factors through $p_{L^{\prime \prime}}^{L}$.
The specialization formula \eqref{e:spechat} says $\mathrm{q}(\hat{\mathscr{L}}_{A / L})=0$, and hence $\hat{\mathscr{L}}_{A / L} \in(1+\sigma_{v})$.
\end{proof}

\begin{lemma}\label{l:cased}
Suppose $v_{0} \notin S$ is a given place of $K$. There is a $\mathbb{Z}_{p}^{f}$-extension $\tilde{L} / K$, for some $f$, with the Galois group $\tilde{\Gamma}$, satisfying the following:
\begin{enumerate}
\item $\tilde{L} / K$ is unramified outside $S \cup \{v_{0} \}$ and $L \subset \tilde{L}$.
\item For each $v \in S$, $\tilde{\Gamma}_{v} \cap \operatorname{Gal}(\tilde{L} / L)=\{i d\}$, and hence the natural projection $\tilde{\Gamma}_{v} \longrightarrow \Gamma_{v}$ is an isomorphism.
\item The homomorphism $\prod_{v \in S} \tilde{\Gamma}_{v}\longrightarrow \tilde{\Gamma}$ induced by $\tilde{\Gamma}_{v} \longrightarrow \tilde{\Gamma}$ is injective and its image forms a direct summand of $\tilde{\Gamma}$.
\end{enumerate}
\end{lemma}

\begin{proof} For each $v \in S$, write $\N_{v}=\Nm_{L_{v} / K_{v}}(L_{v}^{*})$ so that  by the class field theory $K_{v}^{*} / \N_{v}$ is identified with a dense open subgroup of $\Gamma_{v}$. Put
$$
\N_{S}:=\prod_{v \in S} \N_{v} \times \prod_{v=v_{0}}\{1\} \times \prod_{v \notin S \cup \{v_{0}\}} \mathcal{O}_{v}^{*} \subset \mathbb{A}_{K}^{*} .
$$

Let $Y$ and $Z_{v_{0}}$ denote respectively the pro-$p$ completion of $K^{*} \backslash \mathbb{A}_{K}^{*} / \N_{S}$ and $\mathcal{O}_{v_{0}}^{*}$. Then $Y / Z_{v_{0}}$ is finitely generated over $\mathbb{Z}_{p}$, while $Z_{v_{0}}$ is isomorphic to $\mathbb{Z}_{p}^{\infty}$, the direct product of countably infinite copies of $\mathbb{Z}_{p}$ \cite{we74}. Therefore, we have $Y \simeq \mathbb{Z}_{p}^{\infty} \times T$, where $T$ is a finite $p$-group. Our $\Gamma$ is a quotient of $G:=Y / T \simeq \mathbb{Z}_{p}^{\infty}$. Let $Y_{v}$ denote the topological closure of the image of the composition $K_{v}^{*} \longrightarrow \mathbb{A}_{K}^{*} \longrightarrow Y$ and let $Y_{S} \subset Y$ be the closed subgroup generated by $Y_{v}, v \in S$. Let $G_{v}$ and $G_{S}$ denote respectively the image of $Y_{v}$ and $Y_{S}$ in $G$. Each $G_v$, $v\in S$ is a free $\Z_p$-module of finite rank. 

We claim that if $x \in Y$ and $y_{v} \in Y_{v}$, $v \in S$, are such that $x^{p^m}=\prod_{v \in S} y_{v}$, for some $m$, then for 
each $v$, there exists $z_{v} \in Y_{v}$ such that $y_{v}=z_v^{p^m}$.
The claim implies that $G_{S}$ is the direct product of $G_{v}, v \in S$, and $G/G_S$ is torsion free.
Choose a $\Z_p$-basis $\mathfrak A
=\{\mathfrak v_j\mid j\in \mathfrak a\}$ of $G_S$ so that $G_S=\oplus_{j\in \mathfrak a} \Z_p\cdot \mathfrak{v}_j$, and choose a direct product $\Z_p$-basis
$\mathfrak B=\{\mathfrak{e}_i\mid  i\in\mathfrak b \}$ of $G$ in the sense that every $g\in G$ can be uniquely expressed as the infinite sum: $\sum_{i\in\mathfrak b} a_i\mathfrak e_i$, $a_i\in\Z_p$. 
By substitution method in Linear Algebra, we can find a subset $\mathfrak B'\subset \mathfrak B$, such that  the union
$\mathfrak A\cup\mathfrak B'$ also form a a direct product $\Z_p$-basis of $G$. Let $G'\subset G$ be the $\Z_p$-submodule spanned by $\mathfrak B'$ in the direct product sense, so that
$G=G_{S} \times G^{\prime}$ and $G^{\prime} \simeq \mathbb{Z}_{p}^{\infty}$.

By the class field theory, $G$ is the Galois group of a pro-$p$ abelian extension $M / K$ unramified outside $S \cup\{v_{0}\}$ and $\Gamma$ can be identified with a quotient $G / H$ by some closed subgroup $H$. Set 
$H_{0}=G^{\prime} \cap H$, $\tilde{\Gamma}:=G / H_{0}=G_{S} \times G^{\prime} / H_{0}$, 
and let $\tilde{L}$ denote the fixed field $M^{H_{0}}$. Since $G^{\prime} / H_{0}$ can be embedded as a subgroup of $G / H=\Gamma$, it is finitely generated over $\mathbb{Z}_{p}$. Hence, $\tilde{L} / K$ is a $\mathbb{Z}_{p}^{f}$-extension for some finite $f$. Then the assertion (1) is proved, since $L \subset \tilde{L} \subset M$. The class filed theory also identifies each $G_{v}$ with the decomposition subgroup $\tilde{\Gamma}_{v}$, and hence (3) holds. The assertion (2) holds, because the composition
$$
\xymatrix{\Gamma_{v} \ar[r]^-\sim & \overline{K_{v}^{*} / \N_{v}} \ar[r]^-\sim & Y_{v} \ar[r] &  G_{v} 
\ar[r]^-\sim &  \tilde{\Gamma}_{v} \ar[r] & \Gamma_{v}}
$$
is the identity map.

Now we prove the claim. Let $\alpha=(\alpha_{v})_{v} \in \mathbb{A}_{K}^{*}$ represent $x$ and $
\beta_{v} \in K_{v}^{*}$, $v \in S$, represents $y_{v}$. Then for each positive integer $n$, there exist 
$\gamma^{(n)}=(\gamma_{v}^{(n)})_{v} \in \N_{S}$, $\eta^{(n)} \in K^{*}$, and $\iota^{(n)}=(\iota_{v}^{(n)})_{v} \in(\mathbb{A}_{K}^{*})^{p^{n}}$ such that in $\mathbb{A}_{K}^{*}$,

\begin{equation}\label{e:abc}
\prod_{v \in S} \beta_{v}=\eta^{(n)} \cdot \alpha^{p^m} \cdot \gamma^{(n)} \cdot \iota^{(n)}.
\end{equation}

If $n\geq m$, at $v_{0}$, we have $\eta^{(n)}=\alpha_{v_{0}}^{-p^m} \cdot (\iota_{v_{0}}^{(n)})^{-1} \in (K_{v_{0}}^{*})^{p^m}$. The local Leopoldt conjecture \cite{kis93} implies $\eta^{(n)} \in (K^{*})^{p^m}$. Then \eqref{e:abc} says $\beta_{v} \in (K_{v}^{*})^{p^m} \cdot N_{v}$, for each $v \in S$. Hence $y_{v} \in Y_{v}^{p^m}$.
\end{proof}

\begin{myproposition}\label{p:additional}
The element $\hat{\mathscr{L}}_{A / L}$ is divisible by $\dag_{A / L}$ in $\mathbb{Q}_{p} \Lambda$.
\end{myproposition}

\begin{proof} 
By Definition \ref{d:gimel}, the only non-unit $\dag_{A / L, v}$ are of the form $1-\sigma_{v}$ or $-1-\sigma_{v}$, and the corresponding $v$ are in $S_1$.

Choose a good ordinary place $v_{0}$ of $A / K$ with $v_0\not\in S$ and choose a $\mathbb{Z}_{p}^{f}$-extension $\tilde{L} / K$ satisfying the condition in Lemma \ref{l:cased} so that for each $v \in S$, we can write $\tilde{\Gamma}_{v}=\Gamma_{v}$. Then put $\tilde{\sigma}_{v}=\sigma_{v}$, for $v\in S_1$, and apply Definition \ref{d:gimel} to $\tilde{L} / K$. Because the set $\{\tilde{\sigma}_{v} \mid v \in S_1\}$ is extendable to a $\mathbb{Z}_{p}$-basis of $\tilde{\Gamma}$, for distinct $v, v^{\prime} \in S$, the corresponding $\dag_{A / \tilde{L},v}$ and $\dag_{A / \tilde{L}, v^{\prime}}$ are relatively prime to each other in $\tilde{\Lambda}:=\mathbb{Z}_{p}[[\tilde{\Gamma}]]$. In view of Lemma \ref{l:thetav}, we see that $\hat{\mathscr{L}}_{A / \tilde{L}}$ is divisible by $\dag_{A/\tilde L}$,
and hence $p^{\tilde{L}}_L(\hat{\mathscr{L}}_{A / \tilde{L}})$ is divisible by $\dag_{A / L}$. Now Lemma \ref{l:thetacomp} says
$$
p^{\tilde L}_L(\hat{\mathscr{L}}_{A / \tilde{L}})=(1-\alpha_{v_{0}}^{-1}[v_{0}]_{L / K})^{\epsilon} \cdot(1-\alpha_{v_{0}}^{-1}[v_{0}]_{L / K}^{-1})^{\epsilon} \cdot \hat{\mathscr{L}}_{A / L}
$$
where $\epsilon=1$ or 0 depending on if $\tilde{L} / K$ is ramified at $v_{0}$ or not. To complete the proof, we need to show that each $\dag_{A / L, v}$ is relatively prime to $1-\alpha_{v_{0}}^{-1}[v_{0}]_{L / K}^{ \pm 1}$. The main point is that since $\alpha_{v_{0}}$ is a Weil $q_{v_{0}}$-number, it is not a root of unity, and hence
\begin{equation}\label{e:not0}
\omega(1-\alpha_{v_{0}}^{-1}[v_{0}]_{L / K}^{ \pm 1}) \neq 0, \text { for all continuous character } \omega: \Gamma \longrightarrow \mu_{p^{\infty}}.
\end{equation}
By properly choosing $\sigma_{v}$, we can write $\sigma_{v}=\tau^{p^{m}}$, for some $\tau$ which is extendable to a basis of $\Gamma$. Suppose
$$
\dag_{A / L, v}=1-\sigma_{v}=\prod_{i=0}^{p^{m}-1}(1-\zeta^{i} \tau),
$$
where $\zeta_{i}$ is a generator of $\mu_{p^{m}}$. Write $t=\tau-1$ and $\mathcal{O}=\mathbb{Z}_{p}[\zeta]$. Then in $\mathcal{O}[[t]]$,
$$
1-\zeta^{i} \tau=-\zeta^{i}(t+(1-\zeta^{i}))
$$
is a linear polynomial, and hence irreducible in $\mathcal{O}[[\Gamma]]$. Also, there exists a continuous character $\omega_{i}: \Gamma \longrightarrow \mu_{p^{m}}$ sending $\tau$ to $\zeta^{-i}$ so that 
$\omega_{i}(1-\zeta^{i} \tau)=0$. By \eqref{e:not0}, the element $1-\alpha_{v_{0}}^{-1}[v_{0}]_{L / K}^{ \pm 1}$ is not divisible by any $1-\zeta^{i} \tau$, it is relatively prime to $\dag_{A / L, v}$.

The case $\dag_{A / L, v}=-1-\sigma_{v}$ only occurs when $p=2$, and we have
$$
-1-\sigma_{v}=-\prod_{j}(1-\zeta_{j} \tau)
$$
where $\zeta_{j}$ runs through generators of $\mu_{2^{m+1}}$. A similar argument
also leads to the conclusion that $1-\alpha_{v_{0}}^{-1}[v_{0}]_{L / K}^{ \pm 1}$ is relatively prime to $\dag_{A / L, v}$.
\end{proof}

\subsection{The $p$-adic $L$-function $\mathscr L_{A/L}$}\label{su:L} 
That $\hat{\mathscr L}_{A/L}$ needs to be modified to satisfy the desired specialization formula
results in the following definition of $\mathscr L_{A/L}$.


\begin{definition}\label{d:l}Define $\mathscr{L}_{A / L}=\mathrm{t}_{A / L} \cdot \nabla_{A / L} \cdot \dag_{A / L}^{-1} \cdot \hat{\mathscr{L}}_{A / L}$. If it is needed to refer to $K$ as well, we shall denote $\mathscr{L}_{A / L / K}:=\mathscr{L}_{A / L}$.
\end{definition}

Proposition \ref{p:additional} says that $\hat{\mathscr{L}}_{A / L}$ is divisible by $\dag_{A / L}$ in 
$\mathbb{Q}_{p} \Lambda$. Therefore,

$$
\mathscr{L}_{A / L} \in \mathbb{Q}_{p} \Lambda.
$$
The ideal $(\mathscr L_{A/L})$ is independent of the choices of $\sigma_v$, $v\in S_1$, so is the element
$\dag_{A/L}\cdot \mathscr L_{A/L}$.
 
 \subsection{Mazur-Tate-Teitelbaum conjecture}\label{su:mtt} An $\mathscr H\in\Q_p\Lambda_\Gamma$ is said to have the order of vanishing (at the identity) $r$ if $\mathscr H\in \Q_pI_\Gamma^r$
 and $\mathscr H\not\in \Q_pI_\Gamma^{r+1}$, in this case, the so called leading term means the class of
 $\mathscr H$ in $\Q_pI_\Gamma^{r}/\Q_pI_\Gamma^{r+1}$. The order of vanishing of $0$ is $\infty$.
In this subsection, we discuss the order of vanishing of $\hat{\mathscr L}_{A/L}$ as well as its leading term.   
 
 \subsubsection{The order of vanishing}\label{ss:vanishing}  

 Choose a set of topological generators $\sigma_1,...,\sigma_d$ of $\Gamma$ viewed as a multiplicative group
 and put
 $t_i=\sigma_i-1$, $i=1,...,d$. Then $\Lambda_\Gamma=\Z_p[[t_1,...,t_d]]$ and 
 $I_\Gamma=(t_1,...,t_d)$. 
Every $\mathscr H\in\Q_\cdotp\Lambda_\Gamma$ can be uniquely written as $\sum_{i=0}^\infty \mathscr H_i$, where $\mathscr H_i$ is a $\Q_p$-homogeneous polynomial in $t_1,...,t_d$ of degree $i$, so that the order of vanishing
of $\mathscr H$ equals $r$, if and only if $\mathscr H_i=0$, for all $i<r$ and $\mathscr H_r\not =0$. 
It is clear that $\mathscr H_i=0$ if and only if $\mathscr H_i(a_1,...,a_d)=0$, for all 
 $(a_1,...,a_d)\in \Z_p^d\setminus \{0\}$. 
 
Let $L'/K$ be an intermediate $\Z_p$-extension of $L/K$ and let $\sigma$ be a topological generator of $\Gamma'$ viewed as a multiplicative group. The ring homomorphism $\Q_p\Lambda_\Gamma\longrightarrow\Q_p\Lambda_{\Gamma'}$
sends $\sigma_i$ to some $\sigma^{a_i}$, for some $\in\Z_p$ such that $\mathrm{g.c.d.}(a_1,...,a_d)=1$, and conversely, every such $d$-tuple $(a_1,...,a_d)$ corresponds to a pair $(L',\sigma)$. Furthermore, if $t=\sigma-1$, then 
\begin{equation*}\label{e:hh'll'}
p^L_{L'}(\mathscr H_i)\equiv \mathscr H_i(a_1,...,a_d)\cdot t^i\pmod{I_{\Gamma'}^{i+1}}.
\end{equation*}
Thus, $p^L_{L'}(\mathscr H)\in I_{\Gamma'}^{r}$, if and only if $\mathscr H_i(a_1,...,a_d)=0$, for all $i<r$.
The lemma below is proven by induction on $r$.

 \begin{lemma}\label{l:augment} An $\mathscr H$ belongs to $\Q_p I_\Gamma^{r}$, if and only if
 $p^L_{L'}(\mathscr H)\in \Q_pI_{\Gamma'}^{r}$, for every intermediate $\Z_p$-extension
 $L'/K$ of $L/K$.  
 \end{lemma}

 Write $S_m=S_{sm}\sqcup S_{nm}$, where $S_{sm}$ is the subset consisting of split multiplicative places, and denote
 $s_L:=^\# S_{sm}$. 
 
 Let $L'/K$ be an intermediate $\Z_p$-extension of $L/K$. 
By \eqref{e:spechat}, for some $*\in \Q_p\Lambda_{\Gamma'}$,
\begin{equation*}\label{e:sll'}
p^L_{L'}(\hat{\mathscr L}_{A/L})=*\cdot \prod_{v\in S_{sm}\setminus S'}(1-[v]_{L'}^{-1})\cdot
\hat{\mathscr L}_{A/L'}. 
\end{equation*}

Since $\dag_{A/L'}$ belongs to $\Q_pI_{\Gamma'}^{s_{L'}}$ and
$s_{L'}+^\# (S_{sm}\setminus S')=s_{L}$, the above equation together with Definition
\ref{d:l} implies
\begin{equation}\label{e:hatll'}
 \mathrm{t}_{L'}\cdot \nabla_{A/L'}\cdot  p^L_{L'}(\hat{\mathscr L}_{A/L})=*\cdot \prod_{v\in S_{sm}\setminus S'}(1-[v]_{L'}^{-1})\cdot\dag_{A/L'}\cdot\mathscr L_{A/L'}\in \mathscr L_{A/L'}\cdot \Q_pI_{\Gamma'}^{s_{L}}.
\end{equation}



 \begin{lemma}\label{l:vanishing} The order of vanishing of $\hat{\mathscr L}_{A/L}$ is at least $s_{L}$.
 \end{lemma}
\begin{proof} By \eqref{e:varrho} and Lemma \ref{l:varrho},  $\mathrm{t}_{L'}\cdot \nabla_{A/L'}$
does not belongs to the prime ideal $\Q_pI_{\Gamma'}$,
the lemma follows from Lemma \ref{l:augment} and \eqref{e:hatll'}.

\end{proof}   
 
\subsubsection{A consequence of Iwasawa Main Conjecture}\label{ss:conimc} Next, we investigate the residue class
of $\hat{\mathscr L}_{A/L}$ in $\Q_pI_\Gamma^{s_L}/\Q_pI_{\Gamma}^{s_L+1}$. For $v\in S_{sm}$, let $Q_v$
be the local Tate period and let $\bar Q_v$ denote its image under the local reciprocity law map $K_v^*\longrightarrow \Gamma_v\subset \Gamma$. Put 
$$\mathscr Q_L:=\prod_{v\in S_{sm}} (\bar Q_v-1)\in I_\Gamma^{s_L}.$$

\begin{myproposition}\label{p:mtt}
Assume that the Iwasawa main conjecture holds for every $\Z_p$-extension $L'/K$ unramified outside a finite set of ordinary places.
Then there is $c_L\in\Q_p$ such that
\begin{equation}\label{e:mtt}
\hat{\mathscr L}_{A/L}\equiv c_L\cdot \mathscr Q_L \pmod{\Q_p I_\Gamma^{s_L+1}}.
\end{equation}
Moreover, if the analytic rank of $A/K$ is positive, then $c_L=0$; otherwise, 
\begin{equation}\label{e:cl}
c_L=\frac{|\Sha_K|\cdot \prod_{v\in S_{nm}}(-2m_v)\cdot
\prod_{v\not\in S_m}m_v}{|A(K)|^2} \cdot \prod_{v\in S_o}(1-\alpha_v^{-1})(1-\alpha_v^{-1}).
\end{equation}

\end{myproposition}
\begin{proof}
In view of \eqref{e:spechat}, we may replace $L$ by $K_\infty^{(p)}L$. 
Also, we can replace $L$ by a field $\tilde L$ satisfying the condition in Lemma \ref{l:cased}, because
if the proposition holds for $\tilde L$, the formula \eqref{e:spechat} together with the fact
that for $\gamma\in\Gamma$,
$$1-a\gamma=1-a+a(1-\gamma)\equiv 1-a\pmod{I_\Gamma}$$ 
implies \eqref{e:mtt} holds for $c_L=c_{\tilde L}\cdot (1-\alpha_{v_0}^{-1})^{-1}(1-\alpha_{v_0}^{-1})^{-1}$,
where $v_0$ is the place described in Lemma \ref{l:cased},
and if the analytic rank of $A/K$ is zero, then \eqref{e:cl} also holds.
Thus, for proving the proposition, we may assume that $K_\infty^{(p)}\subset L$
and $\iota:\prod_{v \in S} {\Gamma}_{v}\longrightarrow {\Gamma}$ injective having image a direct summand of ${\Gamma}$.

 For each $v\in S_{sm}$, choose a $\Z_p$-basis $\sigma_{v,1},...,\sigma_{v,d_v}$ of $\Gamma_v$ and write 
 $\bar Q_v-1\equiv \mathscr Q_v\pmod{I_{\Gamma_v}^2}$, where $\mathscr Q_v$ is a homogeneous linear polynomial in $t_{v,i}:=\sigma_{v,i}-1$, $i=1,...,d_v$. Then 
 $$\mathscr Q_L\equiv \prod_{v\in S_{sm}} \mathscr Q_v\pmod{I_\Gamma^{s_L+1}}.$$
 By the nature of $\iota$, we can extend $\{\sigma_{v,i}\;\mid\; v\in S_{sm},
 i=1,...,d_v\}$ to a $\Z_p$-basis $\sigma_1,...,\sigma_d$ of $\Gamma$. For distinct $v,v'\in S_{sm}$,
$\mathscr Q_v,\mathscr Q_{v'}$ are polynomials in distinct variables, and hence relatively prime as elements in
$\Lambda_\Gamma$. For simplicity, denote $\mathscr H=\hat{\mathscr L}_{A/L}$ and let $\mathscr H_i$ be as in \S\ref{ss:vanishing}. In view of Lemma \ref{l:vanishing}, for \eqref{e:mtt}, we need to show that $\mathscr H_{s_L}$
is a multiple of $\mathscr Q_{s_L}:=\prod_{v\in S_{sm}} \mathscr Q_v$ (because both have degree $s_L$), or equivalently, that if $\mathscr Q_{s_L}(a_1,...,a_d)=0$, $a_i\in\Z_p$, $\mathrm{g.c.d.}(a_1,..a_d)=1$, then $\mathscr H_{s_L}(a_1,...,a_d)=0$. As in \S\ref{ss:vanishing}, let $(L',\sigma)$ be the pair corresponding to
$(a_1,...,a_d)$. If $\mathscr Q_{s_L}(a_1,...,a_d)=0$, then $\mathscr Q_v(a_1,...,a_d)=0$,
for some $v\in S_{sm}$, and hence $p^L_{L'}(\bar Q_v)=1$,
by the isomorphism 
$$\Gamma\longrightarrow I_{\Gamma}/I_{\Gamma}^2, \;\gamma\mapsto \gamma-1.$$
Thus, we show that $p^L_{L'}(\hat{\mathscr L}_{A/L})\in \Q_pI_{\Gamma'}^{s_L+1}$, whenever
$p^L_{L'}(\bar Q_v)=1$, for some $v\in S_{sm}$, distinguishing two cases.

If $v$ is ramified over $L'/K$, then $\Gamma_{v}'=\Gamma'\simeq\Z_p$ and $p^L_{L'}(\bar Q_v)=1$ gives $\mathrm{CH}_{\Z_p}(\Gamma_{v}'/\overline{\{Q_v\}})=\mathrm{CH}_{\Z_p}(\Z_p)=0$,
so by Definition \ref{d:vartheta}(c), $\vartheta_{L'/K,v}=(0)$.
By the specialization formula \eqref{e:chspf} and $\varrho_{L'/K}\not=0$ (see \eqref{e:varrho}), $p^{L'}_K(\mathrm{CH}_{\Lambda_{\Gamma'}}(X_{L'}))=0$, and the Iwasawa main conjecture for $A/L'$ gives $p^{L'}_K(\mathscr L_{A/L'})=0$, that is, $\mathscr L_{A/L'}\in\Q_pI_{\Gamma'}$. Then \eqref{e:hatll'} yields $p^L_{L'}(\hat{\mathscr L}_{A/L})\in\Q_pI_{\Gamma'}^{s_L+1}$.

If $v$ is unramified in $L'/K$, then $p^L_{L'}(\bar Q_v)=[v]_{L'}^{m_v}$, so $p^L_{L'}(\bar Q_v)=1$ and the torsion-freeness of $\Gamma'\simeq\Z_p$ force $[v]_{L'}=1$. Since $v$ is split multiplicative, $\lambda_v=1$, so the factor $\lambda_v-[v]_{L'}^{-1}=0$ occurs in \eqref{e:spechat} and $p^L_{L'}(\hat{\mathscr L}_{A/L})=0$.

In either case $\mathscr H_{s_L}(a_1,\dots,a_d)=0$.

It is known that if the analytic rank of $A/K$ is positive, then $X_K$ is non-torsion and
$\mathrm{CH}_{\Z_p}(X_K)=0$; if the analytic rank of $A/K$ is zero, then $X_K$ is torsion and BSD holds for $A/K$ (\cite{Tat66a,mil75}).

If the analytic rank of $A/K$ is positive, then by the Iwasawa main conjecture,  for every intermediate $\Z_p$-extension $L'/K$ of $L/K$, $p^{L'}_K(\mathscr L_{A/L'})=p^{L'}_K(\mathrm{CH}_{\Lambda_{\Gamma'}}(X_{L'}))$, which, by \eqref{e:varrho} and \eqref{e:chspf}, is zero,
so it follows from \eqref{e:hatll'} that $p^L_{L'}(\hat{\mathscr L}_{A/L})\in \Q_pI_{\Gamma'}^{s_L+1}$. Hence,
by Lemma \ref{l:augment}, $\hat{\mathscr L}_{A/L}\in \Q_pI_{\Gamma}^{s_L+1}$ and $c_L=0$.

Suppose the analytic rank of $A/K$ is zero. The BSD formula says (\cite[(3),(9),(46)]{tan95}),
\begin{equation}\label{e:bsd}
L_A(\omega_0,1)=\frac{|\Sha_K|\cdot \prod_{v}m_v}{|A(K)|^2}\cdot q^{-\frac{\deg \Delta}{12}-\kappa+1},
\end{equation}
for the trivial character $\omega_0$.
 Take $L'=K_\infty^{(p)}$ and apply $p^L_{L'}$. Since $p^L_{L'}(\bar Q_v)=[v]_{L'}^{m_v}$, 
$$p^L_{L'}(\bar Q_v-1)\equiv m_v \cdot (1-[v]_{L'}^{-1}) \pmod{I^2_{\Gamma'}},$$
so
\begin{equation}\label{e:scrQ}
\mathscr Q_L\equiv \prod_{v\in S_{sm}} m_v\cdot (1-[v]_{L'}^{-1}) \pmod{I_{\Gamma'}^{s_L+1} }.
\end{equation}
We have  $\hat{\mathscr L}_{A/L'} \equiv \omega_0(\hat{\mathscr L}_{A/L'})\pmod{\Q_pI_{\Gamma'}}$,
so by \eqref{e:spechat}, \eqref{e:inthat} and \eqref{e:bsd},
$$
\begin{array}{rcl}
p^L_{L'}(\hat{\mathscr L}_{A/L})
& \equiv & e_L\cdot \omega_0(\hat{\mathscr L}_{A/L'})\\
 & \equiv & e_L\cdot q^{\frac{\deg \Delta}{12}+\kappa-1}\cdot L_A(\omega_0,1)\\
& \equiv &  e_L \cdot  \frac{|\Sha_K|\cdot \prod_{v}m_v}{|A(K)|^2}\pmod{\Q_pI_{\Gamma'}^{s_L+1}},
\end{array}$$
where $e_L=\prod_{v\in S_{sm}}(1-[v]_{L'}^{-1})\cdot\prod_{v\in S_{nm}} (-2)\cdot \prod_{v\in S_{o}}(1-\alpha_v^{-1})(1-\alpha_v^{-1})$. Hence, by \eqref{e:scrQ},
$$c_L=\prod_{v\in S_{sm}} m_v^{-1}\cdot \prod_{v\in S_{nm}} (-2)\cdot \prod_{v\in S_{o}}(1-\alpha_v^{-1})(1-\alpha_v^{-1})\cdot \frac{|\Sha_K|\cdot \prod_{v}m_v}{|A(K)|^2},$$
as desired.
\end{proof}

The formulae \eqref{e:mtt} and \eqref{e:cl} form a function field version of Mazur-Tate-Teitelbaum conjecture,
it is proven, without assuming the Iwasawa main conjecture, for the case where 
$K$ is a rational function field and $S=S_{sm}=\{v\}$ \cite{hl06}.

 \section{The basic properties of the $p$-adic $L$-function $\mathscr L_{A/L}$}\label{s:padicL} In this section, we demonstrate the basic properties of $\mathscr L_{A/L}$:
the interpolation formula, the functional equation, and the specialization formula. Then we discuss the cases
in which the Iwasawa main conjecture is known to hold.


\subsection{The interpolation formula}\label{su:intL}
The lemma below is a direct consequence of \eqref{e:inthat}.

\begin{lemma}\label{l:intL}
If $\omega \in \hat{\Gamma}$, $\omega(\dag_{A / L}) \neq 0$, then
\begin{equation}\label{e:intL}
\omega(\mathscr{L}_{A / L})=\mathrm{t}_{A / L} \cdot \omega(\nabla_{A / L}) \cdot \omega(\dag_{A / L})^{-1} \cdot \alpha_{D_{\omega}}^{-1} \cdot \tau_{\omega} \cdot q^{\frac{\deg(\Delta)}{12}+\kappa-1} \cdot \Xi_{S, \omega} \cdot L_{A}(\omega, 1) .
\end{equation}
\end{lemma}

If $\omega(\dag_{A / L})=0$, the exact value of $\omega(\mathscr{L}_{A / L})$ is unknown to us. 
However, as explained in \S\ref{ss:intpol},  formula \eqref{e:intL} completely determines $\mathscr{L}_{A / L}$.

\subsection{The functional equation}\label{su:afe}  
Write $N=\sum_{v} n_{v} \cdot v$, so that $n_v=0$ or $1$, for $v$ belonging to $ S$.
For an effective divisor $D$ supported on $S$, put 
$$N_D:=\sum_{v \mid D} n_{v} \cdot v, \quad N'_D:=\sum_{v \nmid D} n_{v} \cdot v.$$

The cases of non-constant curves and constant curves will be separately treated, while Proposition \ref{p:fel} below
holds for both cases.
\subsubsection{Non-constant curves}\label{ss:non} 
Suppose $A / K$ is not a constant elliptic curve. By \cite[Proposition 3]{tan93}, there is $\varepsilon_{N_{D}^{\prime}}= \pm 1$ depending on the divisor $N_{D}^{\prime}$, such that
\begin{equation}\label{e:sharptheta}
\Theta_{D}^{\sharp}=\varepsilon_{N_{D}^{\prime}} \cdot \Theta_{D} \cdot [N_D^{\prime}]_D^{-1}. 
\end{equation}

\begin{lemma}\label{1:sign} If $A / K$ is non-constant and $\mathrm{Supp}(D) \subset S$, then
\begin{equation}\label{e:varep}
\varepsilon_{N_{D}^{\prime}}=\varepsilon_{N} \cdot \prod_{v \mid N_D}(-\lambda_{v}).
\end{equation}
\end{lemma}

\begin{proof}  If $D=0$, then $N_{D}^{\prime}=N$, so $\varepsilon_{N_D^{\prime}}=\varepsilon_{N}$
and hence \eqref{e:varep} holds trivially. Thus, it is sufficient to show that if $D^{\prime}=D+m\cdot v$, $v \in S$, 
$v \notin \mathrm{Supp}(D)$, then $\varepsilon_{N_{D^{\prime}}^{\prime}}=-\lambda_{v} \cdot \varepsilon_{N_D^{\prime}}$. 

For good ordinary $v\in S$, $n_v=0$, so $N_{D'}'=N_D'$ and $\varepsilon_{N_{D'}'}=\varepsilon_{N_D'}$, in agreement with \eqref{e:varep} since $v\nmid N_D$. The argument below treats the multiplicative case, where $\lambda_v=\pm1$.

Theorem \ref{t:theta}(b) together with \eqref{e:sharptheta} implies
$$
\varepsilon_{N_{D^{\prime}}^{\prime}} \cdot Z_{D}^{D^{\prime}}(\Theta_{D^{\prime}} \cdot ([N_{D^{\prime}}]_{D'}^{\prime})^{-1})=Z_{D}^{D^{\prime}}(\Theta_{D^{\prime}})^{\sharp}=(\lambda_{v}-[v]_{D}) \cdot \varepsilon_{N_D^{\prime}} \cdot \Theta_{D} \cdot [N_D^{\prime}]^{-1}_D,
$$
where (because $\lambda_{v}= \pm 1$ ) the right-hand side equals
$$
-\lambda_{v} \cdot \varepsilon_{N_D^{\prime}} \cdot (\lambda_{v}-[v]_D^{-1}) \cdot \Theta_{D} \cdot [N_D^{\prime}]_D^{-1} \cdot[v]_{D},
$$
which is the same as $-\lambda_{v} \cdot \varepsilon_{N_D^{\prime}} \cdot Z_{D}^{D^{\prime}}(\Theta_{D^{\prime}} \cdot [N_{D^{\prime}}^{\prime}]^{-1}_{D'})$. This shows that $\varepsilon_{N_{D^{\prime}}^{\prime}}=-\lambda_{v} \cdot \varepsilon_{N_D^{\prime}}$ holds unless $Z_{D}^{D^{\prime}}(\Theta_{D^{\prime}})=0$. But if $\omega$ is a primitive character of $W_{D}$, then by \eqref{e:taus} and Theorem \ref{t:theta}, for $s$ varies,

$\omega \omega_{s}(Z_{D}^{D^{\prime}}(\Theta_{D^{\prime}}))=(\lambda_{v}-\omega([v]_D)^{-1} \cdot q^{\deg(v)}) \cdot \tau_{\omega} \cdot q^{\frac{\deg(\Delta)}{12}+(\kappa-1)(2 s+1)+s \cdot \deg(D)} \cdot L_{A}(\omega, 1+s)$,
which is a non-zero function of $s$.

\end{proof}

Write $N'_S:=N'_D$, for any $D$ with $\mathrm{Supp}(D)=S$. 
Then we also have

\begin{equation}\label{e:fehatL}
\hat{\mathscr{L}}_{A / L}^{\sharp}=\varepsilon_{N} \cdot(\prod_{v \in N \cap S}(-\lambda_{v})) \cdot [N_{S}^{\prime}]_{L}^{-1} \cdot \hat{\mathscr{L}}_{A / L} .
\end{equation}

\subsubsection{The constant curve case}\label{ss:const} In \cite[(30)]{lltt14a}, if $L\not= K$,
an element $\mathscr{L}_{A, L,S} \in \Lambda$ is defined as the product: $\mathscr{L}_{A, L,S}:=\theta_{A, L, S}^{+} \cdot(\theta_{A, L, S}^{+})^{\sharp}$, where $\theta_{A, L, S}^{+} \in \Lambda$ is a certain modified twisted Stickelberger element. In particular,

\begin{equation}\label{e:lsharp}
\mathscr{L}_{A, L,S}^{\sharp}=\mathscr{L}_{A, L,S}.
\end{equation}

\begin{lemma}\label{l:las}If $A / K$ is a constant elliptic curve and $L\not= K$,
then
$$
\mathscr{L}_{A, L,S}=\alpha^{-2 \kappa+2} \cdot \mathscr{L}_{A / L}.
$$
\end{lemma}
\begin{proof}We need to show that $\omega(\mathscr{L}_{A, L,S})=\alpha^{-2 \kappa+2} \cdot \omega(\mathscr{L}_{A / L})$, for all continuous characters of $\Gamma$. This can be deduced from 
\cite[Theorem 4.7]{lltt14a}, \eqref{e:compare}, and \eqref{e:intL}, as follows (because $\Delta=0$ and $\dag_{A / L}=1$ ):

$$
\begin{array}{rcl}
\omega(\mathscr{L}_{A, L,S}) & =& \omega(b)^{-1} \cdot \tau^{W}(\omega^{-1}) \cdot(q^{1 / 2} \cdot \alpha^{-1})^{2 \kappa-2+\operatorname{deg}(D_{\omega})} \omega(\nabla_{A / L}) \cdot \Xi_{S, \omega} \cdot L_{A}(\omega, 1) \\
& =& \tau_{\omega} \cdot q^{\kappa-1} \alpha^{-(2 \kappa-2+\operatorname{deg}(D_{\omega}))} \cdot \omega(\nabla_{A / L}) \cdot \Xi_{S, \omega} \cdot L_{A}(\omega, 1) \\
& =&\alpha^{-2 \kappa+2} \cdot \omega(\mathscr{L}_{A / L}) .
\end{array}
$$
\end{proof}
Let $S_2 \subset S_{1}$ be the subset consisting of $v$ such that 
$\dag_{A / L, v}=\lambda_v-\sigma_{v}$. 

\begin{myproposition}\label{p:fel} If $A / K$ is an elliptic curve, then
\begin{equation}\label{e:fel}
\mathscr{L}_{A / L}^{\sharp}=\varepsilon \cdot(\prod_{v \in N \cap S \backslash S_2}(-\lambda_{v})) \cdot(\prod_{v \in N \cap S_2} \sigma_{v}^{-1}) \cdot[N_{S}^{\prime}]_{L}^{-1} \cdot \mathscr{L}_{A / L}
\end{equation}
holds for some constant $\varepsilon= \pm 1$, while if $A / K$ is a constant curve, then $\varepsilon=1$.
Thus, as principal fractional ideals of $\Lambda$,
$$\mathscr{L}_{A / L}^{\sharp}\cdot\Lambda=\mathscr L_{A/L}\cdot\Lambda.$$

\end{myproposition}

\begin{proof} We may assume that $L \neq K$. If $v\in S_2$, then $\dag_{A / L, v}=\lambda_v-\sigma_{v}$, 
$\lambda_{v}=\pm 1$, so 

$$
\dag_{A / L, v}^{\sharp}=\lambda_v-\sigma_{v}^{-1}=-\lambda_{v} \cdot \sigma_{v}^{-1} \cdot \dag_{A / L, v}.
$$
If $A / K$ is non-constant, then $\nabla_{A / L}=1$ and the proposition can be deduced from Definition \ref{d:l}
and \eqref{e:fehatL}. If $A / K$ is constant, the proposition follows from \eqref{e:lsharp} and Lemma \ref{l:las}.
Since the fudge factor in \eqref{e:fel} is a unit in $\Lambda$, the last assertion follows.
\end{proof}

\subsection{The specialization formula}\label{su:spf} Let $L^{\prime} / K$ be an intermediate $\mathbb{Z}_{p}^{e}$-extension of $L / K$, $d>e \geq 0$. 
Denote 
$\Psi:=\operatorname{Gal}(L / L^{\prime})$.

Let $\omega \in \hat{\Gamma}^{\prime} \subset \hat{\Gamma}$. Since $\alpha_{D_{\omega}}$ and $\tau_{\omega}$ depend solely on $\omega$, by \eqref{e:intL},
if $\omega(\dag_{A / L}) \neq 0$ and $\omega(\dag_{A / L'}) \neq 0$ , then
\begin{equation}\label{e:42}
\omega(\eth_{L/L'}) \cdot \omega(\mathscr{L}_{A / L})=\omega(\pounds_{L / L^{\prime}}) \cdot \omega(\mathscr{L}_{A / L^{\prime}}),
\end{equation}
where
$$\eth_{L/L'}=\mathrm{t}_{A / L}^{-1} \cdot \mathrm{t}_{A / L^{\prime}} \cdot
p_{L^{\prime}}^{L}(\nabla_{A / L}^{-1}) \cdot  \nabla_{A / L^{\prime}},$$
$$
p_{L^{\prime}}^{L}(\dag_{A / L}) \cdot\pounds_{L / L^{\prime}}= \dag_{A / L^{\prime}} \cdot \prod_{v \in S_{m} \backslash S_{m}^{\prime}}(\lambda_{v}-[v]^{-1}_{L^{\prime}}) \cdot \prod_{v \in S_o \backslash S_o^{\prime}}(1-\alpha_{v}^{-1}[v]_{L^{\prime}})(1-\alpha_{v}^{-1}[v]_{L^{\prime}}^{-1})).
$$

Lemma \ref{l:varrho} says
\begin{equation}\label{e:eth}
\varrho_{L/L'}=\eth_{L/L'}\cdot \Lambda',
\end{equation}
and a direct case by case checking (see Appendix) shows that if $p_{L^{\prime}}^{L}(\dag_{A / L}) \neq 0$, then
\begin{equation}\label{e:beth}
\vartheta_{L / L^{\prime}}=\pounds_{L / L^{\prime}}\cdot\Lambda'.
\end{equation}


 \begin{myproposition}\label{p:spl} If $p_{L^{\prime}}^{L}(\dag_{A / L}) \neq 0$, then 
 $$\eth_{L/L'}\cdot p^L_{L'}(\mathscr L_{A/L})=\pounds_{L/L'}\cdot \mathscr L_{A/L'}.$$
Hence,  as principal fractional  ideals of $\Lambda^{\prime}$,
\begin{equation}\label{e:lspf}
p_{L^{\prime}}^{L}(\mathscr{L}_{A / L}) \cdot \varrho_{L / L^{\prime}}=\mathscr{L}_{A / L^{\prime}} \cdot \vartheta_{L / L^{\prime}}.
\end{equation}
\end{myproposition}

\begin{proof} In view of \eqref{e:eth} and \eqref{e:beth}, it is sufficient to show 
$$\mathscr D:= \eth_{L/L'}\cdot p_{L^{\prime}}^{L}(\mathscr{L}_{A / L})- \pounds_{L/L'}\cdot \mathscr{L}_{A / L^{\prime}}=0.$$ 
By \eqref{e:42}, $\omega(p^L_{L'}(\dag_{A/L})\cdot \dag_{A/L'}\cdot \mathscr D)=0$, for all $\omega\in\hat\Gamma'$, so
$p^L_{L'}(\dag_{A/L})\cdot \dag_{A/L'}\cdot \mathscr D=0$. This implies $\mathscr D=0$, since $\dag_{A/L'}\not=0$ and $p^L_{L'}(\dag_{A/L})$ is assumed to be non-zero.
\end{proof}

It is worthwhile to mention that if \eqref{e:lspf} holds unconditionally, then Proposition \ref{p:mtt}
can be proven without
assuming the Iwasawa main conjecture. Also, \eqref{e:beth} can be used to prove the second equality of Lemma
\ref{l:spvartheta}.

\subsubsection{Direct application of the specialization formulae}\label{ss:pfimcsp} The following lemma is a direct consequence of the specialization formulae. Recall that $\varrho_{L/L'}\not=0$, by \eqref{e:varrho}.

\begin{lemma}\label{l:1} Suppose $p_{L^{\prime}}^{L}(\dag_{A / L}) \cdot \vartheta_{L / L^{\prime}}\neq 0$.
If \eqref{e:imc} holds for $A/L$, then \eqref{e:imc} also holds for $ A/L'$.
\end{lemma}
The next lemma is for the opposite direction.
\begin{lemma}\label{l:imcsp}Suppose $\mathscr L_{A/L}\in\Lambda$ and $p_{L^{\prime}}^{L}(\dag_{A / L}) \cdot \vartheta_{L / L^{\prime}}\neq 0$.
If $X_{L'}$ is torsion, the Iwasawa main conjecture holds for $A/L'$, and there are elements $\varepsilon_{1}, \varepsilon_{2} \in \Lambda$, at least one of them is a unit, such that in $\Lambda$,
\begin{equation}\label{e:varepp}
(\varepsilon_{1}) \cdot \mathrm{CH}_{\Lambda}(X_L)=(\varepsilon_{2}) \cdot
(\mathscr{L}_{A / L}),
\end{equation}
then the Iwasawa main conjecture holds for $A / L$.
\end{lemma} 
\begin{proof}
Apply $p_{L'}^L$ and use the simple fact that an $\varepsilon\in\Lambda$ is a unit if and only if $p^L_{L'}(\varepsilon)$ is
an unit of $\Lambda_{\Gamma'}$. Write $\bar{\varepsilon}_{1}$, $\bar{\varepsilon}_{2}$ for the image of $\varepsilon_{1}, \varepsilon_{2}$ under $p^L_{L'}$. In view of \eqref{e:varrho}, \eqref{e:lspf} and \eqref{e:chspf}, choose a generator $\eta \in \Lambda_{\Gamma'}$ of $\mathrm{CH}_{\Lambda_{\Gamma'}}(X_{L'})$ such that
$$
\bar{\varepsilon}_{1} \cdot \eta=\bar{\varepsilon}_{2} \cdot \mathscr{L}_{A /L'} .
$$
Here, at least one of $\bar{\varepsilon}_{1}, \bar{\varepsilon}_{2}$ is a unit in $\Lambda_{\Gamma'}$. The assumption of the Lemma says both $\eta$ and $\mathscr{L}_{A / L'}$ are non-zero and for some unit $\iota$ in $\Lambda_{\Gamma'}$,
$$
\eta=\iota \cdot \mathscr{L}_{A / L'}.
$$
These imply that both $\bar{\varepsilon}_{1}$ and $\bar{\varepsilon}_{2}$ are units in $\Lambda_{\Gamma'}$, hence both $\varepsilon_{1}$ and $\varepsilon_{2}$ are units in $\Lambda$.
\end{proof}

To have an effective application of the specialization formulae, 
we have to ensure that $p_{L^{\prime}}^{L}(\dag_{A / L}) \cdot\vartheta_{L / L^{\prime}}\neq 0$.
Recall that for a split multiplicative place $v$ of $A/K$, $Q_v$ denotes the local Tate period.

\begin{lemma}\label{l:zero} We have $\vartheta_{L / L^{\prime}}=0$ if and only if there is a split multiplicative $v \in S$ such that $\Gamma_{v}^{\prime}=0, \Gamma_{v} \simeq \mathbb{Z}_{p}$ and $\overline{\{Q_{v}\}}=0$, or $\Gamma_{v}^{\prime}=0$ and $\Gamma_{v} \simeq \mathbb{Z}_{p}^{f}, f \geq 2$.
\end{lemma}

\begin{proof}
The only other case needs checking is Definition \ref{d:vartheta}(b). In such case,  since $\alpha_{v}$, being a Weil $q_v$-number, is not a root of unity, we have $\vartheta_{L / L^{\prime}, v} \neq 0$.
\end{proof}

\begin{corollary}\label{c:0} $p^L_{L'}(\dag_{A/L})\cdot\vartheta_{L/L'}\not=0$, if and only if
$\Gamma_v'\not=0$, for all split multiplicative $v \in S$.
\end{corollary}
\begin{proof} 
At a split multiplicative $v\in S$, if $\Gamma_v=\Z_p$ and $\Gamma'_v=0$, then 
$p^L_{L'}(\dag_{A/L,v})=0$, no matter $\overline{\{Q_{v}\}}=0$ or not; if instead $\mathrm{rank}_{\Z_p}\Gamma_v\geq2$ and $\Gamma_v'=0$, then $p^L_{L'}(\dag_{A/L,v})=1$, while $\Psi_v=\Gamma_v\simeq\Z_p^f$, $f\geq2$, so Definition \ref{d:vartheta}(c) gives $\vartheta_{L/L',v}=0$; in both situations the product $p^L_{L'}(\dag_{A/L})\cdot\vartheta_{L/L'}$ vanishes. Also, for non-split multiplicative $v\in S$, we have $p^L_{L'}(\dag_{A/L,v})\not=0$, because $p^L_{L'}(\sigma_v)$ is either trivial or of infinite order.

\end{proof}
\begin{corollary}\label{c:1}
If $L'=K_\infty^{(p)}$, or $L=K_\infty^{(p)}L'$, then $p^L_{L'}(\dag_{A/L})\cdot\vartheta_{L/L'}\not=0$.
\end{corollary}

\subsection{The restriction}\label{su:rest}
At first we show that for each $f\in\Lambda_\Gamma$,  
$$f_{\Phi}^{\Gamma} \in \Lambda_{\Phi}.$$ 
To see this, put $\O:=\mathbb{Z}_{p}[\chi(\Gamma)]$ and write $f_{\Phi}^{\Gamma}=\sum_{\sigma \in C} f_{\Phi, \sigma} \cdot \sigma$, where each $f_{\Phi, \sigma} \in \O[[\Phi]]$. 
Since $(f_{\Phi}^{\Gamma})_{\chi}=f_{\Phi}^{\Gamma}$, by
summing over $\chi\in \widehat{\Gamma / \Phi}$, we obtain
$$|\Gamma / \Phi| \cdot f_{\Phi}^{\Gamma}=\sum_{\sigma \in C} \sum_{\chi \in \widehat{\Gamma / \Phi}} f_{\Phi, \sigma} \cdot \chi(\sigma) \cdot \sigma.
$$
Denote $\{\sigma_{0}\}=C \cap \Phi$. Since for $\sigma \in C$, the sum $\sum_{\chi \in \widehat{\Gamma / \Phi}} \chi(\sigma)=|\Gamma / \Phi|$, or 0 , depending on $\sigma=\sigma_{0}$ or not, 
the above equality implies $f_{\Phi}^{\Gamma}=f_{\Phi, \sigma_{0}} \in \O[[\Phi]]$. Because $f_{\Phi}^{\Gamma}$ is fixed by the action of $\Gal(\Q_{p}(\chi(\Gamma)) / \Q_{p})$, it is in $\Lambda_{\Phi}$. 

\begin{lemma}\label{l:open} Let $W$ be a $\Lambda_{\Gamma}$-module. The following holds:
\begin{enumerate}
\item $W$ is finitely generated (resp. torsion, pseudo-null) over $\Lambda_{\Phi}$ if and only if it is finitely generated (resp. torsion, pseudo-null) over $\Lambda_{\Gamma}$.

\item If $W=\Lambda_{\Gamma} /(f)$, then $\mathrm{CH}_{\Lambda_{\Phi}}(W)=(f)_{\Phi}^{\Gamma}$.

\item If $\mathrm{CH}_{\Lambda_{\Gamma}}(W)=(f)$, then $\mathrm{CH}_{\Lambda_{\Phi}}(W)=(f)_{\Phi}^{\Gamma}$.
\end{enumerate}
\end{lemma}

\begin{proof} For (1), one direction is obvious.
The fact that $\Lambda_{\Gamma}=\oplus_{\sigma \in C} \Lambda_{\Phi} \cdot \sigma$ implies that
if $W$ is finitely generated over $\Lambda_{\Gamma}$, then it is finitely generated over $\Lambda_{\Phi}$. Also, since $f$ divides $f_{\Phi}^{\Gamma}$, if $f \cdot W=0$, then $f_{\Phi}^{\Gamma} \cdot W=0$. 

Suppose $W$ is pseudo-null over $\Lambda_{\Gamma}$.
Then by \cite[Lemma 2]{gr78} and the discussion after it, there is a closed subgroup $\Upsilon \subset \Gamma$ such that $\Gamma / \Upsilon \simeq \Z_{p}$, with the property that $W$ is finitely generated over $\Lambda_{\Upsilon}$ and is annihilated by some non-zero $g \in \Lambda_{\Upsilon}$. 
Denote $\Phi^{\prime}:=\Phi \cap \Upsilon$. Then $g_{\Phi^{\prime}}^{\Upsilon} \in \Lambda_{\Phi^{\prime}}$ annihilates $W$. Since $\xymatrix{\Phi / \Phi^{\prime} \ar@{^(->}[r] & \Gamma / \Upsilon}$, we have $\Phi=\Phi^{\prime} \oplus \Phi^{\prime \prime}$, $\Phi^{\prime \prime} \simeq \mathbb{Z}_{p}$. Let $\tau \in \Phi^{\prime \prime}$ be a topological generator, and write $\mathsf{t}:=\tau-1$. Then $\Lambda_{\Phi}=\Lambda_{\Phi^{\prime}}[[\mathsf t]]$. 

Now that $W$ is finitely generated over $\Lambda_{\Phi^{\prime}}$, let
$\mathsf{e}_{1}, \ldots, \mathsf{e}_{n}$ be $\Lambda_{\Phi^{\prime}}$ generators of $W$. 
The action of $\mathsf{t}$ on $W$ gives rise to the characteristic polynomial,
a degree $n$ monic polynomial,  
$\mathsf{p}(\mathsf{t}) \in \Lambda_{\Phi^{\prime}}[\mathsf{t}]$, that also annihilates $W$. Obviously, $\mathsf{p}(\mathsf{t})$ and $g_{\Phi^{\prime}}^{\Upsilon}$ are relatively prime.

To show (2), we can work on $\O[[\Gamma]]$ and $\O[[\Phi]]$ instead of $\Lambda_{\Gamma}$ and $\Lambda_{\Phi}$. For $f\in\O[[\Gamma]]$, define $f^{\Gamma}_{\Phi}$ in a similar way,
for simplicity, write $[f]_{\Phi}^{\Gamma}$ for $\mathrm{CH}_{\O[[\Phi]]}(\O[[\Gamma]] /(f))$. We shall show that $[f]_{\Phi}^{\Gamma}=(f)_{\Phi}^{\Gamma}$, in $\O[[\Phi]]$. If $f, g \in \O[[\Gamma]]$, then the exact sequence
$$
\xymatrix{0 \ar[r] &  \O[[\Gamma]] /(g) \ar[r]^-{\cdot f} & \O[\Gamma]] /(f g) \ar[r] &
\O[[\Gamma]] /(f) \ar[r] & 0}
$$
shows that $[f g]_{\Phi}^{\Gamma}=[f]_{\Phi}^{\Gamma} \cdot[g]_{\Phi}^{\Gamma}$. For every $\chi \in \widehat{\Gamma / \Phi}$, the map $\chi^{*}: \O[[\Gamma]] \longrightarrow \O[[\Gamma]]$, that sends $\xi$ to $\xi_{\chi}$, is an $\O[[\Phi]]$-isomorphism. Thus it induces the equality
$$
[f_{\chi}]_{\Phi}^{\Gamma}=[f]_{\Phi}^{\Gamma}.
$$
Consequently, 
$$[f^{\Gamma}_{\Phi}]^{\Gamma}_{\Phi}=\prod_{\chi \in \widehat{\Gamma / \Phi}}[f_{\chi}]_{\Phi}^{\Gamma}=([f]_{\Phi}^{\Gamma})^{|\Gamma / \Phi|}.$$
If $g\in\O[[\Phi]]$, then the decomposition
$
\O[[\Gamma]] /(g)=\bigoplus_{\sigma \in C}(\O[[\Phi]] / (g))  \cdot \sigma
$
shows that $[g]^{\Gamma}_{\Phi}=g^{|\Gamma/\Phi|}$, so in particular, $[f^{\Gamma}_{\Phi}]^{\Gamma}_{\Phi}=(f^{\Gamma}_{\Phi})^{|\Gamma/\Phi|}$.
These imply $([f]_{\Phi}^{\Gamma})^{|\Gamma / \Phi|}=(f^{\Gamma}_{\Phi})^{|\Gamma/\Phi|}$, and hence (2) follows.

To show (3), consider the exact sequence
$$
\xymatrix{0 \ar[r] &  \bigoplus_{j=1}^{\nu} \Lambda_{\Gamma} /(f_{j})\ar[r] & W \ar[r] & U \ar[r] & 0,}
$$
where $U$ is a pseudo-null $\Lambda_{\Gamma}$-module.

\end{proof}

It is interesting to see that if $\Psi \subset \Phi$ is an open subgroup, 
then it follows directly from the definition that for $f \in \Lambda_{\Gamma}$
\begin{equation}\label{e:trans}
f_{\Psi}^{\Gamma}=(f_{\Phi}^{\Gamma})_{\Psi}^{\Phi}.
\end{equation}

Denote ${K}^{\prime}:={L}^{\Phi}$.

\begin{myproposition}\label{p:restriction} Let the notation be as the above. Then as principal fractional  ideals 
\begin{equation}\label{e:xlphi}
\mathrm{CH}_{\Lambda_{\Phi}}(X_{ L})=\mathrm{CH}_{\Lambda_{\Gamma}}(X_{L})_{\Phi}^{\Gamma},
\end{equation}
and
\begin{equation}\label{e:lphi}
(\mathscr{L}_{A / {L} / {K}^{\prime}})=(\mathscr{L}_{A / {L} / {K}})_{\Phi}^{\Gamma}.
\end{equation}
\end{myproposition}

\begin{proof} The formula \eqref{e:xlphi} is a direct consequence of Lemma \ref{l:open}(3), while \eqref{e:lphi} follows from various product formulae. To show it, by finding a sequence ${K}_{1}={K} \subset {K}_{1} \subset \cdots \subset {K}_{n}={K}^{\prime}$ of fields such that each ${K}_{i+1} / {K}_{i}$ is a cyclic extension of degree $p$ and by applying \eqref{e:trans}, we may assume that $\Gamma / \Phi$ is actually cyclic of degree $p$.

Let $\omega \in \hat{\Phi}$ and let $\tilde{\omega} \in \hat{\Gamma}$ be an extension of $\omega$ in the sense that $\tilde{\omega}\mid_{\Phi}=\omega$. Every character of the form $\tilde{\omega} \cdot \chi$, $\chi \in \widehat{\Gamma / \Phi}$, is an extension of $\omega$ and vice versa. Write $K'_\omega$ for $L^{\ker(\omega)}$. The first product formula is
\begin{equation}\label{e:decomplas}
L_{A/ K'}(\omega, s)=\prod_{\chi \in \widehat{\Gamma / \Phi}} L_{A/K}(\tilde{\omega} \cdot \chi, s).
\end{equation}
As usual, this can be checked locally:
\begin{equation}\label{e:locallr}
\prod_{v',\,v'\mid_{ K}=v} L_{A/K', v'}(\omega, s)=\prod_{\chi \in \widehat{\Gamma / \Phi}} L_{A/K, v}(\tilde{\omega} \cdot \chi, s).
\end{equation}

If $\omega$ is ramified at a place $v'$ sitting over $v$, then it is ramified at all places sitting over $v$, 
so the left-hand side equals 1. In this case, $\tilde{\omega} \cdot \chi$ is ramified at $v$, for every $\chi$, and hence the right-hand side also equals 1. 

Suppose $\omega$ is unramified at $v'$ and $K'/K$ is unramified at $v$. Then
all extensions $\tilde{\omega}$ are unramified at $v$. If $v$ splits completely over ${K}^{\prime}$, then every factor on both side of \eqref{e:locallr} equals 
$L_{A/K, v}(\tilde{\omega}, s)$, so the formula holds. If there is only one $v'$ sitting over $v$, then $q_{v'}=q_{v}^{p}$, $\omega([v']_{D_{\omega}})=\tilde{\omega}([v]_{D_{\tilde{\omega}}})^{p}$.
In this situation, if $v$ is good ordinary, then $\alpha_{v'}=\alpha_{v}^{p}$; if $v$ is multiplicative but $p\not= 2$, then $\lambda_{v'}=\lambda_v^p$, so  in both cases $L_{A/K', v'}(\omega, s)$ equals the product on the right-hand side of \eqref{e:locallr}. If $p=2$ and $v$ is multiplicative, then $v'$ is split multiplicative, and \eqref{e:locallr} reads
$$
1-\omega([v']_{D_{\omega}} q_{v'}^{-s})=(1+\tilde{\omega}([v]_{D_{\tilde{\omega}}}) q_{v}^{-s})(1-\tilde{\omega}([v]_{D_{\tilde{\omega}}}) q_{v}^{-s}).
$$

Suppose $\omega$ is unramified at $v'$ while $K'/K$ is ramified at $v$.
Then $\Gal(K'_\omega/K)$ is the direct product of $\Gal(K'/K) $ and $\Gal(K'_\omega/K')$. There is a unique unramified extension $\tilde\omega$ of $\omega$ and there is a unique place $v'$ of $K'$ sitting over $v$.
Since $[v]_{D_{\tilde{\omega}}}=[v']_{D_{\omega}}$, $q_{v'}=q_{v}$, \eqref{e:locallr} reads $L_{A/K', v'}(\omega, s)=L_{A/K, v}(\tilde{\omega}, s)$. This ends the proof of \eqref{e:decomplas}.

Arguments similar to the above lead to the following three product formulae:
\begin{equation}\label{e:nablakk'}
\omega(\nabla_{A / {L} / {K}^{\prime}})=\prod_{\chi \in \widehat{\Gamma / \Phi}}(\tilde{\omega} \cdot \chi)(\nabla_{A / {L} / {K}});
\end{equation}

\begin{equation}\label{e:xikk'}
\Xi_{T, \omega}=\prod_{\chi \in \widehat{\Gamma / \Phi}} \Xi_{S, \tilde{\omega} \cdot \chi};
\end{equation}
\begin{equation}\label{e:gimelkk'}
\omega(\dag_{A / {L} / {K}^{\prime}})=\prod_{\chi \in \widehat{\Gamma / \Phi}}(\tilde{\omega} \cdot \chi)(\dag_{A /{L} / {K}}).
\end{equation}
Note that if $\Phi_{v'}\simeq \Z_p$ then $\Gamma_v\simeq \Z_p$, and if furthermore, the inertia subgroup $\Phi^1_{v'}\subsetneq \Phi_{v'}$, then $K'/K$ is unramified at $v$. In particular, if $L={K'}_\infty^{(p)}$, then $L=K_\infty^{(p)}$; if $v'\in T_1$, where $T$ denotes the ramification locus of ${L} /{K'}$, then $v\in S_1$.
Also, if $\Phi_{v'}=0$, then $\Gamma_v=0$.
In such case, $m_{v'}=m_v$ occurs in of both sides of \eqref{e:gimelkk'} with multiplicity $p$.

Let $q_{K}$ and $\kappa_{K}$ (resp. $q_{K'}$ and $\kappa_{K'}$) denote the order of the constant field and the genus of $K$ (resp. $K'$). 
The next product formula is about the Gauss sums, namely,
\begin{equation}\label{e:taudisc}
\tau_{\omega} \cdot q_{{K}^{\prime}}^{\kappa_{{K}^{\prime}}-1}=\prod_{\chi \in \widehat{\Gamma / \Phi}}(\tau_{\tilde{\omega} \cdot \chi} \cdot q_{{K}}^{\kappa_{{K}}-1}).
\end{equation}
Recall the classical $L$-function associated to each $\psi\in\hat\Gamma$ :
$$
L_{{K}}(\psi, s):=\prod_{v \notin \mathrm{Supp}(D_\psi)} \frac{1}{1-\psi([v]_{D_{\psi}}) q_{v}^{-s}},
$$
as well as the complete Dedekind zeta function
$$\zeta_K(s)=\prod_{\text{all } v } \frac{1}{1-q_{v}^{-s}}.$$
By \cite[VII, \S 6, Theorem 4]{we74}, we have the functional equation
$$\zeta_K(s)=q_K^{\kappa_K-1}\cdot q_K^{-s(2\kappa_K-2)}\cdot\zeta_K(1-s).$$
If $\psi$ is obtained from a principal quasi-character in the sense of \cite[VII, \S 3]{we74},
then $\psi$ is unramified everywhere and there is an $s_0\in\C$, such that $\psi([v]_{D_{\psi}})=q_v^{-s_0}$, at every place $v$. Moreover, in this case,
$$D_{\psi}=0,\;\;\tau_\psi=q_K^{s_0(2\kappa_K-2)},\;\;\text{and}\;\; L_K(\psi^{-1},s)=\zeta_K(s-s_0),$$
so the above functional equation of $\zeta_K$ implies
\begin{equation}\label{e:funeqclass}
L_K(\psi^{-1}, s)=\tau_{\psi} \cdot q_K^{\kappa_{K}-1} \cdot q_K^{-s(2 \kappa_K-2+\deg D_\psi)} \cdot L_K(\psi, 1-s).
\end{equation}
Actually, in view of \eqref{e:compare} and \cite[VII, \S 8,Theorem 6]{we74}, Equation \eqref{e:funeqclass} also
holds for all non-principal $\psi$. In particular, we have
$$L_{{K}^{\prime}}(\omega^{-1}, s)=\tau_{\omega} \cdot q_{K'}^{\kappa_{{K}^{\prime}}-1} \cdot q_{K'}^{-s(2 \kappa_{{K}^{\prime}}-2+\deg D_{\omega})} \cdot L_{{K}^{\prime}}(\omega, 1-s),
$$
and for each $\chi \in \widehat{\Gamma /\Phi}$,
$$
L_{{K}}((\tilde{\omega} \cdot \chi)^{-1}, s)=\tau_{\omega \cdot \chi} \cdot q_{K}^{\kappa_{{K}}-1} \cdot q_{K}^{-s(2 \kappa_{{K}}-2+\deg D_{\tilde{\omega}} \cdot \chi)} \cdot L_{{K}}(\tilde{\omega} \cdot \chi, 1-s).
$$
These together with the classical product formulae
$$L_{K'}(\omega^{-1}, s)=\prod_{\chi \in \widehat{\Gamma / \Phi}} L_{{K}}(\tilde{\omega}^{-1} \cdot \chi, s),\quad L_{K'}(\omega, 1-s)=\prod_{\chi \in \widehat{\Gamma / \Phi}} L_{{K}}(\tilde{\omega} \cdot \chi,1-s)$$
imply \eqref{e:taudisc}.
Also, since $ K'/ K$ is ramified only at ordinary places, 
\begin{equation}\label{e:qk'k}
q_{ K'}^{\deg \Delta_{A/ K'}}=q_{ K}^{p\cdot \deg \Delta_{A/ K}}=\prod_{\chi \in\widehat{ \Gamma /\Phi}}q_{ K}^{\deg\Delta_{A/ K}}.
\end{equation}

Finally, we prove the last product formula:
\begin{equation}\label{e:alphad}
\alpha_{D_{\omega}}=\varepsilon \cdot \alpha^{-1}_{\mathrm{Disc }_{ K^{\prime} /K}} \cdot \prod_{\chi \in \widehat{\Gamma /\Phi}} \alpha_{D_{\tilde{\omega}} \cdot \chi}.
\end{equation}
Here $\mathrm{Disc}_{ F^{\prime} /F}$ denotes the discriminant (as a divisor of $ F$) of the field extension $F^{\prime} / F$ and
$$
\varepsilon=(-1)^{p \cdot \#\{v\mid \mathrm{Disc}_{ K^{\prime} /K} \text { and is non-split multiplicative}\}}.
$$
For this purpose, we first recall that locally at each place $v$ of ${K}$,
\begin{equation}\label{e:ddiscd}
\sum_{v',v'\mid_{{K}=v}} \log _{q_{v}} q_{v'} \cdot \ord_{v'} D_{\omega}=-\ord_{v} \mathrm{Disc}_{{K}^{\prime} / {K}}+\sum_{\chi \in \widehat{\Gamma / \Phi}} \ord_{v} D_{\tilde{\omega} \cdot \chi}.
\end{equation}
To see this, we may assume that $\omega$ is of order $p^n$, $n>0$, because for trivial $\omega$, \eqref{e:ddiscd} is nothing but the conductor discriminant formula.


Since $D_{\omega^{i}}=D_{\omega}$, for $i$ prime to $p$, by the conductor discriminant formula, one deduces
\begin{equation}\label{e:disck'omegak'}
\mathrm{Disc}_{ K_{\omega}^{\prime} /K^{\prime}}=\mathrm{Disc}_{K_{\omega^{p}}^{\prime} / K^{\prime}}+p^{n-1}(p-1) D_{\omega}.
\end{equation}
 The conductor discriminant
formula also leads to
\begin{equation}\label{e:disck'omegak}
\mathrm{Disc}_{ K_{\omega}^{\prime} /K}=  
\mathrm{Disc}_{ K_{\omega^{p}}^{\prime} / K}+
p^{n-1}(p-1) \sum_{\chi \in \widehat{\Gamma /\Phi}} D_{\tilde{\omega} \chi},
\end{equation}
since if $\tilde\omega$ is also of order $p^n$, then $K_{\tilde{\omega}}:=L^{\ker \tilde{\omega}}$ and $K'$ are disjoint over $K$, so 
$$\mathrm{Disc}_{ K_{\omega}^{\prime} /K}=  
\sum_{i=0}^{p^{n}-1} \sum_{\chi \in \widehat{\Gamma / \Phi}} D_{\tilde{\omega}^{i} \cdot \chi}=
\sum_{j=0}^{p^{n}-2} \sum_{\chi \in \widehat{\Gamma / \Phi}} D_{\tilde{\omega}^{pj} \cdot \chi}+
p^{n-1}(p-1) \sum_{\chi \in \widehat{\Gamma /\Phi}} D_{\tilde{\omega} \chi};$$
if $\tilde\omega$ is of order $p^{n+1}$, then $K'_\omega/K$ is a cyclic extension, 
each $\chi=\tilde\omega^{p^n j}$, $j=0,...,p-1$, so
$$\mathrm{Disc}_{ K_{\omega}^{\prime} /K}=\sum_{i=0}^{p^{n+1}-1} D_{\tilde{\omega}^{i}}=\sum_{i=0}^{p^n-1} D_{\tilde{\omega}^{pi}}+p^n(p-1) D_{\tilde\omega}. $$

By the transitivity of discriminants, for the trace map $\Nm_{{K}^{\prime} / {K}}$, one has
\begin{equation}\label{e:omegak'k}
\mathrm{Disc}_{ K_{\omega}^{\prime} /K}=\Nm_{{K}^{\prime} / {K}}(\mathrm{Disc}_{{K}_{\omega}^{\prime} / {K}^{\prime}})+p^{n} \cdot \mathrm{Disc}_{{K}^{\prime} / {K}}, 
\end{equation}
\begin{equation}\label{e:omegapk'k}
\mathrm{Disc}_{{K}_{\omega^p}^{\prime} / {K}}=\Nm_{{K}^{\prime} / {K}}(\mathrm{Disc}_{{K}_{\omega^p}^{\prime} / {K}^{\prime}})+p^{n-1} \cdot \mathrm{Disc}_{{K}^{\prime} / {K}}.
\end{equation}

Then \eqref{e:disck'omegak}, \eqref{e:disck'omegak'}, \eqref{e:omegak'k}, and \eqref{e:omegapk'k} together lead to the global version of \eqref{e:ddiscd}:
\begin{equation}\label{e:nk'k}
\Nm_{{K}^{\prime} / {K}}(D_{\omega})=-\mathrm{Disc}_{{K}^{\prime} / {K}}
+\sum_{\chi \in \widehat{\Gamma /\Phi}} D_{\tilde{\omega} \cdot \chi}.
\end{equation}

If $v$ is a good ordinary place, then in view of \eqref{e:ddiscd} and the equality $\alpha_{v'}=\alpha_{v}^{\log _{q_{v}} q_{v'}}$, 
$$\prod_{v'\mid v} \alpha_{v'}=\alpha_v^{-\ord_v \mathrm{Disc}_{K'/K}}\cdot \prod_{\chi\in\widehat{\Gamma/\Phi}} \alpha_v^{\ord_v D_{\tilde\omega\chi}} .$$
Thus, for verifyng \eqref{e:alphad}, it remains to show that for each non-split multiplicative place $v$,
$$\mathsf A_v=\mathsf B_v,$$
where 
$$\mathsf A_v:=\prod_{v'\mid v, \; v'\mid D_\omega}\lambda_{v'}^{\ord_v D_\omega -1},$$
and
$$\mathsf B_v:=
\begin{cases}(-1)^{\ord_v \mathrm{Disc}_{K'/K}-1}\cdot (-1)^p\cdot \prod_{\chi\in\widehat{\Gamma/\Phi}, \;v\mid D_{\tilde\omega\chi}} (-1)^{\ord_v D_{\tilde\omega\chi}-1},  & \text{ if}\; v\mid \mathrm{Disc}_{K'/ K};\\
\prod_{\chi\in\widehat{\Gamma/\Phi}, \;v\mid D_{\tilde\omega\chi}} (-1)^{\ord_v D_{\tilde\omega\chi}-1}, & \text{ otherwise}.
\end{cases}
$$

We first treat the case where $\omega$ is unramified at one (and hence at every) place sitting over $v$.
In this case, $\mathsf A_v=1$. If $K'/K$ is unramified at $v$, then $\mathsf B_v=1$; while if 
$K'/K$ is ramified at $v$, then there is a unique $\tilde\omega$ unramified at $v$, so by \eqref{e:ddiscd},
$\mathsf B_v=(-1)^{1-p+(p-1)}=1$, too.
Thus, for the rest of the proof, we assume that $\ord_{v'}D_\omega>0$.

In general, $\lambda_{v'}=-1$, unless $p=2$ and $K'/K$ is inert at $v$, in that case, there is a unique $v'$ sitting over $v$ with $\lambda_{v'}=1$, so $\mathsf A_v=1$. By \eqref{e:ddiscd}, $\mathsf B_v=1$ as well.

If $p>2$ and $K'/K$ is inert at $v$, then by \eqref{e:ddiscd}, $\mathsf A_v/\mathsf B_v=(-1)^{-1+p}=1$.

If $v$ splits over $K'/K$, then $\ord_{v'}D_\omega=\ord_v D_{\tilde\omega}$, so \eqref{e:ddiscd} says $\mathsf A_v=\mathsf B_v$.

If $K'/K$ is ramified at $v$, then again, by \eqref{e:ddiscd}, $\mathsf A_v/\mathsf B_v=(-1)^{-1+1-p+p}=1$. 
 Therefore, \eqref{e:alphad} holds.

Write $\xi$ for $\mathscr{L}_{A / L / K}$. 
All the above product formulae together with \eqref{e:intL} imply that for all $\omega \in \hat{\Phi}$, 
$\omega(\dag_{A / {L} / {K}^{\prime}} )\not=0$, we have
$\omega (\mathscr{L}_{A / {L} / {K}^{\prime}}-\varepsilon \cdot \alpha^{-1}_{\mathrm{Disc }_{K^{\prime} /K}} \cdot \xi_{\Phi}^{\Gamma})=0$. Hence
$$
\dag_{A / L / K^{\prime}}(\mathscr{L}_{A / {L} / {K}^{\prime}}-\varepsilon \cdot \alpha^{-1}_{\mathrm{Disc }_{K^{\prime} / K}} \cdot \xi_{\Phi}^{\Gamma})=0 .
$$
But since $\dag_{A /  L /K^{\prime}} \neq 0$, it follows that
$
\mathscr{L}_{A / {L} / {K}^{\prime}}=\varepsilon \cdot \alpha^{-1}_{\mathrm{Disc }_{K^{\prime} / K}} \cdot \xi_{\Phi}^{\Gamma}$, while $\varepsilon, \alpha^{-1}_{\mathrm{Disc }_{K^{\prime} / K}}\in \Lambda_\Phi^*$.

\end{proof}

\begin{corollary}\label{c:restriction}
If the Iwasawa main conjecture holds for $A/L/K$, then it also holds for $A/L/K'$.
\end{corollary}
\subsection{Iwasawa main conjecture}\label{su:imc} At first, we consider the simple but non-trivial $L=K$
case.
 
\subsubsection{The $L=K$ case}\label{ss:l=k} Basically, this is due to BSD conjecture.
 
 \begin{theorem}\label{t:l=k}
If $L=K$, then the Iwasawa main conjecture \eqref{e:imc} holds.
\end{theorem}

\begin{proof} Basically, the theorem follows from the results on the conjecture of Birch and Swinnerton-Dyer (see \cite{Tat66a, mil75, kt03}). If $\mathscr{L}_{A / L}=0$, then $L_{A}(\omega_{0}, 1)=0$, it follows that $X_{L}$ is non-torsion over $\Lambda=\mathbb{Z}_{p}$, and hence $\mathrm{CH}_{\Lambda}(X_{L})=0$. If $\mathscr{L}_{A / L} \neq 0$, then BSD holds. Hence $X_{L}$ is torsion and equals $\Sha_{p^{\infty}}(A / L)^\vee=\Sha_{p^{\infty}}(A / L)$, the $p$-primary part of the Tate-Shafarevich group of $A / L$. By \eqref{e:intL} and \cite[(3),(9),(46)]{tan95},
$$
\mathscr{L}_{A / L}=\star \cdot |\Sha_{p^{\infty}}(A / L) |, \;\text {for some } \star \in \mathbb{Z}_{p}^{*} .
$$
 \end{proof}
 
\subsubsection{The constant curve case}\label{ss:ccc}
\begin{theorem}\label{t:constant}If $A / K$ is an ordinary constant elliptic curve, then the Iwasawa main conjecture
\eqref{e:imc} holds.
\end{theorem}

\begin{proof}If $L=K$, this is Theorem \ref{t:l=k}. Suppose $L\not=K$. By \cite[Theorem 4.7]{lltt14a}, we have $\mathrm{CH}_{\Lambda}(X_{L})=(\mathscr{L}_{A, L})$. Then apply Lemma \ref{l:las}.
\end{proof}

\subsubsection{The case of constant field extension}\label{ss:cfc} 
Next we consider the case $L=K_{\infty}^{(p)}$, or equivalently, $S=\emptyset$. Assume that $A$ has semistable reduction at every place of $K$. Let $Z$ be the support of the conductor $N$. For each character $\omega: \Gamma \longrightarrow \mu_{p^{\infty}}$ define
$$
L_{Z}(A, \omega, s)=\prod_{v \in Z}(1-\lambda_{v} \omega([v]) q_{v}^{-s}) \cdot L_{A}(\omega, s).
$$

\begin{theorem}\label{t:constantfield} If $A / K$ has semistable reduction at all places of $K$ and $L / K$ is the constant $\mathbb{Z}_{p}$-extension, then the Iwasawa main conjecture \eqref{e:imc} holds.
\end{theorem}
\begin{proof}
We may assume that $A/K$ is not a constant curve.
By \cite[Theorem 1.1, Proposition 2.2.4, Theorem 3.1.5]{lltt16}, there is a generator $c_{A / L}$ of $\mathrm{CH}_{\Lambda}(X_{L})$ such that the element
\begin{equation}\label{e:fc}
f_{A / L}:=q^{-(\deg(Z)+\kappa-1+\frac{\deg(\Delta)}{12})} 
\cdot c_{A / L}
\end{equation}
satisfies that for each character $\omega$,
\begin{equation}\label{e:omegaf}
\omega(f_{A / L})=L_{Z}(A, \omega^{-1}, 1).
\end{equation}
Note that in the formula of \cite[Proposition 2.2.4]{lltt16}, $l=0$ and $g(v)=0$ for all $v\in Z$, since $A(L)[p^\infty]$ 
is finite ($A/K$ is non-constant) and the reduction at $v$ is multiplicative.

Because $S=\emptyset$, we have $\dag_{A / L}=1$ and $D_{\omega}=0$, for all $\omega$, hence $\tau_\omega=\omega([a]_{D_\omega}^{-1})$. Here by the abuse of notation, we let $a$ denote the divisor determined by the differential idele $a$, also, since $D_\omega=0$, we write $[a]\in\Gamma$ for the image of 
$[a]_{D_\omega}$ under the homomorphism $W_{D_\omega}\longrightarrow \Gamma$. 
Since $q^{\deg(Z)}=\prod_{v \in Z} q_{v}$, by comparing the interpolation formula \eqref{e:intL} with \eqref{e:omegaf}, we find that

$$
\begin{array}{rcl}
\mathscr{L}_{A / L}^\sharp & = & [a] \cdot \prod_{v \in Z} \frac{q_{v}}{q_{v}-\lambda_{v}[v]_{L / K}} \cdot q^{\kappa-1+\frac{\deg(\Delta)}{12}} \cdot f_{A / L} \\
& = & [a] \cdot \prod_{v \in Z}(q_{v}-\lambda_{v}[v]_{L / K})^{-1} \cdot c_{A / L}.
\end{array}
$$
Since each $q_{v}-\lambda_{v}[v]_{L / K}=(q_{v} \pm 1)+\lambda_{v}([v]_{L / K}-1)\in\Lambda^*$, and so is
$[a]$, in view of \eqref{e:fel} we see that $\mathscr L_{A/L}\in\Lambda$ and
$$
\mathrm{CH}_{\Lambda}(X_{L})=(\mathscr{L}_{A / L}^\sharp)=(\mathscr{L}_{A / L}).
$$
\end{proof}

\section{Valuations associated to characters}\label{s:app}
Let $L^{\prime} / K$ be a $\mathbb{Z}_{p}^{e}$-subextension of $L / K$ and let
$\omega \in \hat{\Gamma}^{\prime} \subset \hat{\Gamma}$. Then
$$\mathsf{v}_{\omega,L}=\mathsf{v}_{\omega,L'}\circ p^L_{L'}.$$
So, the specialization formula enable us to compare the equalities
\begin{equation}\label{e:vv}
\mathsf{v}_{\omega, L}(\mathrm{CH}_{\Lambda}(X_{L}))=\mathsf{v}_{\omega, L}(\mathscr{L}_{A / L})
\end{equation}
and
\begin{equation}\label{e:v'v'}
\mathsf{v}_{\omega, L^{\prime}}(\mathrm{CH}_{\Lambda'}(X_{L^{\prime}}))=\mathsf{v}_{\omega, L^{\prime}}(\mathscr{L}_{A / L^{\prime}}).
\end{equation}

\begin{lemma}\label{l:red} Suppose $\omega\in\hat\Gamma'$ and 
$p^L_{L'}(\dag_{A/L})\not=0$. Then the following holds:
\begin{enumerate}
\item $\eqref{e:v'v'} \Longrightarrow \eqref{e:vv}$.
\item $\omega(\vartheta_{L/L'}) = 0 \Longrightarrow \eqref{e:vv}$.
\item If $\omega(\vartheta_{L/L'}) \neq 0$, then $ \eqref{e:vv} \Longrightarrow \eqref{e:v'v'}$.
\end{enumerate}
\end{lemma}
\begin{proof}Note that $\omega(\varrho_{L / L^{\prime}} )\not=0$.
Since $p^L_{L'}(\dag_{A/L})\not=0$, the the specialization formulae imply:

$$
\mathsf v_{\omega,L}(\varrho_{L / L^{\prime}} )+ \mathsf v_{\omega,L}(\mathscr{L}_{A / L})
=\mathsf v_{\omega,L}(\vartheta_{L / L^{\prime}})+ \mathsf v_{\omega,L'}(\mathscr{L}_{A / L^{\prime}}),
$$
$$\mathsf v_{\omega,L}(\varrho_{L / L^{\prime}})+ \mathsf v_{\omega,L}(\mathrm{CH}_{\Lambda_\Gamma}(X_L))=\mathsf v_{\omega,L}(\vartheta_{L / L^{\prime}})+ \mathsf v_{\omega,L'}(\mathrm{CH}_{\Lambda_{\Gamma'}}(X_{L'})).
$$
\end{proof}

\subsection{The product formula}\label{su:prod}
For a given $\omega\in\hat\Gamma$, put 
$\Phi^{\prime}:=\ker(\omega)$, $\Phi^{\prime \prime}:=\ker (\omega^{p})$, and denote $K^{\prime}:={L}^{\Phi^{\prime}}$, $K^{\prime \prime}:={L}^{\Phi^{\prime \prime}}$. For each $f\in\Q_p\Lambda$, write $f^{L/K}_{L/K'}$ for $f^\Gamma_{\Phi'}$, and $f^{L/K}_{L/K''}$ for $f^\Gamma_{\Phi''}$.

\begin{lemma}\label{l:prod} Suppose $f\in\Q_p\Lambda$ is non-zero and $\omega$ is a character in $\hat\Gamma$ of order $p^n$. Then
\begin{equation}\label{e:prod}
\begin{array}{rcl}
\mathsf{v}_{\omega, L}(f)&=&\frac{1}{p^{n}-p^{n-1}} \cdot \mathsf{v}_{p}(\prod_{i \in (\Z / p^{n} \Z)^{*}} p_{K}^{L}(f_{\omega^{i}})).\\
{}&=& \frac{1}{p^{n}-p^{n-1}} \cdot \mathsf{v}_{p}\circ p^L_K(f_{L/K'}^{L/K}/ f_{L/K''}^{L/K})\\
{}&=& \frac{1}{p^{n}-p^{n-1}} \cdot \mathsf{v}_{p}\circ p^L_{K''}(f_{L/K'}^{L/K}/ f_{L/K''}^{L/K}).
\end{array}
\end{equation}
Note that if $n=0$, we take $p^{n-1}=0$ in the above formulae.
\end{lemma}
\begin{proof}
For each $g\in\Gal(\bar\Q_p/\Q_p)$, define $\tensor[^g]\omega{}$ to be the character such that $\tensor[^g]\omega{}(\gamma)=\tensor[^g]{(\omega(\gamma))}{}$. 
Then $\tensor[^g]\omega{}=\omega^i$, for some $i \in  (\Z/p^n\Z)^*$, and vice versa.
Since $\mathsf v_{\omega,L}(f)=\mathsf v_{\tensor[^g]\omega{},L}(f)$,

\begin{equation*}\label{e:val}
\mathsf v_{\omega,L}(f) = \frac{1}{p^n-p^{n-1}}\cdot  \sum_{i\in (\Z/p^n\Z)^*} \mathsf v_{\omega^i,L}(f).
\end{equation*}
Also, $\omega(f)=p_{K}^{L}(f_{\omega})$, so
$
\mathsf{v}_{\omega, L}(f)=\mathsf{v}_{p}\left(p_{K}^{L}\left(f_{\omega}\right)\right)
$.
Hence
\begin{equation*}\label{e:partial}
\mathsf{v}_{\omega, L}(f)=\frac{1}{p^{n}-p^{n-1}} \cdot \mathsf{v}_{p}\circ p^L_K (\prod_{i \in (\Z / p^{n} \Z)^{*}} f_{\omega^{i}}),
\end{equation*}
while
 $$\prod_{i \in(\Z / p^{n} \Z)^{*}} f_{\omega^{i}}=f_{L/K'}^{L/K}/ f_{L/K''}^{L/K}.$$
 The last equality in \eqref{e:prod} is due to the fact that $p^L_{K''}(f_{L/K'}^{L/K}/ f_{L/K''}^{L/K})$ is already a constant.
 \end{proof}

\subsection{Valuations associated to characters}\label{su:val}
 In this subsection, we prove Proposition \ref{p:val} and Proposition \ref{p:int}. 
 


Recall that in Monsky's topology, if $Y_1\subsetneq\hat\Gamma$ and $Y_2\subsetneq\hat\Gamma$ are two proper closed subsets, then (see \cite[Theorem 1.8]{monsky})
\begin{equation}\label{e:union}
Y_1\cup Y_2\subsetneq\hat\Gamma.
\end{equation}

If $\eta\in\Lambda$ is non-zero, then its zero set
$$\diamondsuit_\eta:=\{\omega\in\hat\Gamma\;\mid\; \omega(\eta)=0\}$$
is a proper closed subset \cite[Theorem 2.6]{monsky}.


\subsubsection{The proof of Proposition \ref{p:val}}\label{ss:pfpval}

If $\omega=\omega_0$,
the trivial character, then $\omega$ is the same as $p^L_K$, so the proposition is a consequence of Lemma \ref{l:red}(1) and Theorem \ref{t:l=k}.

Suppose $\omega\not=\omega_0$ and is of order $p^n$. We first treat the $K_\infty^{(p)}\subset L$ case,  under such condition $\dag_{A/L}$ is a natural number.
Let $\Phi'$, $\Phi''$, $K'$ and $K''$ be as in \S\ref{su:prod}.
Choose an isomorphism $\xymatrix{\mu_{p^{n}} \ar[r]^-\iota_-\sim & \Z_{p} / p^{n} \Z_{p}}$ as well as a surjective homomorphism $\tilde{\omega}: \Gamma \longrightarrow \Z_{p}$, such that $\iota \circ \omega=q \circ \tilde{\omega}$, where $q: \Z_{p} \longrightarrow \Z_p / p^{n} \Z_{p}$ is the natural map. 
Then $\Upsilon:=\ker(\tilde{\omega}) \subset \Phi'$ is of $\Z_{p}$-rank $d-1$ and $L^{\prime}:=L^{\Upsilon}$ is a $\Z_p$-extension of $K$ containing $K'$. Note that if $K^{\prime} \subset K_{\infty}^{(p)}$, then the above $\tilde\omega$ can be chosen to have $L'=K_{\infty}^{(p)}$. Since the Iwasawa conjecture holds for $A/L'$, we can apply Lemma \ref{l:red}(1) to complete the proof.

Consider the $K'\nsubseteq K_\infty^{(p)}$ case. By Lemma \ref{l:red} again, we replace $L$ by $L'K_\infty^{(p)}$. 
Since $L' \cap K_{\infty}^{(p)} \subseteq$ $K^{\prime \prime}$, in $L/K''$, the two $\Z_p$-extensions $L' / K^{\prime \prime}$ and $K_{\infty}^{\prime \prime}{ }^{(p)} / K^{\prime \prime}$ are disjoint. Thus, when applying Lemma \ref{l:prod}, since $f^{L/K}_{L/K'}\in \Q_p \Lambda_{\Phi'}$, we can write
$$p^L_{K''^{(p)}_\infty}(f^{L/K}_{L/K'})=p^L_{K'^{(p)}_\infty}(f^{L/K}_{L/K'}),$$
by identifying
$\Lambda_{\Gal(K_{\infty}^{\prime (p)} / K^{\prime})}$ with $\Lambda_{\Gal(K_{\infty}^{\prime \prime (p)} / K^{\prime \prime})}$ via the restriction of Galois actions 
$$\Gal(K_{\infty}^{\prime (p)} / K^{\prime})\longrightarrow \Gal(K_{\infty}^{\prime \prime (p)} / K^{\prime \prime}).$$

By Theorem \ref{t:constantfield}, the Iwasawa main conjecture holds for $A$ over $K'^{(p)}_\infty/K'$, so by 
the specialization formulae and Proposition \ref{p:restriction}, we can write 
$$p^L_{K'^{(p)}_\infty}((\mathscr L_{A/L})^{L/K}_{L/K'})=p^L_{K'^{(p)}_\infty}(\mathrm{CH}(X_L)^{L/K}_{L/K'}),$$
and similarly,
$$p^L_{K''^{(p)}_\infty}((\mathscr L_{A/L})^{L/K}_{L/K''})=p^L_{K''^{(p)}_\infty}(\mathrm{CH}(X_L)^{L/K}_{L/K''}).$$
Then apply Lemma \ref{l:prod}, noting that $p^L_{K''}=p^{K''^{(p)}_\infty}_{K''}\circ p^L_{K''^{(p)}_\infty}$.
This proves the $K^{(p)}_\infty\subset L$ case with $\mathsf Z=\emptyset$.

In general, put $\tilde L:=LK_\infty^{(p)}$, 
and take 
$$\mathsf Z:=\{\omega\in\hat\Gamma\;\mid\;\omega(\vartheta_{\tilde L/L})=0\}.$$ 
In this case, $\dag_{A/\tilde L}$ is a natural number, so by the Lemma \ref{l:red}(1)(3), if 
$\omega\not\in\mathsf Z$, then the desired equality \eqref{e:vv} holds. 
Let $S_{1,m}$ (resp. $S_{1,n}$) denote the subset of $S_1$ consisting of split-multiplicative places
(resp. non-split multiplicative places $v$ with $\F_{q_v^2}\subset L_v$).
By Definition \ref{d:vartheta}, the set
$$\mathsf Z=\bigcup_{v\in S_{1,m}}\diamondsuit_{1-\sigma_v}\cup \bigcup_{v\in S_{1,n}}\diamondsuit_{1+\sigma_v},$$
so it is a proper closed Monsky subset of $\hat\Gamma$.

\subsubsection{The proof of Proposition \ref{p:int}}\label{ss:int}
\begin{proof}
Recall that if $f\in\Lambda$ is not divisible by $p$, then
$$
\Omega_f:=\{\omega \in \hat{\Gamma} \mid \mathsf{v}_{\omega, L}(f) \geq 1\}
$$
is contained in a proper closed subset \cite[Theorem 2.3]{monsky}.

First consider the
$\mathscr{L}_{A / L}\not=0$ case and denote
$\mathscr{L}_{A / L}=p^{m} \cdot \xi_0$, where $\xi_0 \in \Lambda$, not divisible by $p$. 
Let $\mu$ be the $\mu$-invariant of $X_L$. 
If $\mu> m$, then
by Proposition \ref{p:val}, for $\omega\not\in \mathsf Z$, 
$$\mathsf{v}_{\omega, L}(\xi_0)= \mathsf{v}_{\omega, L}(p^{-m}\cdot \mathrm{CH}_{\Lambda}(X_L))\geq \mu-m\geq 1.$$
This means $\Omega_{\xi_0} \cup \mathsf Z=\hat{\Gamma}$, a contradiction to \eqref{e:union}.
Thus, $m\geq\mu$. A similar argument shows $\mu\geq m$.
Note that an $f\in\Lambda$ is zero if and only if $\diamondsuit_f=\hat\Gamma$, so

$$\begin{array}{rcl}
\mathrm{CH}_{\Lambda}(X_L)=0 
 &\iff& \diamondsuit_{\mathrm{CH}_{\Lambda}(X_L)}
 =\hat\Gamma\\
{} &\iff& \diamondsuit_{\xi_0}\cup \mathsf Z =\hat\Gamma\\
{} &\iff& \mathscr L_{A/L}=0.
\end{array}
$$
The proposition holds for the $\mathscr L_{A/L}=0$ case.

\end{proof}

\section{More applications}\label{s:gen} 
The proof of Proposition \ref{p:e} is divided into three steps, the first two are purely Iwasawa-Algebra theoretic. 

\subsection{The first step}\label{su:conjsp}
Consider the setting that there is given a $\Z_p^d$-extension $L/K$ of a field $K$ and there are elements $\xi_{L}\in \Q_p\Lambda$ and $\eta_{L}\in \Lambda$, where $\Gamma:=\Gal(L/K)$, $\Lambda:=\Lambda_{\Gamma}$.
In this setting, the main conjecture means
\begin{equation}\label{e:mc}
\xi_L\in\Lambda,\;\text{and in } \Lambda\; \text{the ideals }\;(\xi_L)=(\eta_{L}).
\end{equation}

Suppose $e<d$ and $L'/K$ is an intermediate $\Z_p^e$-extension of $L/K$.
As before, denote
$\Gamma':=\Gal(L'/K)$, $\Lambda'=\Lambda_{\Gamma'}$. 
It is obvious that if the main conjecture holds, then
\begin{equation}\label{e:conjsp}
p^L_{L'}(\xi_L)\in\Lambda',\;\text{and in } \Lambda' \text{ the ideals } (p^L_{L'}(\xi_L))=p^L_{L'}(\eta_L).
\end{equation}
The following proposition says that for $e\geq 2$, \eqref{e:conjsp} collectively can be used to deduce the main conjecture \eqref{e:mc}. Its proof is given in \S\ref{su:pfgen}.

Recall that $\mho(e,\Gamma)$ denotes the set of all intermediate $\Z_p^e$-extensions $L'/K$.
It is identified with the Grassmanninan $\mathrm{Gr}(d-i,d)(\Z_p)$, and hence endowed with
the Zariski topology.

\begin{myproposition}\label{p:gen} Suppose $d>e\geq 2$.
If there is a non-empty Zariski open subset $O\subset \mho(e,\Gamma)$ such that \eqref{e:conjsp} holds
for every $L'/K\in O$, then \eqref{e:mc} holds for $L/K$.
\end{myproposition}

\subsection{Preliminary results}\label{su:prim}
\subsubsection{Partial flag varieties}\label{su:flag} For $0< d_1<d_2<\cdots<d_{n-1}<d$, let $\mathsf{F} (d_1,d_2,..,d_{n-1},d)$ denote the partial flag variety over $\Q_p$ representing the functor that assigns to
each algebraic extension $k$ of $\Q_p$ the set of all sequences of $k$-subspaces
$$V_1\subset V_2\subset V_{n-1}\subset k^d,\;\dim_k V_i=d_i.$$
We identify $\mho(i,\Gamma)$ with $\mathsf F(i,d)(\Q_p)$. 
A partial flag varieties can be viewed as the a homogeneous space $\mathrm{GL}(d,\Q_p)/P$, where $P$ is certain parabolic subgroup,
so that its Zariski topology coincides with the quotient (Zariski) topology. Thus, if $E':=\{d'_1,...,d'_{m-1}\}$ is an ordered subset of $E:=\{d_1,...,d_{n-1}\}$, the projection
$$\pi_{\scriptscriptstyle{E},{E'}}:\mathsf F(d_1,...,d_{n-1},d)\longrightarrow \mathrm F(d'_1,...,d'_{m-1},d)$$
is both open and closed. 

\begin{lemma}\label{l:ij}The following holds.
\begin{enumerate}
\item For a given a non-trivial closed subgroup $\Phi\subset \Gamma$ and a natural number $i$, 
$$C_\Phi:=\{L'/K\in \mho(i,\Gamma)\;\mid\; \Gal(L/L')\supset \Phi \}$$
is a Zariski closed subset of $\mho(i,\Gamma)$ .
\item Suppose $0<i<j\leq d$. For a non-empty Zariski open subset $O_i\subset \mho(i,\Gamma)$,
$$O_j:=\{ L'/K\in\mho(j,\Gamma)\;\mid\; L'\supset L'',\;\text{for some } L''/K\in O_i\}$$
is a non-empty Zariski open subset of $\mho(j,\Gamma)$. 
\item Suppose $0<i<j\leq d$. For a non-empty Zariski open subset $U_j\subset \mho(j,\Gamma)$, 
$$U_i:=\{ L'/K\in\mho(i,\Gamma)\;\mid\; L'\subset L'',\;\text{for some } L''/K\in U_j\}$$
is a non-empty Zariski open subset of $\mho(i,\Gamma)$.

\end{enumerate}

\end{lemma}

\begin{proof}Let $d-j$ denote the $\Q_p$-dimension of $\Q_p\Phi\subset \Q_p\Gamma$. Put $\tilde\Phi=\Q_p \Phi\cap \Gamma$, $L''=L^{\tilde\Phi}$ and denote
$x:=L''/K$ as an element of $\mho(j,\Gamma)(\Q_p)=\mathsf F(j,d)(\Q_p)$. If $i>j$, then $C_\Phi=\emptyset$;, if $i=j$, then $C_\Phi=\{x\}$.
For $j>i$, take $E=\{i,j\}$, $E'=\{i\}$, and $E''=\{j\}$. Then $C_\Phi$ is nothing but $\pi_{\scriptscriptstyle{E},{E'}}((\pi_{\scriptscriptstyle{E},{E''}}^{-1}(x)))$, so is closed in $\mho(i,\Gamma)$. This proves (1).

For the rest, we may assume that $j<d$. Take $E$, $E'$ and $E''$ as above. Then use
$O_j=\pi_{\scriptscriptstyle{E},{E''}}(\pi_{\scriptscriptstyle{E},{E'}}^{-1}(O_{i}))$  
and $U_i=\pi_{\scriptscriptstyle{E},{E'}}(\pi_{\scriptscriptstyle{E},{E''}}^{-1}(U_{j}))$
to complete the proof.
\end{proof}



Suppose $0< i<j\leq d$. Let $L''/K\in\mho(j,\Gamma)$ and denote $\Gal(L''/K)=\Gamma''$.
Then the assignment $L'/K\mapsto L'/K$ gives rise to an injective morphism 
\begin{equation}\label{e:l''gamma}
\digamma_{i,L''}:\mho(i,\Gamma'') \longrightarrow \mho(i,\Gamma)
\end{equation}
whose image is the Zariski closed subset
$$\mho(i,L''/K,\Gamma):=\{L'/K\in\mho(i,\Gamma)\;\mid\;L'\subset L''\}\subset \mho(i,\Gamma).$$

\begin{lemma}\label{l:digamma}Suppose $O$ is a Zariski open set of $\mho(i,\Gamma)$ and $L''/K\in\mho(j,\Gamma)$, $i\leq j$. Then those $L'/K\in \mho(i,\Gamma'')$, which is also in $O$, form a Zariski open subset of $\mho(i,\Gamma'')$.

\end{lemma}

\begin{proof}
Via \eqref{e:l''gamma}, the set in question can be expressed as $\digamma^{-1}_{i,L''}(O\cap \mho(i,L''/K,\Gamma))$.
\end{proof}
\subsubsection{Valuations}\label{su:morev} In our setting, the counter parts of Proposition \ref{p:val} and Proposition \ref{p:int}
hold as follows. 
\begin{lemma}\label{l:countp}Suppose $1\leq e <d$ and there is a non-empty open set $O\subset\mho(e,\Gamma)$ such that \eqref{e:conjsp} holds for every $L'/K\in O$. Then  
 \begin{equation}\label{e:lx}
\mathsf v_{\omega,L}(\xi_L)=\mathsf v_{\omega,L}(\eta_L),\:\text{for all } \omega\in\hat\Gamma.
\end{equation}
\end{lemma}
\begin{proof}
Put $U_e:=O$ and let $U_1\subset \mho(1,\Gamma)$ be the corresponding open subset described in Lemma \ref{l:ij}(3). We first consider those $\omega$ belonging to $\hat\Gamma'\subset\hat\Gamma$ 
for some $L'/K\in U_1$.
Suppose $L'/K$ is a sub-extension of some $L''/K\in U_e$. Since $p^L_{L''}(\xi_L)=p^{L}_{L''}(\eta_L)$,
by applying $\omega$ to both sides and then taking the valuation, we obtain \eqref{e:lx}. 

In general, $\omega\in \widehat{\Gal(L_1/K)}$ for some $L_1/K\in\mho(1,\Gamma)$. 
Choose $L'/K\in U_1$ and put $L''=L'L_1'$. The Zariski closed subset $\mho(1,L''/K,\Gamma)\subset \mho(1,\Gamma)$ defined in \ref{su:flag} is isomorphic $\mho(1,\Gamma'')=\mathbb P^1$.
Since $U_1\cap \mho(1,L''/K,\Gamma)$, being containing $L'/K$, is a non-empty Zariski open subset of $\mho(1,L''/K,\Gamma)$, its complement is a finite subset. On the other hand, the given $\omega$ belongs to the dual group
of $\Gal(M/K)$, for infinitely many $M/K$ in $ \mho(1,L''/K,\Gamma)$, so we can arrange to have 
$M/K\in U_1$ and hence deduce \eqref{e:lx}.
 \end{proof}

\begin{corollary}\label{c:countp}
 $\xi_L$ has the same $\mu$-invariant as that of $\eta_L$ and in particular, $\xi_L\in\Lambda_\Gamma$.
\end{corollary}
\begin{proof}By the argument in the proof of Proposition \ref{p:int}.
\end{proof}

\subsubsection{Plain extensions}\label{su:normalb}
Suppose a non-zero $f\in \Q_p\Lambda_\Gamma$ has $\mu$-invariant $m$ and $f=p^m\cdot f_0$. We say that a
$\Z_p$-extension $L'/K$  in $\mho(1,\Gamma)$ is $f$-plain if $p^L_{L'}(f_0)\in\Lambda'$ is not divisible by $p$.

\begin{lemma}\label{l:normalb}
For a given non-zero $f\in \Q_p\Lambda_\Gamma$, there exists a non-empty Zariski open subset $V_f$ of $ \mho(1,\Gamma)$ such that every $L'/K$ belonging to $V_f$ is $f$-plain.
\end{lemma}
\begin{proof} 
By \cite[Lemma 1.5,Theorem 2.3]{monsky}, in Monsky's topology the set
$$\Omega:=\{\omega \in \hat{\Gamma} \mid \mathsf{v}_{\omega, L}(f_0) \geq 1\}$$ 
is contained in a proper closed subset $\mathsf T\subset \hat\Gamma$ and $\mathsf T$ is a finite union of
$\Z_p$-flats $\Upsilon_i$.

Let $Z_f(1,\Gamma)\subset \mho(1,\Gamma)$ be the subset consisting of intermediate $\Z_p$-extension $L'/K$
of $L/K$ such that in $\Lambda_{\Gamma'}$, $p^L_{L'}(f_0)$ is divisible by $p$ . 
If $L'/K$ belongs to $Z_f(1,\Gamma)$, then every character in $\hat\Gamma'$ belongs to
$\Omega$, so $\hat\Gamma'\subset \mathsf T$. By \cite[Definition 1.4, Theorem 1.8]{monsky}, the group
$\hat\Gamma'$ is actually a $\Z_p$-flat, and hence an irreducible closed subset in the Noetherian topology of Monsky, so $\hat\Gamma'\subset \Upsilon_i$, for some $i$. There exist $n_i\in \mathbb N$, $\tau_{i,j}\in\Gamma$, $j=1,..., n_i$, extendable to a $\Z_p$-basis of $\Gamma$, and  $\zeta_{i,j}\in\mu_\infty$, such that
$$\Upsilon_i=\{ \omega\in\hat\Gamma\; \mid\;  \omega(\tau_{i,j})=\zeta_{i,j}, j=1,...,n_i\}.$$
Since $\hat\Gamma'$ is a $p$-divisible group, we must have $\zeta_{i,j}=1$ for all $j$, so $\Upsilon_i=\hat\Gamma_i''$,
where $\Gamma_i''$ is the quotient of $\Gamma$ by the group $\Phi_i$ topologically generated by $\tau_{i,1},...,.\tau_{i,n_i}$. Let $Z_i\subset \mho(1,\Gamma)$ denote the proper Zariski closed subset consisting of $L'/K$ such that $\Gal(L/L')\supset \Phi_i$. The above discussion implies that $Z_f(1,\Gamma)$ is contained in the finite union of $Z_i$,
so its complement $Z^c_f(1,\Gamma)$ contains a non-empty Zariski open subset, which we take to be $V_f$.
\end{proof}

\subsection{The proof of Proposition \ref{p:gen}}\label{su:pfgen}
We prove by the induction on $e$.
Denote $O_e:=O$, the open set in the proposition and for $j>e$, let $O_{j}$ be the open subset described in Lemma \ref{l:ij}(2). In the induction process, take 
$O_{e+1}$ to be the required open subset of $\mho(e+1,\Gamma)$.  If $L''/K\in O_{e+1}$, then $O_e\cap\mho(e,L''/K,\Gamma)$ is a non-empty
Zariski open subset of $\mho(e,L''/K,\Gamma)$ which can be identified with $\mho(e,\Gamma'')$ (Lemma \ref{l:digamma}). Thus, it remains to prove the proposition for $L=L''$ and $d=e+1$.

By Corollary \ref{c:countp}, we can write $\xi_L=p^\mu\cdot \xi_0$ and $\eta_L=p^\mu\cdot \eta_0$, where
$\xi_0$ and $\eta_0$ are in $\Lambda$, not divisible by $p$.
Take $U_e=O$ and let $U_1\subset \mho(1,\Gamma)$ be the open subset in Lemma \ref{l:ij}(3). 
Choose a $\Z_p$-extension $L'/K$ in $V_{\xi_L}(1,\Gamma)\cap V_{\eta_L}(1,\Gamma)\cap U_1$, which is
a non-empty Zariski open subset of $\mho(1,\Gamma)$, to have $L'/K$ both $\xi_L$-plain and $\eta_L$-plain (Lemma \ref{l:normalb}).
 Lift a topological generator $\bar\sigma_1$ of $\Gamma'$ to $\sigma_1\in\Gamma$, and choose a $\Z_p$-basis $\sigma_2,...,\sigma_d$ of $\Gal(L/L')$. Then $\sigma_1,...,\sigma_d$ form a $\Z_p$-basis of $\Gamma$. Put $t_i=\sigma_i-1$ and write 
 $$\xi_0(t_1,...,t_d)=\sum_{j=0}^\infty \xi_0^{(j)}\cdot t_1^j, \quad \eta_0(t_1,...,t_d)=\sum_{i=0}^\infty \eta_0^{(i)}\cdot t_1^i,$$ 
 where $\xi_0^{(j)}$ and $\eta_0^{(i)}$ are formal power series in $t_2,...,t_d$ over $\Z_p$. 
 Since $\xi_0(t_1,0,...,0)$ and $\eta_0(t_1,0,...,0)$ can be respectively identified with $p^L_{L'}(\xi_0)$ and
 $p^L_{L'}(\eta_0)$, they are both not divisible by $p$, so
there are non-negative integers $r_1$ and $r_2$ such that the constant terms of the formal
power series $\xi_0^{(r_1)}$ and $ \eta_0^{(r_2)}$ are not divisible by $p$,
while $p$ divides the constant terms of $\xi_0^{(j)}$ and $\eta_0^{(i)}$, for $j<r_1$, $i<r_2$. 
The equality \eqref{e:lx} actually implies that $r_1=r_2$ (denoted as $r$ later), since if the order of $\omega\in\Gamma'$ is large enough, then
$$\mathsf v_{\omega,L}(\xi_0)=\mathsf v_{\omega,L}(\sum_{j=0}^\infty \xi_0^{(j)}(0,...,0)\cdot\omega(t_1)^j)=
r_1\cdot \mathsf v_{\omega,L}(\omega(\sigma_1)-1),$$
and also 
$$\mathsf v_{\omega,L}(\eta_0)=r_2\cdot \mathsf v_{\omega,L}(\omega(\sigma_1)-1).$$

By Weierstrass Preparation Theorem  \cite[VII, \S 8, Proposition 6]{bou}, there exist
$$\xi'_0=t_1^r+\sum_{j=0}^{r-1} a_j\cdot t_1^j=u_1\cdot \xi_0,\; a_j\in \Z_p[[t_2,...,t_d]],\; 
u_1\in\Lambda_\Gamma^*,$$
$$\eta'_0=t_1^r+\sum_{j=0}^{r-1} b_j\cdot t_1^j=u_2\cdot \eta_0,\; b_j\in \Z_p[[t_2,...,t_d]],\; 
u_2\in\Lambda_\Gamma^*.$$
The subset 
$\mho(L'/K,e,\Gamma):=\{L''/K\in \mho(e,\Gamma)\;\mid\; L''\supset L'\}\subset \mho(e,\Gamma)$
is Zariski closed, since it can be identified with the fibre at $L'/K$ of the composition 
$$\mho(e,\Gamma)=\mathsf F(e,d)\longrightarrow \mathsf F(1,e,d)\longrightarrow \mathsf F(1,d)=\mho(1,\Gamma),$$
so its non-empty open subset $V:=O\cap  \mho(L'/K,e,\Gamma)$ is of infinite cardinality. 
The hypothesis of the proposition ensures that for each $L''/K\in V$,
$$p^L_{L''}(\xi_0)=u_{L,L''}\cdot p^L_{L''}(\eta_0), \; \text{for some}\; u_{L,L''}\in \Lambda_{\Gamma''}^*,$$
where $\Gamma''=\Gal(L''/K)$. Hence, 
\begin{equation}\label{e:pll''}
p^L_{L''}(\xi_0')=p^L_{L''}(u_1^{-1}u_2)\cdot u_{L,L''}\cdot p^L_{L''}(\eta_0').
\end{equation}

Choose a topological generator $\tau_d$ of $\Gal(L/L'')$, extend it to a $\Z_p$-basis $\tau_2,...,\tau_d$
of $\Gal(L/L')$, and put $s_i=\tau_i-1$. Then $\Z_p[[t_2,...,t_d]]=\Z_p[[s_2,...,s_d]]$ and $\sigma_1,\tau_2,...,\tau_d$ form a $\Z_p$-basis of $\Gamma$, while under
$p^L_{L''}$, $s_i\mapsto s_i$, for $2\leq i\leq d-1$, $s_d\mapsto 0$, and $t_1\mapsto t_1$. In particular,
$p^L_{L''}(\xi_0')$ and $p^L_{L''}(\eta_0')$ are distinguished polynomials in $\Lambda_{\Gamma''}$ viewed
as the ring of formal power series in $t_1$ over $\Z_p[[s_2,...,s_{d-1}]]$.
Thus, in view of \eqref{e:pll''}, we deduce that $p^L_{L''}(\xi_0'-\eta_0')=0$, which means $\xi_0'-\eta_0'$ is divisible by $s_d$, since $s_d$ generates $\ker(p^L_{L''})$. 
The variable $s_d$ is associated to $L''$.
If $L''$ varies, the associated $s_d$ are relatively prime in $\Lambda_\Gamma$.
We must have $\xi_0'-\eta_0'=0$, since it is divisible by infinitely many relatively prime $s_d$.

\subsection{The second step}\label{su:2nd}  In the second step, we add $\xi_{L'}$, $\eta_{L'}$, and the specialization formulae for $L/L'$, into the setting.

\begin{myassumption}\label{assump} Assume that
 \begin{enumerate}
\item For each $L'/K\in\mho(i,\Gamma)$, $2\leq i< d$, there are also given elements $\xi_{L'}\in \Q_p\Lambda'$ and $\eta_{L'}\in \Lambda'$.

\item For each $2\leq i < d$, there exists a non-empty Zariski open subset $U(i)\subset \mho(i,\Gamma)$ such that 
if $L'/K\in U(i)$, then there are principal ideals $\mathsf R_{L/L'}\not=0$ and $\smallint_{L/L'}$ in $\Lambda'$ 
satisfying that in $\Q_p\Lambda'$, the principal fractional ideas 
\begin{equation}\label{e:spxi}
p^L_{L'}(\xi_L)\cdot \mathsf R_{L/L'}=\xi_{L'}\cdot {\smallint}_{L/L'},
\end{equation}
and 
\begin{equation}\label{e:speta}
p^L_{L'}(\eta_L)\cdot \mathsf R_{L/L'}=\eta_{L'}\cdot {\smallint}_{L/L'}.
\end{equation}

\end{enumerate}
\end{myassumption}
We call \eqref{e:spxi} and \eqref{e:speta} the specialization formulae. 



In this setting, by the main conjecture for $L'/K$, we mean 
\begin{equation}\label{e:mc'}
\xi_{L'}\in\Lambda',\;\text{and in } \Lambda'\; \text{the ideals }\;(\xi_{L'})=(\eta_{L'}).
\end{equation}

\begin{myproposition}\label{p:ge} Suppose Assumption holds.
The main conjecture \eqref{e:mc} holds for $L/K$, if for some $e$, $2\leq e<d$,
there is a non-empty Zariski open subset $O_e\subset\mho(e,\Gamma)$
such that for every $L'/K\in O_e$, the main conjecture \eqref{e:mc'} holds.
\end{myproposition}
\begin{proof} Let $U(e)$ be the non-empty Zariski open subset of $\mho(e,\Gamma)$ in Assumption
and put $O=O_e\cap U(e)$, which is non-empty.
Then \eqref{e:spxi}, \eqref{e:speta} and \eqref{e:mc'} together imply that \eqref{e:conjsp} holds for $L'/K\in O$.
Thus the Proposition follows from Proposition \ref{p:gen}.
\end{proof}
\begin{remark}\label{r:kl}Our results in \S\ref{su:conjsp} to \S\ref{su:2nd} are solely on the group ring
$\Lambda$, or equivalently the corresponding formal power series ring, they have nothing to do with 
the nature of the ground field $K$ or the arithmetic meaning of $\xi_{L}$ and $\eta_L$. 
\end{remark}

\subsection{The proof of Proposition \ref{p:e}}\label{su:specialcase} Basically, the proof is a direct consequence of
Proposition \ref{p:ge}. Take $\xi_{L'}=\mathscr L_{A/L'}$,
$\eta_{L'}$ a generator of $\mathrm{CH}_{\Lambda'}(X_{L'})$, and  put $\mathsf R_{L/L'}=\varrho_{L/L'}$, $\smallint_{L/L'}=\vartheta_{L/L'}$. In view of \eqref{e:varrho}, Proposition \ref{p:aspf}, and Proposition \ref{p:spl},
we let $U(i)$ to be the set
$$U(i,\Gamma):=\{L'/K\;\mid\; p^L_{L'}(\dag_{A/L})\cdot \vartheta_{L/L'}\not=0\}\subset \mho(i, \Gamma),$$ 
so that the specialization formulae \eqref{e:spxi} and \eqref{e:speta} hold for $L'/K\in U(i)$.
The following lemma says the Assumption is satisfied in our setting.
\begin{lemma}\label{l:ui} $U(i,\Gamma)$ is a non-empty Zariski open subset of $\mho(i,\Gamma)$.
\end{lemma}
\begin{proof} By Corollary \ref{c:0}, the complement $V$ of $U(i,\Gamma)$ consists of $L'/K \in\mho(i,\Gamma)$ such that $\Gal(L/L')$ contains the decomposition subgroup $\Gamma_v\subset \Gamma$, at some split multiplicative $v\in S$ (so that $\Gamma_v\not=0$, and hence $V$ is a proper Zariski closed subset of 
$\mho(i,\Gamma)$). By Lemma \ref{l:ij}, $V$ is a Zariski closed subset of $\mho(i,\Gamma)$.
\end{proof}
\appendix
\section{The omitted details}
We shall recover the details omitted in the proof of Lemma \ref{l:spvartheta} and \eqref{e:beth}. 
Assume that
$K\subset L'\subset L$ where $L'$ is a $\Z_p^e$ extension of $K$. 
Define
$$\Xi_{L/L'}:=\prod_v \Xi_{L/L',v},$$
where
$$\Xi_{L/L',v}:=
\begin{cases}
\lambda_v-[v]_{L/K}, & \text{if } v\in S_m, v\notin S_m';\\
(1-\alpha_{v}^{-1}[v]_{L^{\prime}})
(1-\alpha_{v}^{-1}[v]_{L^{\prime}}^{-1}), & \text{if } v\in S_o, v\notin S_o';\\
1, &\text{otherwise}.
\end{cases}
$$
\begin{lemma}\label{l:locsp}For every $v$,
\begin{equation}\label{e:question}
\Xi_{L/L',v}\cdot \dag_{A/L',v}\cdot\Lambda'=p^L_{L'}(\dag_{A/L,v})\cdot \vartheta_{L/L',v}.
\end{equation}
\end{lemma}
\begin{proof} We prove by checking.
\begin{enumerate}
\item[(A)] Suppose $v\not\in S$. Then $\Xi_{L/L',v}=1$. If $\Psi_v\not=0$, then $\Gamma_v=\Psi_v\simeq \Z_p$ and $\Gamma_v'=0$, so $\dag_{A/L,v}=1$, $\dag_{A/L',v}=m_v$, hence \eqref{e:question} reads $m_v\cdot\Lambda'=m_v\cdot\Lambda'$;
if $\Psi_v=0$, then $\Gamma_v=\Gamma'_v$, so $\Lambda'=\Lambda'$.
\item[(B)] Suppose $v$ is a good ordinary place in $S$. then $\dag_{A/L,v}=1$ and $\dag_{A/L',v}=1$. If $v\not\in S'$, then \eqref{e:question} reads $\Xi_{L/L',v}\cdot\Lambda'=\vartheta_{L/L',v}$; if $v\in S'$, then $\Lambda'=\Lambda'$.

\item[(C)] Suppose $v$ is a split multiplicative place in $S$. 
\begin{enumerate}
\item[(C1)] If $\Psi_v\simeq \Z_p^f$, $f\geq 2$, $\Gamma_v'=0$, then $\dag_{A/L,v}=1$, $\dag_{A/L',v}=1$,
$\Xi_{L/L',v}=0$ and $\vartheta_{L/L',v}=0$, so $0=0$.
\item[(C2)] If $\Psi_v\simeq\Z_p^f$, $f\geq 1$, $v\in S_1'$, then $\Xi_{L/L',v}=1$, $\dag_{A/L,v}=1$, $\dag_{A/L',v}=1-\sigma_v$, $\vartheta_{L/L',v}=(1-\sigma_v)\cdot\Lambda'$, so 
$(1-\sigma_v)\cdot\Lambda'=(1-\sigma_v)\cdot\Lambda'$.
\item[(C3)] If $\Psi_v\simeq\Z_p^f$, $f\geq 1$, $\Gamma_v'\simeq\Z_p$, $v\not\in S'$, then
$\Xi_{L/L',v}=1-[v]_{L'}$, $\dag_{A/L,v}=1$, $\dag_{A/L',v}=1$, $\vartheta_{L/L',v}=(1-[v]_{L'})\cdot\Lambda'$, so 
$(1-[v]_{L'})\cdot\Lambda'=(1-[v]_{L'})\cdot\Lambda'$.
\item[(C4)] If $\Psi_v\simeq\Z_p$, $\Gamma_v'=0$, then $\Xi_{L/L',v}=0$, $\dag_{A/L,v}=1-\sigma_v$, $\dag_{A/L',v}=m_v$, $\vartheta_{L/L',v}=\mathfrak w_v\cdot\Lambda'$,
so $0=0$.
\item[(C5)] If $\Gamma_v'\simeq\Z_p$ and $\Psi_v=0$, then $\Xi_{L/L',v}=1$, $\dag_{A/L,v}=1-\sigma_v$,
$\dag_{A/L',v}=1-\sigma_v$, $\vartheta_{L/L',v}=\Lambda'$, so $(1-\sigma_v)\cdot\Lambda'=(1-\sigma_v)\cdot\Lambda'$; if $\Gamma_v'\simeq\Z_p^g$, $g\geq 2$, then
$\Xi_{L/L',v}=1$, $\dag_{A/L,v}=1$,
$\dag_{A/L',v}=1$, $\vartheta_{L/L',v}=\Lambda'$, so $\Lambda'=\Lambda'$.
\end{enumerate}
\item[(D)] Suppose $v$ is a non-split multiplicative place in $S$.
\begin{enumerate}
\item[(D1)] If $\Gamma'_v=0$ and $\Psi_v\simeq \Z_p^f$, $f\geq 2$, then $\Xi_{L/L',v}=-2$, $\dag_{A/L,v}=1$, $\dag_{A/L',v}=m_v$, so $-2m_v\cdot\Lambda'=2m_v\cdot\Lambda'$.
\item[(D2)] If $\Gamma_v'=0$, $\Psi_v\simeq \Z_p$ and $\F_{Q_v^2}\nsubseteq L_v$, then $\Xi_{L/L',v}=-2$,
$\dag_{A/L, v}=1$, $\dag_{A/L',v}=m_v$, so $-2\cdot m_v\cdot\Lambda'=2m_v\cdot\Lambda'$.
\item[(D3)] If $\Gamma_v'=0$, $\Psi_v\simeq \Z_p$ and $\F_{Q_v^2}\subset L_v$, then $\Xi_{L/L',v}=-2$,
$p^L_{L'}(\dag_{A/L, v})=-2$, $\dag_{A/L',v}=m_v$, so $-2\cdot m_v\cdot\Lambda'=-2\cdot m_v\cdot\Lambda'$.
\item[(D4)] If $\Gamma'_v\simeq \Z_p$, $v\not\in S'$, then $\Psi_v\not=0$, $\Xi_{L/L',v}=-1-[v]_{L'}$, $\dag_{A/L, v}=1$, $\dag_{A/L', v}=1$, so $(-1-[v]_{L'})\cdot\Lambda'=(1+[v]_{L'})\cdot\Lambda'$.
\item[(D5)] If $v\in S_1'$, $\Psi_v\not=0$ and $\F_{Q_v^2}\subset L'_v$, then
$\Xi_{L/L',v}=1$, $\dag_{A/L, v}=1$, $\dag_{A/L', v}=-1-\sigma'_v$, so $(-1-\sigma'_v)\cdot\Lambda'=(1+\sigma'_v)\cdot\Lambda'$.
\item[(D6)] If $v\in S_1'$, $\Psi_v\not=0$ and $\F_{Q_v^2}\nsubseteq L'_v$, then
$\Xi_{L/L',v}=1$, $\dag_{A/L, v}=1$, $\dag_{A/L', v}=1$, so $\Lambda'=\Lambda'$.
\item[(D7)] If $v\in S_1'$, $\Psi_v=0$, then $\Xi_{L/L',v}=1$, $p^L_{L'}(\dag_{A/L, v})\cdot\Lambda'=\dag_{A/L', v}\cdot\Lambda'$ (since $\Gamma_v=\Gamma'_v$), so $\dag_{A/L', v}\cdot \Lambda'=\dag_{A/L', v}\cdot \Lambda'$.
\item[(D8)] If $\Gamma_v'\simeq \Z_p^g$, $g\geq 2$, then $\Xi_{L/L',v}=1$, $\dag_{A/L, v}=1$, $\dag_{A/L', v}=1$, so $\Lambda'=\Lambda'$.
\end{enumerate}
\end{enumerate}

\end{proof}

 Suppose $L''/K$ is a $\Z_p^f$-extension, with $L'\subset L''\subset L$. By the Lemma, 
$$
\begin{array}{rcl}
\Xi_{L/L',v}\cdot \dag_{A/L',v}\cdot\Lambda'&=&p^L_{L'}(\dag_{A/L,v})\cdot \vartheta_{L/L',v},\\
\Xi_{L''/L',v}\cdot \dag_{A/L',v}\cdot\Lambda'&=&p^{L''}_{L'}(\dag_{A/L'',v})\cdot \vartheta_{L''/L',v},\\
p^{L''}_{L'}(\Xi_{L/L'',v})\cdot p^{L''}_{L'}(\dag_{A/L'',v})\cdot\Lambda'
&=&p^{L}_{L'}(\dag_{A/L,v})\cdot p^{L''}_{L'}(\vartheta_{L/L'',v}).
\end{array}
$$
Multiplying $p^{L''}_{L'}(\dag_{A/L'',v})$ to the first formula, then comparing the resulting formula with the product of the second and the third formulae, and using the equality
$$ \Xi_{L/L',v}=\Xi_{L''/L',v}\cdot p^{L''}_{L'}(\Xi_{L/L'',v}),$$
we obtain
\begin{equation}\label{e:loc}
p^L_{L'}(\dag_{A/L,v})\cdot p^{L''}_{L'}(\dag_{A/L'',v})\cdot \vartheta_{L/L',v}=
p^L_{L'}(\dag_{A/L,v})\cdot p^{L''}_{L'}(\dag_{A/L'',v})\cdot \vartheta_{L''/L',v}\cdot p^{L''}_{L'}(\vartheta_{L/L'',v}).
\end{equation}

\begin{lemma}\label{l:locspvartheta} If $L'\subset L''\subset L$, then for every $v$,
$$\vartheta_{L/L',v}=\vartheta_{L''/L',v}\cdot p^{L''}_{L'}(\vartheta_{L/L'',v}).$$
\end{lemma}
\begin{proof} In view of \eqref{e:loc}, we just treat the $p^L_{L'}(\dag_{A/L,v})\cdot p^{L''}_{L'}(\dag_{A/L'',v})=0$
case.

Suppose $p^L_{L'}(\dag_{A/L,v})=0$. Then $\Gamma_v'=0$, $v\in S_1$, and $\vartheta_{L/L',v}=\mathfrak w_v\cdot\Lambda'$. If $\Gamma_v''=0$, then $\vartheta_{L''/L',v}=\Lambda'$ and $\vartheta_{L/L'',v}=\mathfrak w_v\cdot\Lambda''$; if $\Gamma_v''\not=0$, then $L_v=L''_v$, $v\in S_1''$, $\vartheta_{L/L'',v}=\Lambda''$, $\vartheta_{L''/L',v}=\mathfrak w_v\cdot\Lambda'$, so in both cases, 
$$\mathfrak w_v\cdot\Lambda'=\mathfrak w_v\cdot\Lambda'.$$

Suppose $p^L_{L'}(\dag_{A/L,v})\not=0$, but $p^{L''}_{L'}(\dag_{A/L'',v})=0$. Then $\Gamma_v'=0$, $v\in S_1''$ and $\Psi_v\simeq\Z_p^f$, $f\geq 2$, thus, $\vartheta_{L/L',v}=0$, $\vartheta_{L''/L'}=\mathfrak w_v\cdot\Lambda'$, $\vartheta_{L/L'',v}=(1-\sigma_v'')\cdot\Lambda''$, so we get $0=0$. 

\end{proof}
Lemma \ref{l:locspvartheta} implies its global version: 
$$\vartheta_{L/L'}=p^{L''}_{L'}(\vartheta_{L/L''})\cdot \vartheta_{L''/L'},$$
and hence the first formula in Lemma \ref{l:spvartheta} follows. Also, by
Lemma \ref{l:locsp},
$$p^L_{L'}(\dag_{A/L})\cdot \vartheta_{L/L'}=\Xi_{L/L'}\cdot \dag_{A/L'}=p^L_{L'}(\dag_{A/L})\cdot \pounds_{L/L'}.$$
Thus, if $p^L_{L'}(\dag_{A/L})\not=0$, then \eqref{e:beth} holds.

\end{document}